\documentclass{amsart}
\usepackage{style}
\usepackage[T1]{fontenc}
\title[The inverse limit topology and profinite descent on Picard groups in $\Sp_{K(n)}$]{The inverse limit topology and profinite descent on\protect\\ Picard groups in $K(n)$-local homotopy theory}
\author{Guchuan Li and Ningchuan Zhang}
\address{School  of Mathematical Sciences, Peking University, Beijing, China}
\hemail{liguchuan@math.pku.edu.cn}
\address{Department of Mathematics, Indiana University Bloomington, Bloomington, IN 47405, USA}
\hemail{nz7@iu.edu}

\newcommand{\vResAr}[1][]{\dar[bend right, color=blue,thick][swap]{#1}}
\newcommand{\vTrAr}[1][]{\uar[bend right, color=brown,thick][swap]{#1}}
\newcommand{\fin}{\mathrm{fin}}
\newcommand{\pro}{\mathrm{pro}}
\newcommand{\Smash}{\wedge}
\newcommand{\bigsmash}{\bigwedge}
\newcommand{\hsmash}{\hat{\Smash}}
\newcommand{\bighsmash}{\wh{\bigsmash}}
\allowdisplaybreaks

\begin{document}
	\begin{abstract}
		In this paper, we study profinite descent theory for Picard groups in $K(n)$-local homotopy theory through their inverse limit  topology. Building upon Burklund's result on the multiplicative structures of generalized Moore spectra, we prove that the module category over a $K(n)$-local commutative ring spectrum is equivalent to the limit of its base changes by a tower of generalized Moore spectra of type $n$. As a result, the $K(n)$-local Picard groups are endowed with a natural inverse limit topology. This topology allows us to identify the entire $E_1$ and $E_2$-pages of a descent spectral sequence for Picard spaces of $K(n)$-local profinite Galois extensions. 
		
		Our main examples are $K(n)$-local Picard groups of homotopy fixed points $E_n^{hG}$ of the Morava $E$-theory $E_n$ for all closed subgroups $G$ of the Morava stabilizer group $\mathbb{G}_n$. The $G=\mathbb{G}_n$ case has been studied by Heard and Mor. At height $1$, we compute Picard groups of $E_1^{hG}$ for all closed subgroups $G$ of $\mathbb{G}_1=\mathbb{Z}_p^\times$ at all primes as a Mackey functor.
	\end{abstract}
	\maketitle
	\tableofcontents
	
	One of the most important invariants of a symmetric monoidal category is the Picard group, which has wide applications in algebraic geometry and number theory.  For a symmetric monoidal $\infty$-category $\Ccal$, one can define its Picard $\infty$-groupoid $\pic(\Ccal)$, viewed as a \emph{space}. The Picard group is then the zeroth homotopy group of the Picard space. The idea of studying Picard groups in stable homotopy theory was initiated by Hopkins \cite{HMS_picard,Strickland_interpolation}. For the symmetric monoidal $\infty$-category of spectra $\Sp$, its Picard group is isomorphic to $\Z$, generated by $S^1$. 
	
	From the chromatic perspective of stable homotopy theory, we study $\Sp$ by the height filtration of formal group laws at each prime $p$. One way to extract the monochromatic information of spectra is to localize them at the height $n$ Morava $K$-theory $K(n)$. 
	The Picard group of the category of $K(1)$-local spectra $\Sp_{K(1)}$ is: \cite{HMS_picard,Strickland_interpolation}
	\begin{equation*}
		\Pic\left(\Sp_{K(1)}\right)=\begin{cases}
			\Zp\oplus\Z/(2p-2), & p>2;\\
			\Z_2\oplus\Z/4\oplus\Z/2, &p=2.
		\end{cases}
	\end{equation*}
	
	One observation is that the $K(1)$-local Picard group $\Pic\left(\Sp_{K(1)}\right)$ is a \emph{profinite} abelian group. This holds for $\Pic\left(\Sp_{K(n)}\right)$ in general for any $n$ by \cite[Proposition 14.3.(d)]{Hovey-Strickland_1999} (see also \cite{Goerss-Hopkins}). The main purpose of this paper is to lift the inverse limit  topology\footnote{Not necessarily profinite. Recall the limit of (discrete) sets $X=\lim_i X_i$ can be modeled as a subset of $\prod X_i$, which has a non-trivial product topology. The limit $X$ then inherits a subspace topology.  This is also the weakest topology on $X$ to make the projection maps $X\to X_i$ continuous with each $X_i$ being discrete. It is called profinite if the $X_i$'s are all finite sets. } on Picard groups of $K(n)$-local ring spectra to the level of Picard spaces. This topology allows us to identify the entire $E_1$ and $E_2$-pages of a descent spectral sequence for Picard spaces of $K(n)$-local profinite Galois extensions. 
	
	In \cite{Hovey-Strickland_1999}, Hovey--Strickland described the inverse limit  topology on $\Pic\left(\Sp_{K(n)}\right)$ as follows. They first constructed a  tower of generalized Moore spectra $\{M_j\}$ of type $n$ such that $X\simeq \lim_j  L_{K(n)}(X\wedge M_j)$ for any $K(n)$-local spectrum $X$.  Then they defined a basis $\{V_j\}$ for closed neighborhoods of the identity element $S^0_{K(n)}$ in the Picard group by setting  $V_j=\left\{X\in \Pic\left(\Sp_{K(n)}\right)\mid L_{K(n)}(X\wedge M_j)\simeq L_{K(n)}M_j\right\}$. Burklund's recent work \cite{Burkland_mult_Moore}   implies that the $M_j$'s in this tower can be chosen to have compatible $\Ebb_{j}$-algebra structures (\Cref{prop:genMoore_mult}). This refines the result of Davis--Lawson in \cite{Davis-Lawson_2014}  that the tower $\{M_j\}$ admits an $\einf$-ring structure as a \emph{pro}-spectrum.  For a ring (spectrum) $R$, its Picard group and space are defined to be the Picard group and space, respectively, of its module category $\Mod(R)$. Our first main result in this paper is the following Grothendieck existence theorem (\Cref{rem:Grothendieck_existence}) for $K(n)$-local ring spectra in \Cref{sec:prodisc_topology}.
	\begin{thm*}[{\Cref{thm:Kn_Pic_limit}}]
		Let $R$ be a $K(n)$-local $\einf$-ring spectrum. Then the limit of base change functors induces an equivalence of symmetric monoidal $\infty$-categories
		\begin{equation*}
			\Mod_{K(n)}(R)\xrightarrow[\ref{thm:R-mod_moore}]{\sim}\lim_j\Mod_{K(n)}(R\Smash M_j) \xleftarrow[\ref{lem:Rk-mod}]{\cong} \lim_j\Mod(R\Smash M_j).
		\end{equation*}
		This yields an equivalence of group-like $\einf$-spaces and an isomorphism of abelian groups:
		\begin{align*}
			\pic_{K(n)}(R)\simto \lim_j\pic(R\Smash M_j),\quad 
			\Pic_{K(n)}(R)\simto \lim_j\Pic(R\Smash M_j).
		\end{align*}
	\end{thm*}
	The isomorphism above endows $\Pic_{K(n)}(R)$ with a natural \emph{inverse limit topology}. We show in \Cref{prop:choice_Mk} that this topology is independent of the choice of the tower of generalized Moore spectra of type $n$  in \Cref{prop:genMoore_mult}. 
	
	When $R=S^0_{K(n)}$, our result gives an $\infty$-categorical explanation of the profinite topology on $\Pic\left(\Sp_{K(n)}\right)$ described in \cite{Hovey-Strickland_1999,Goerss-Hopkins}. It is well-established that the Picard group functor $\Pic\colon \CAlg(\Sp)\to \Ab$ for $\einf$-ring spectra  commutes with filtered colimits and \emph{finite} products. In \Cref{prop:Pic_Kn_profin_prod} and \Cref{prop:Pic_Kn_colim}, we show that the $K(n)$-local Picard group functor $\Pic_{K(n)}\colon \CAlg\left(\Sp_{K(n)}\right)\to \Pro(\Ab)$ commutes with $K(n)$-local filtered colimits and \emph{profinite} products (see \Cref{defn:profin_prod_Kn}). 	Our computations at height $1$ in \Cref{sec:Pic_Mack} give concrete examples when the \emph{discrete} $K(1)$-local Picard group functor $\Pic_{K(1)}\colon \CAlg\left(\Sp_{K(1)}\right)\to\Pro(\Ab)\xrightarrow{\lim} \Ab$ does not commute with filtered colimits (\Cref{rem:Pic_colim_p_odd} and \Cref{rem:Pic_colim_p2}). 
	
	Building on the descent theory for $K(n)$-local module categories in \cite{Mathew_Galois}, the structure of units of a ring spectrum in \cite[\S5]{MS_Picard}, and \Cref{thm:Kn_Pic_limit}, we prove our second main result in \Cref{subsec:desc_ss}.
	\begin{thm*}[{\Cref{thm:pic_ss}}]
		Let $A\to B$ be a descendable $K(n)$-local profinite $G$-Galois extension of $\einf$-ring spectra. Suppose the inverse system $\{\pi_t(B\Smash M_j)\}_{j\ge1}$ satisfies the Mittag-Leffler condition for any $t\in \Z$. Then there are continuous homotopy fixed point spectral sequences:
		\begin{align*}
			\!^{\HFP} E_1^{s,t}=\map_c(G^{\times s}, \pi_t(B)),\qquad \!^{\HFP} E_2^{s,t}=H_c^s(G; \pi_t(B)) &\Longrightarrow \pi_{t-s}(A);\\
			\!^{\pic} E_1^{s,t}=\map_c(G^{\times s}, \pi_t\pic_{K(n)}(B)),\qquad \!^{\pic} E_2^{s,t}=H_c^s(G; \pi_t\pic_{K(n)}(B)) 
			&\Longrightarrow \pi_{t-s}\pic_{K(n)}(A), \qquad t-s\ge0. 
		\end{align*}
		
		In particular, the second spectral sequence abuts to $\Pic_{K(n)}\left(A\right)$ when $t = s$. 
		The differentials in both spectral sequences are of the form $d_r^{s,t}\colon E_r^{s,t} \to  E_r^{s+r,t+r-1}$. When either $t-s>0$ and $s>0$, or $2\le r\le t-1$, we have $^{\pic}d_r^{s,t}=\!^{\HFP}d_r^{s,t-1}$.
	\end{thm*}
	Having set up the general theory, we mainly focus on $K(n)$-local Picard groups of homotopy fixed points $E_n^{hG}$ of the height $n$ Morava $E$-theory, where $G$ is a \emph{closed} subgroup of the Morava stabilizer group $\Gbb_n$. These fixed points were explicitly constructed by Devinatz--Hopkins in \cite{fixedpt}. To apply \Cref{thm:pic_ss}, we first show in \Cref{cor:En_hG_descendable} that $E_n^{hG}\to E_n$ admits descent using the following fact. From descent theory, we know if $A\to B \to C$ is descendable, then so is $A\to B$. If in addition $A\to B \to C$ is a composition of profinite Galois extensions, we prove $B\to C$ is also descendable in \Cref{prop:BhH_descent}. 
	
	The descent spectral sequence for $\pic_{K(n)}\left(E_n^{hG}\right)$ has the $E_1$-page:
	\begin{equation*}
		E_1^{s,t}=\pi_t\pic_{K(n)}\left[L_{K(n)}\left(\bigsmash\nolimits_{E_n^{hG}}^{s+1}E_n\right)\right]\cong \pi_t\pic_{K(n)}\map_c\left(G^{\times s},E_n\right).
	\end{equation*}
	When $t\ge 1$,  we have $\pi_t\pic_{K(n)}\map_c\left(G^{\times s},E_n\right)\cong  \map_c\left(G^{\times s},\pi_t\pic_{K(n)}(E_n)\right)$, which was essentially computed in \cite{fixedpt}. When $t=0$, the $E_1^{s,0}$-term is the $K(n)$-local Picard group of a profinite product of $E_n$. This is the main technical difficulty in identifying the entire  $E_1$ and  $E_2$-pages of the profinite descent spectral sequence for $K(n)$-local Picard groups.  From the aforementioned property of the $K(n)$-local Picard groups in \Cref{prop:Pic_Kn_profin_prod}, we obtain:
	\begin{equation*}
		E_1^{s,0}= \Pic_{K(n)}\map_c\left(G^{\times s},E_n\right)\cong \map_c\left(G^{\times s},\Pic_{K(n)}(E_n)\right).
	\end{equation*}
	The case  when $G=\Gbb_n$ (note that $E_n^{h\Gbb_n}\simeq S^0_{K(n)}$) has been recently studied by Heard and Mor.  Heard computes the $E_1$ and $E_2$-pages in the range $t>0$ or $s=t=0$ in \cite[Example 6.18]{Heard_2021Sp_kn-local}. Mor uses the pro\'etale (condensed) method to identify the entire $E_2$-page in \cite{mor2023picard}.
	
	In \cite[\S3.3]{BBGHPS_exotic_h2_p2} and \cite[\S1.3]{CZ_exotic_Picard}, the authors described a descent filtration (see \Cref{const:desc_fil}) on $\Pic_{K(n)}\left(E_n^{h\Gbb_n}\right)$ using the homotopy fixed point spectral sequence. We next study this filtration using the descent spectral sequence \Cref{thm:pic_ss} for $\Pic_{K(n)}\left(E_n^{hG}\right)$  in \Cref{subsec:desc_fil}.
	\begin{prop*}[{\ref{prop:desc_fil_comparison}}]
		The descent filtration in \Cref{const:desc_fil} on $\Pic_{K(n)}\left(E_n^{hG}\right)$ agrees with the filtration associated to the $t-s=0$ stem of the descent spectral sequence for $\Pic_{K(n)}\left(E_n^{hG}\right)$ in \Cref{thm:pic_ss}. 
	\end{prop*}
	This comparison allows us to prove the following algebraicity result for $\Pic_{K(n)}\left(E_n^{hG}\right)$.
	\begin{thm*}[{\Cref{thm:descent_Pic_EnhG}}]
		Fix a prime $p>2$.  Let $G\le \Gbb_n$ be a closed subgroup such that  $G\cap \zpx$ is cyclic of order $m$, where $\zpx\le \Zpx=Z(\Gbb_n)$ is the torsion subgroup of the center $\Zpx$ of $\Gbb_n$. Denote the $p$-adic cohomological dimension of $G$ by $\cd_p G$. 
		\begin{enumerate}
			\item When $2m+1>\cd_p G$, the exotic Picard group $\kappa\left(E_n^{hG}\right)$ vanishes and the descent filtration on $K(n)$-local Picard group $\Pic_{K(n)}\left(E_n^{hG}\right)$ is:
			\begin{equation*}
				\begin{tikzcd}
					0\rar& H_c^{1}(G;\pi_{0}(E_n)^\times)\rar & \Pic_{K(n)}\left(E_n^{hG}\right)\rar["\phi_0"] &\Pic_{K(n)}(E_n)=\Z/2\rar & 0,
				\end{tikzcd}	 	
			\end{equation*}
			where $\phi_0$ is induced by base change along the ring extension $E_n^{hG}\to E_n$.
			\item When $2m+1=\cd_p G$, the map $\phi_1\colon\Pic_{K(n)}^0\left(E_n^{hG}\right)\to \Pic_{K(n)}^{alg,0}\left(E_n^{hG}\right)=H_c^{1}(G;\pi_{0}(E_n)^\times)$ is surjective. 
		\end{enumerate}
	\end{thm*}
	When $G=\Gbb_n$ and $(p-1)\nmid n$,  we have $\cd_p(\Gbb_n)=n^2$ and $m=|\zpx\cap \Gbb_n|=|\zpx|=p-1$. In this case, \Cref{thm:descent_Pic_EnhG} states that $\kappa_n=0$ if $2p-1=2m+1>\cd_p \Gbb_n=n^2$ and $(p-1)\nmid n$. This recovers \cite[Proposition 7.5]{HMS_picard}.
	
	As an application of \Cref{thm:pic_ss}, we compute $K(1)$-local Picard groups of  homotopy fixed points $E_1^{hG}$ for all closed subgroups $G\le \Zpx$ at all primes. Picard groups of Galois extensions of ring spectra are not only a collection of abelian groups -- they are connected by restriction and transfer maps. These data assemble into a Picard \emph{Mackey functor}. Beaudry--Bobkova--Hill--Stojanoska computed $\Pic_{K(2)}\left(E_2\right)$ at the prime $2$ as a $C_4$-Mackey functor in \cite{BBHS_PicE2hC4}.  When $G\le \Zpx$ is a pro-cyclic closed subgroup, the homotopy fixed points $E_1^{hG}$ are all $K(1)$-local algebraic $K$-theory spectra of some finite fields $\Fq$ (\Cref{prop:algK_finfld_p} and \Cref{rem:algK_finfld_2}). 
	\begin{thm*}[{\Cref{thm:Mackey_Pic_k1_p=odd}} in {\Cref{subsec:mackey_odd}}]
		Let $p>2$ be an odd prime number,  $k\ge 1$ and $m\mid (p-1)$ be some positive integers. The Picard groups of $E_1^{hG}$ for all closed subgroups $G\le \Zpx$ are listed below:
		\begin{equation*}
			\Pic_{K(1)}\left(E_1^{hG}\right)=\left\{\begin{array}{lll}
				\Z/(2m)\oplus \Zp, & G=\Z/m\times (1+p^k\Zp);& \textup{\cite{HMS_picard,Strickland_interpolation}}\text{ for $G=\Zpx$}\\
				\Z/(2m), & G=\Z/m. & \textup{\cite{MS_Picard, Baker-Richter_invertible_modules}}
			\end{array}\right.
		\end{equation*}
		Let $m_1\mid m_2\mid (p-1)$ and $1\le k_2\le k_1$ be some  positive integers. The formulas for restriction (left downwards) and transfer (right upwards) maps between subgroups of finite indices are:
		\begin{equation*}
			\begin{tikzcd}[column sep=tiny]
				\Pic_{K(1)}\left(E_1^{h\left(\Z/m_2\times  (1+p^{k_2}\Zp)\right)}\right)\rar[symbol=\cong]&\Zp\oplus \Z/2m_2\vResAr[p^{k_1-k_2}\oplus 1]&~&~&\Pic_{K(1)}\left(E_1^{h\Z/m_2}\right)\rar[symbol=\cong]& \Z/2m_2\vResAr[1]\\
				\Pic_{K(1)}\left(E_1^{h\left(\Z/m_1\times (1+p^{k_1}\Zp)\right)}\right)\rar[symbol=\cong]& \Zp\oplus \Z/2m_1,\vTrAr[1\oplus (m_2/m_1)]&~&~&
				\Pic_{K(1)}\left(E_1^{h\Z/m_1}\right)\rar[symbol=\cong]&  \Z/2m_1\vTrAr[m_2/m_1].
			\end{tikzcd}				
		\end{equation*}
	\end{thm*} 
	The $p=2$ case is more complicated. We manage to resolve the extension problems using the restriction maps between Picard groups. 
	\begin{thm*}[See full statement in {\Cref{thm:Mackey_Pic_k1_p=2}} in {\Cref{subsec:mackey_2}}]
		Let $p=2$ and   $\langle \alpha\rangle$ be the closed subgroup of $\Z_2^\times$ generated by $\alpha$.  The Picard groups of $E_1^{hG}$ for all closed subgroups $G\le \Z_2^\times$ are listed below:
		\begin{equation*}
			\Pic_{K(1)}\left(E_1^{hG}\right)=\left\{\begin{array}{lll}
				\Z_2\oplus\Z/4\oplus \Z/2, &G=\Z_2^\times;& \textup{\cite{HMS_picard,Strickland_interpolation}}\\
				\Z_2\oplus\Z/2,& G=\langle 5\rangle \text{ or } \langle 3\rangle;&\\
				\Z_2\oplus \Z/8\oplus \Z/2, &G=\{\pm 1\}\times \langle 2^k+1\rangle, k\ge 3;&\\
				\Z_2\oplus \Z/2\oplus \Z/2, &G=\langle 2^k+1\rangle \text{ or }\langle2^k-1\rangle, k\ge 3;&\\
				\Z/8, & G=\{\pm 1\};&\textup{\cite{MS_Picard,Gepner-Lawson_Brauer}}\\
				\Z/2, & G=\{1\}.& \textup{\cite{Baker-Richter_invertible_modules}}
			\end{array}\right.
		\end{equation*}
		There are seven cases of the restriction and transfer maps between $\Pic_{K(1)}\left(E_1^{hG_1}\right)$ and $\Pic_{K(1)}\left(E_1^{hG_2}\right)$ when $G_1$ is an index $2$ subgroup of $G_2$, according to the closed subgroup lattice of $\Gbb_1=\Z_2^\times$ in \Cref{lem:Z2x_lattice}. 
		\begin{equation*}
			\begin{tikzcd}[every arrow/.append style={dash,thick},column sep =small]
				\Z_2^\times \ar[dr]\rar \ar[ddr] &\{\pm 1\}\times \langle 9\rangle \rar\ar[dr]\ar[ddr] &  \{\pm 1\}\times \langle 17\rangle \rar\ar[dr]\ar[ddr] &\cdots \rar\ar[dr]\ar[ddr]& \{\pm 1\}\times \langle 2^{k}+1\rangle  \rar\ar[dr]\ar[ddr]& \{\pm 1\}\times \langle 2^{k+1}+1\rangle  \rar\ar[dr]\ar[ddr]& \cdots\rar &\{\pm 1\}\ar[ddr] & \\
				& \langle3\rangle \ar[dr] & \langle7\rangle \ar[dr]& \cdots\ar[dr]& \langle 2^{k-1}-1\rangle \ar[dr]& \langle 2^{k}-1\rangle \ar[dr]&\cdots\ar[drr] && \\
				&  \langle5\rangle \rar &  \langle9\rangle\rar & \cdots\rar &  \langle 2^{k-1}+1\rangle \rar & \langle 2^{k}+1\rangle \rar & \cdots\ar[rr] &&\{1\}.
			\end{tikzcd}
		\end{equation*}
	\end{thm*}
	When $p>2$, we observe in \Cref{rem:Pic_E1hG_topgen} that for an open subgroup $G=\Z/m\times (1+p^k\Zp)\le \Zpx$, the Picard group $\Pic_{K(1)}\left(E_1^{hG}\right)$ is topologically generated by $X=\Sigma^{1/p^{k-1}}E_1^{hG}$. This is the unique $K(1)$-local invertible spectrum over $E_1^{hG}$ whose $p^{k-1}$-st smash power over $E_1^{hG}$ is $\Sigma E_1^{hG}$. However at the prime $2$, the Picard groups $\Pic_{K(1)}\left(E_1^{hG}\right)$ are not topologically cyclic by \Cref{thm:Mackey_Pic_k1_p=2} for any open subgroup $G\le \Z_2^\times$. Moreover, the element $\Sigma E_1^{hG}$ is not divisible by $p=2$ in $\Pic_{K(1)}\left(E_1^{hG}\right)$, as observed in \Cref{rem:Pic_E1hG_topgen_p=2}. 
	\subsection*{Notation and conventions}
	\begin{itemize}
		\item $\Pic$ denotes the Picard \emph{group} of a  monoidal category and $\pic$ denotes the Picard \emph{space} ($\infty$-groupoid) of a  monoidal $\infty$-category. See more details in \Cref{subsec:Pic_gp_sp}. Please note our notation for Picard spaces is slightly different from that in \cite{MS_Picard} (see \Cref{rem:Pic_pic}).
		\item When $X$ and $Y$ are two $K(n)$-local spectra, we denote their \emph{$K(n)$-local} smash product by $X\hsmash Y:=L_{K(n)}(X\wedge Y)$. 
		\item All categories and $\infty$-categories in this paper are assumed to be presentable. 
		\item $\Cat_\infty$ is the $\infty$-category of $\infty$-categories and $\mathrm{Pr}^L$ is the $\infty$-category of presentable $\infty$-categories with small colimit preserving functors. 
		\item By modules over an $\Ebb_k$-algebra, we always mean \emph{left} modules. 
		\item 	
		Let $\Ccal$ be an $\infty$-category. The pro-objects in $\Ccal$ are defined as follows:
		\begin{enumerate}
			\item \cite[Definition 5.3.5.1]{HTT} The category of {ind-objects} $\Ind(\Ccal)$ in $\Ccal$ is the full subcategory of presheaves $\Pcal(\Ccal)=\Fun(\Ccal^\op, \sSet)$ on $\Ccal$ spanned by those maps classifying right fibrations $\wt{\Ccal}\to \Ccal$ where $\wt{\Ccal}$ is filtered.  
			\item 
			The category of {pro-objects} $\Pro(\Ccal)$ in $\Ccal$ is $(\Ind(\Ccal^\op))^\op$.  In particular, given two pro-objects $\underline{x}=(x_i)$ and $\underline{y}=(y_j)$ in $\Ccal$, we set
			\begin{equation*}
				\map_{\Pro(\Ccal)}(\underline{x}, \underline{y}):=\lim_j\colim_i\map_\Ccal(x_i,y_j).
			\end{equation*}
			This mapping space is abbreviated as $\map_c(x,y)$, where $c$ stands for "continuous". 
		\end{enumerate}
		\item We use the Lewis diagram to denote the restriction and the transfer maps in a Mackey functor. Let $\underline{M}$ be a $G$-Mackey functor and $H_1\le H_2$ be subgroups of $G$. Then we write
		\begin{equation*}
			\begin{tikzcd}
				\underline{M}(G/H_2)\vResAr[\res]\\ \underline{M}(G/H_1)\vTrAr[\tr],
			\end{tikzcd}
		\end{equation*} 
		where the left downwards arrow is the restriction map and right upwards arrow is the transfer map. 
	\end{itemize}
	\subsection*{Acknowledgements}
	We want to thank Ko Aoki, Venkata Sai Bavisetty, Agn\`es Beaudry, Mark Behrens, Irina Bobkova, Elden Elmanto, David Gepner, Paul Goerss, Ishan Levy, Svetlana Makarova, Mona Merling, Maximilien P\'eroux,  Charles Rezk, Tomer Schlank, Vesna Stojanoska, and Paul VanKoughnett for many helpful and supportive discussions and comments. We would also like to thank William Balderrama for suggesting a significant simplification of the proof of a key \Cref{lem:Rk-mod}; Tobias Barthel and Drew Heard for carefully explaining the original difficulty in identifying the entire $E_2$-page of the profinite descent sequence for $K(n)$-local Picard groups in \cite{Heard_2021Sp_kn-local}; Anish Chedalavada for pointing out that \Cref{thm:R-mod_moore} is a Grothendieck existence theorem type statement; Daniel Davis for  useful comments to improve the first arXiv version; Michael Mandell for careful comments on the final draft; Akhil Mathew for insightful discussions regarding \Cref{thm:R-mod_moore}; Itamar Mor for kindly explaining his recent related work \cite{mor2023picard}; and the anonymous referee for many helpful comments and suggestions to improve the expositions. 
	
	Some of the work and revision were done when the two authors were at the Max Planck Institute for Mathematics (at separate times). We would like to thank the MPIM for its hospitality and support.  G. Li would also like to thank the Hausdorff Research Institute for Mathematics for its hospitality and support during the trimester program "Spectral Methods in Algebra, Geometry, and Topology" in Fall 2022, funded by the Deutsche Forschungsgemeinschaft under Germany's Excellence Strategy -- EXC-2047/1 -- 390685813. N. Zhang was partially supported by the NSF Grant DMS-2348963/2304719 and the AMS-Simons Travel Grant.
	\section{Recollections on Picard groups and spaces of $K(n)$-local ring spectra}
	In this section, we review basic definitions and properties of Picard groups and spaces of $K(n)$-local ring spectra. These well-established materials can be found in sources such as  \cite{Mathew_Galois,Lurie_HA,MS_Picard}. We then conclude with a  summary of properties of algebra objects in $\infty$-monoidal categories in \Cref{subsec:Alg_O} for later use, following the treatments in \cite{Lurie_HA, Lurie_DAG_III}. While this material is largely expository, readers familiar with these topics may safely skip this section and return to it as needed.
	\subsection{Picard groups and spaces of monoidal $\infty$-categories}\label{subsec:Pic_gp_sp}
	\begin{defn}
		Let $(\Ccal,\otimes, 1_\Ccal)$ be a presentable  monoidal $1$-category. An object $X\in \Ccal$ is called \textbf{invertible} if there is an object $Y\in \Ccal$ such that $X\otimes Y\cong 1_\Ccal$. The \textbf{Picard group} of $\Ccal$ is defined to be:
		\begin{equation*}
			\Pic(\Ccal)=\{X\in \Ccal\mid X\text{ is invertible}\}/\cong,
		\end{equation*}
		where the group multiplication is given by $\otimes$ and the unit is the equivalence class of $1_\Ccal$.
	\end{defn}
	\begin{exmps} Let $R$ be a commutative ring. The module category of $R$ has a natural symmetric monoidal structure $(\Mod_R, \otimes_R, R)$. Write $\Pic(R)$ for $\Pic(\Mod_R)$.
		\begin{enumerate}
			\item When $R=\Fbb$ is a field, $\Pic(\Fbb)=\{\Fbb\}$ is the trivial group. 
			\item When $R$ is a Dedekind domain, $\Pic(R)\cong \mathrm{Cl}(R)$ is isomorphic to the ideal class group of (the fractional field of) $R$. This group is trivial iff $R$ is a unique factorization domain.
			\item \cite[Theorem 3.5]{Fausk_Pic} For the category $\Ch_R$ of chain complexes over $R$ and its derived category $\Dcal_R$, their Picard groups are isomorphic to $\Pic(R)\oplus C(\spec R,\Z)$ where $C(\spec R,\Z)$ is the additive group of $\Z$-valued continuous functions on $\spec R$. 
		\end{enumerate}
	\end{exmps}
	In stable homotopy theory, we study the category $\Sp$ of spectra. This is an example of a symmetric monoidal $\infty$-category. 
	\begin{defn}\label{defn:Pic}
		Let $(\Ccal,\otimes, 1_\Ccal)$ be a presentable $\Ebb_k$-monoidal $\infty$-category for some for $1\le k\le\infty$. The \textbf{Picard space} $\pic(\Ccal)$ of $\Ccal$ is the $\infty$-groupoid (space) of invertible objects in $\Ccal$. The Picard group of $\Ccal$ is the Picard group of its homotopy category $h\Ccal$. Equivalently, we have $\Pic(\Ccal):=\Pic(h\Ccal)\cong \pi_0\pic(\Ccal)$.
	\end{defn}
	One can check the following fact directly from the definition.
	\begin{lem}\label{lem:pic_Ccal_fib}
		When $\Ccal$ is $\Ebb_k$-monoidal, its Picard space $\pic(\Ccal)$ is a group-like $\Ebb_k$-space. Moreover, there is a fibration of group-like $\Ebb_k$-spaces,
		\begin{equation*}
			\begin{tikzcd}
				B\aut(1_\Ccal)\rar & {\pic}(\Ccal) \rar["\pi_0"] & {\Pic}(\Ccal),
			\end{tikzcd}
		\end{equation*}
		where $\aut(1_\Ccal)$ is the automorphism space of $1_\Ccal$ in $\Ccal$ and $\Pic(\Ccal)$ is a discrete group (abelian when $k\ge 2$). This fibration sequence of spaces splits. 
	\end{lem}
	\begin{rem}\label{rem:Pic_pic}
		In \cite{MS_Picard}, the authors denoted the Picard space of a \emph{symmetric} monoidal $\infty$-category $\Ccal$ by $\mathcal{P}ic(\Ccal)$ and its Picard \emph{spectrum} by $\pic(\Ccal)$. By \Cref{lem:pic_Ccal_fib}, $\mathcal{P}ic(\Ccal)$ is a group-like $\einf$-space and hence can be identified with a connective spectrum. However, when $\Ccal$ is only $\Ebb_k$-monoidal for some $k<\infty$, its Picard space cannot be identified with a connective spectrum. 
		
		We use the notation $\pic(\Ccal)$ for the Picard \emph{space} of an $\Ebb_1$-monoidal $\infty$-category $\Ccal$ to make a distinction between Picard spaces and groups. 
	\end{rem}
	When there is no ambiguity, we will write $\pic(R)$ and $\Pic(R)$ for $\pic\left(\Mod_R(\Ccal)\right)$ and $\Pic\left(\Mod_R(\Ccal)\right)$, respectively (see \Cref{notn:Pic_R}).
	\begin{lem}
		Let $R\in \Alg_{\Ebb_2}(\Sp)$ be an $\Ebb_2$-ring spectrum. Then $\aut_R(R)$, the automorphism space of $R$ as an $R$-module spectrum is equivalent to $\GL_1R$, the space of units in the ring spectrum $R$ in the sense of \cite{units_ring_spec}. More precisely, we have a pullback diagram of spaces: 
		\begin{equation*}
			\begin{tikzcd}
				\GL_1R\rar \dar\ar[dr,phantom,very near start,"\lrcorner"]& \Omega^\infty R\dar \\
				\pi_0(R)^\times \rar & \pi_0(R).
			\end{tikzcd}
		\end{equation*}
	\end{lem}
	\begin{cor}\label{cor:pic_E2_ring}
		Let $R$ be an $\Ebb_2$-ring spectrum.  The homotopy groups of its Picard space are:
		\begin{equation*}
			\pi_t\left(\pic(R)\right)=\begin{cases}
				\Pic(R), & t=0;\\
				\pi_0(R)^\times, & t=1;\\
				\pi_{t-1}(R), & t\ge 2. 
			\end{cases}
		\end{equation*}
	\end{cor}
	\begin{exmps} 
		\begin{enumerate}
			\item When $R=S^0$, we have $\Mod_{S^0}(\Sp)=\Sp$, whose Picard space has homotopy groups:
			\begin{equation*}
				\pi_t(\pic(\Sp))=\begin{cases}
					\Z, & t=0; \qquad \textup{\cite{HMS_picard,Strickland_interpolation}}\\
					\Z^\times= \Z/2, & t=1; \\
					\pi_{t-1}(S^0) , & t\ge 2.
				\end{cases}
			\end{equation*}
			\item When $R=HA$ is the Eilenberg--MacLane spectrum for some classical (discrete)  commutative ring $A$, the Dold--Kan correspondence states that $\Mod_{HA}(\Sp)\simeq \Dcal_A$ as symmetric monoidal $\infty$-categories. This implies $\Pic(HA)\cong \Pic(\Dcal_A)\cong \Pic(A)\oplus C(\spec R,\Z)$ and 
			\begin{equation*}
				\pi_t(\pic(HA))=\begin{cases}
					\Pic(A)\oplus C(\spec R,\Z), & t=0;\\
					A^\times& t=1;\\
					0 & t\ge 2.
				\end{cases}
			\end{equation*}
		\end{enumerate}
	\end{exmps}
	\subsection{Picard spaces of $K(n)$-local ring spectra}
	The symmetric monoidal category we are studying in this paper is $\Sp_{K(n)}$, the $\infty$-category of $K(n)$-local spectra \cite{Hovey-Strickland_1999} \cite[\S 10.2]{Mathew_Galois}. 
	It admits the  structure of a presentably symmetric monoidal $\infty$-category $\left(\Sp_{K(n)}, \hsmash, L_{K(n)}S^0\right)$ such that: (see \cite[\S5.1.1]{CSY_ambiCHT})
	\begin{itemize}
		\item For all $X,Y\in \Sp_{K(n)}$, the $K(n)$-local monoidal product is defined by $X\hsmash Y= L_{K(n)}(X\Smash Y)$, where $\Smash$ is the smash product of spectra.   
		\item The localization functor $L_{K(n)}\colon \Sp\to \Sp_{K(n)}$ is strong symmetric monoidal.
		\item The inclusion $\Sp_{K(n)}\to\Sp$ is lax symmetric monoidal. 
	\end{itemize}  
	Suppose $R\in \Alg_{\Ocal}\left(\Sp_{K(n)}\right)$, where $\Ocal=\Ebb_k$ for some $k\ge 2$. By \cite[Lemma 3-6]{GGN_loop}, we have a fully faithful embedding $i\colon \Alg_{\Ocal}\left(\Sp_{K(n)}\right)\hookrightarrow \Alg_{\Ocal}(\Sp)$, with a left adjoint $L_{K(n)}\colon  \Alg_{\Ocal}(\Sp)\to  \Alg_{\Ocal}\left(\Sp_{K(n)}\right)$. A natural question is  to compare the Picard spaces $\pic\left(\Mod_R(\Sp)\right)$ and $\pic\left(\Mod_R\left(\Sp_{K(n)}\right)\right)$. 
	\begin{notn}\label{notn:Pic_R}
		Let $R\in \Alg_{\Ebb_k}\left(\Sp_{K(n)}\right)$ for some $2\le k\le \infty$. From now on, we will denote
		\begin{align*}
			\Mod(R)&:= \Mod_R(\Sp) &&& \Mod_{K(n)}(R)&:= \Mod_R\left(\Sp_{K(n)}\right)\\
			\pic(R)&:= \pic\left(\Mod_R\right) &&&\pic_{K(n)}(R)&:=\pic\left(\Mod_{K(n)}R\right)  \\
			\Pic(R)&:= \Pic\left(\Mod_R\right)&&&\Pic_{K(n)}(R)&:=\Pic\left(\Mod_{K(n)}R\right).
		\end{align*}
	\end{notn}
	As the localization functor $L_{K(n)}\colon \Sp\to \Sp_{K(n)}$ is strong symmetric monoidal, it induces a map of group-like $\Ebb_{k-1}$-spaces $\lambda\colon \pic(R)\to \pic_{K(n)}(R)$ for any $R\in \Alg_{\Ebb_k}\left(\Sp_{K(n)}\right)$.  For a group-like $\Ebb_k$-space $X$, its identity component $\tau_{\ge1} X$ is also a group-like $\Ebb_k$-space. 
	\begin{prop}\label{prop:pic_kn_truncation}
		For $R\in \Alg_{\Ebb_k}\left(\Sp_{K(n)}\right)$, the localization functor  induces an equivalence of connected group-like $\Ebb_{k-1}$-spaces \begin{equation*}
			(\tau_{\ge1}\lambda)\colon \tau_{\ge1}\pic(R)\simto \tau_{\ge1}\pic_{K(n)}(R).
		\end{equation*} 
	\end{prop}
	\begin{proof}
		By \Cref{lem:pic_Ccal_fib}, this is equivalent to showing $\aut_{\Mod(R)}(R)\simeq \aut_{\Mod_{K(n)}R}(R)$.  As  $\Sp_{K(n)}\to \Sp$ is a full subcategory,  any endomorphism of $R$ is automatically $K(n)$-local. As a result, the localization functor induces an equivalence of monoids: 
		\begin{equation*}
			\mathrm{End}_{\Mod(R)}(R)\simto \mathrm{End}_{\Mod_{K(n)}(R)}(R). 
		\end{equation*} 
		The claim follows by taking invertible elements in the endomorphism monoids. 
	\end{proof}
	One central object in $\Sp_{K(n)}$ is the height $n$ \textbf{Morava $E$-theory} $E_n$.  This is a Landweber exact spectrum defined using the Lubin--Tate deformation theory of the height $n$ Honda formal group $\Gamma_n$. The homotopy groups of $E_n$ are given by
	\begin{equation*}
		\pi_*(E_n)=\W\Fbb_{p^n}\llb u_1,\cdots, u_{n-1}\rrb[u^{\pm 1}], \qquad |u_i|=0, |u|=-2. 
	\end{equation*}
	The graded Lubin--Tate deformation ring $\pi_*(E_n)$ admits an action by the \textbf{Morava stabilizer group} $\Gbb_n:=\aut(\Gamma_n/\Fbb_{p^n})\rtimes \gal(\Fbb_{p^n}/\Fp)$, where $\aut(\Gamma_n/\Fbb_{p^n})$ is the automorphism group of the Honda formal group $\Gamma_n$ (extended to the field $\Fbb_{p^n}$). 
	\begin{thm}[Goerss--Hopkins--Miller, Lurie, {\cite{Lurie_Elliptic2,Goerss-Hopkins_moduli_spaces,Rezk_Hopkins-Miller}}]\label{thm:GHML}
		The spectrum $E_n$ is a $K(n)$-local $\einf$-ring spectrum. The $\Gbb_n$-action on $\pi_*(E_n)$ lifts to $\einf$-ring automorphisms of $E_n$.
	\end{thm}
	Following the discussion above, the two Picard spaces $\pic(E_n)$ and $\pic_{K(n)}(E_n)$ of $E_n$ can only differ in their path components. By  \cite[Lemma 6.7]{Heard_2021Sp_kn-local}, we have an isomorphism $\Pic(E_n)\cong \Pic_{K(n)}(E_n)$ of Picard groups. This follows from \cite[Proposition 10.11]{Mathew_Galois}. 
	\begin{thm}[Baker--Richter, {\cite{Baker-Richter_invertible_modules}}]
		The Picard group of $E_n$ in $\Sp$ is algebraic in the sense of an isomorphism 
		\begin{equation*}
			\Pic(\textup{graded $(E_n)_*$-}\Mod)\simto \Pic(E_n).
		\end{equation*}
		The former is isomorphic to $\Z/2$ since $E_n$ is even-periodic and $\pi_0(E_n)$ is a complete local ring. 
	\end{thm}
	By \Cref{cor:pic_E2_ring}, the homotopy groups of the $K(n)$-local Picard space of $E_n$ are:
	\begin{equation}\label{eqn:Pic_En}
		\pi_t\left(\pic_{K(n)} (E_n)\right)\cong \pi_t\left(\pic (E_n)\right)=\begin{cases}
			\Z/2, & t=0;\\
			\pi_0(E_n)^\times, & t=1;\\
			\pi_{t-1}(E_n), & t\ge 2.
		\end{cases}
	\end{equation}
	\subsection{Algebra objects in monoidal $\infty$-categories}\label{subsec:Alg_O}    
	We conclude this section with a brief  discussion of algebra objects in monoidal $\infty$-categories  following \cite{Lurie_DAG_III,Lurie_HA}. Let $\Ccal$ be a symmetric monoidal $\infty$-category. This data is encoded in a co-Cartesian fibration $p\colon \Ccal^{\otimes}\to N(\mathsf{Fin}_*)\simeq \einf^{\otimes}$ of $\infty$-operads. 
	\begin{defn}\label{defn:OAlg}
		Let $\Ocal$ be an $\infty$-operad with a map $p'\colon \Ocal^\otimes\to \einf^\otimes$.  An $\Ocal$-algebra in $\Ccal$ is a map of $\infty$-operads $\alpha\colon \Ocal^{\otimes}\to \Ccal^{\otimes}$ such that $p'\circ \alpha\simeq p$. Denote the $\infty$-category of $\Ocal$-algebras in $\Ccal$ by $\Alg_\Ocal(\Ccal)$. When $p=\id\colon \einf^\otimes\to \einf^\otimes$, the corresponding algebra object in $\Ccal$ is called commutative and we denote $\CAlg:=\Alg_{\einf}$. 
	\end{defn} 
	\begin{prop}[{\cite[Corollaries 3.2.2.5 and 3.2.3.2]{Lurie_HA}}]\label{prop:alg_O_lim}
		If a symmetric monoidal $\infty$-category $\Ccal$ has all limits and filtered colimits, then so does $\Alg_\Ocal (\Ccal)$. In that case, the forgetful functor $U\colon \Alg_\Ocal(\Ccal)\to\Ccal$ preserves and detects limits and filtered colimits. 
	\end{prop}
	\begin{prop}[{\cite[Proposition 7.1.2.6]{Lurie_HA}}]\label{prop:monoidal_mod_cat}
		Let $R\in \Alg_{\Ebb_k}(\Ccal)$, where $2\le k\le \infty$.  Then the $\infty$-category $\Mod_R(\Ccal)$ of its left modules in $\Ccal$ is $\Ebb_{k-1}$-monoidal. 
	\end{prop}
	We will work very closely with limits of $\infty$-categories in this paper. There are several different settings where those limits can take place. The default is $\Cat_\infty$, the  $\infty$-category of $\infty$-categories. Another natural setting is $\mathrm{Pr}^L$, the $\infty$-category of presentable $\infty$-categories with small colimit preserving functors. By \cite[Proposition 5.5.3.13]{HTT}, the inclusion $\mathsf{Pr}^L\subseteq \Cat_\infty$ preserves all limits. One very useful fact is: 
	\begin{prop}\label{prop:einf}
		Let $\Ccal$ be a symmetric monoidal $\infty$-category whose underlying $\infty$-category is complete. 
		\begin{enumerate}
			\item Consider a sequence of algebra objects $\cdots \to R_{k+1}\to R_k\to \cdots \to R_1$ in $\Ccal$. If $R_k$ is an $\Ebb_k$-algebra over $R_{k+1}$ for all $k$, then $R:=\lim_k R_k$ admits a structure of an $\einf$-algebra in $\Ccal$ such that $R\to R_k$ is a morphism of $\Ebb_k$-algebras in $\Ccal$. 
			\item Let $R, R'\in \CAlg(\Ccal)$. Suppose $f\colon R\to R'$ is a morphism in $\Alg_{\Ebb_k}(\Ccal)$ for all $k$, then $f$ is a morphism in $\CAlg(\Ccal)$. 
		\end{enumerate}
	\end{prop}
	\begin{proof}
		By \cite[Corollary 3.2.2.5]{Lurie_HA} (\Cref{prop:alg_O_lim}), the $\infty$-categories $\Alg_{\Ebb_k}(\Ccal)$ are complete for all $k$ when the underlying $\infty$-category $\Ccal$ is complete. It follows that the limit $R$ is an $\Ebb_k$-algebra for all $k$. This implies that $R$ is $\einf$ as explained below. Recall from \Cref{defn:OAlg}, an $\Ocal$-algebra in $\Ccal$ is a lax monoidal functor of $\infty$-operads $\Ocal^{\otimes}\to \Ccal^\otimes$. Denote by $\Ebb_k^\otimes$ the $\infty$-operad associated to $\Ebb_k$-algebras. The colimit of the $\infty$-operads $\Ebb_k^\otimes$  is equivalent to the commutative $\infty$-operad $\einf^\otimes$ by \cite[Corollary 5.1.1.5]{Lurie_HA}.  This implies: 
		\begin{equation*}
			\CAlg\left(\Ccal\right)= \Fun^{\lax}\left(\einf^\otimes,\Ccal^\otimes\right)\simeq \Fun^{\lax}\left(\colim_k\Ebb_k^\otimes,\Ccal^\otimes\right)\simeq \lim_k \Fun^{\lax}\left(\Ebb_k^\otimes,\Ccal^\otimes\right)=\lim_k \Alg_{\Ebb_k}\left(\Ccal\right). 
		\end{equation*}
		Both claims then follow from the descriptions of objects and morphisms in the limit category on the right hand side. 
	\end{proof}
	\begin{cor}\label{cor:lim_Ek_monoidal}
		Let $\{\cdots\to \Ccal_{k+1}\to \Ccal_k\to \cdots\to \Ccal_1\}$ be an inverse system of monoidal $\infty$-categories such that $\Ccal_k$ is $\Ebb_{k}$-monoidal and the functor $\Ccal_{k+1}\to \Ccal_k$ is $\Ebb_{k}$-monoidal. Then $\lim_k\Ccal_k$ has a natural $\einf$-monoidal structure such that $\lim_k\Ccal_k\to \Ccal_j$ is $\Ebb_{j}$-monoidal for any $j$.  Moreover, let $f\colon \Ccal\to \Dcal$ be a functor such that $\Ccal$ and $\Dcal$ are both symmetric monoidal. If $f$ is $\Ebb_{k}$ for any $k$, then it is a symmetric monoidal functor. 
	\end{cor}
	\begin{proof}
		Let $\Ocal^\otimes$ be an $\infty$-operad. By \cite[Corollary 1.4.15]{Lurie_DAG_III}, an $\Ocal$-monoidal $\infty$-category is an $\Ocal$-algebra object in $\Cat_\infty$ with  Cartesian  symmetric monoidal product.  The claim now follows from \Cref{prop:einf} since $\Cat_\infty$ is closed under small limits. 
	\end{proof}
	\section{The inverse limit  topology on $K(n)$-local Picard groups}\label{sec:prodisc_topology}
	The Picard group of the $K(n)$-local category $\Sp_{K(n)}$ has a profinite topology indexed by a tower of generalized Moore spectra.  This has been described in \cite{Hovey-Strickland_1999,Goerss-Hopkins}. In this section, we show that the $K(n)$-local Picard functors $\pic_{K(n)}$ and $\Pic_{K(n)}$ can be lifted to \emph{pro}-spaces and \emph{pro}-groups, respectively. The inverse limit  topology on $\Pic\left(\Sp_{K(n)}\right)$ then recovers the profinite topology in those earlier works. Moreover, it plays an essential role in our proof of the profinite descent spectral sequence for Picard spaces in \Cref{thm:pic_ss}, which is in turn the foundation of all other main results of this paper. 
	
	Here are more details. In \Cref{prop:genMoore_mult}, we incorporate Burklund's result in \cite{Burkland_mult_Moore} to refine the tower of generalized Moore spectra $\{M_j\}$ of type $n$ in \cite{Hovey-Strickland_1999}.  Then we show in \Cref{thm:R-mod_moore} that the module category of a $K(n)$-local $\einf$-ring spectrum is equivalent to the limit of its base changes by the generalized Moore spectra of type $n$ with some $\Ebb_k$-algebra structure:
	\begin{equation*}
		\Mod_{K(n)}(R)\simto \lim_j \Mod(R\Smash M_j). 
	\end{equation*}
	This can be viewed as a Grothendieck existence theorem (see \Cref{rem:Grothendieck_existence}) for $K(n)$-local ring spectra. Once the equivalence is established, it is a formal argument to prove it is symmetric monoidal and hence induces an equivalence of Picard spaces and isomorphism of Picard groups: (\Cref{thm:Kn_Pic_limit})
	\begin{equation*}
		\pic_{K(n)}(R)\simto \lim_j \pic(R\Smash M_j),\qquad \Pic_{K(n)}(R)\cong \lim_j \Pic(R\Smash M_j).
	\end{equation*}
	This isomorphism endows $K(n)$-local Picard groups with a natural  inverse limit  topology. In \Cref{prop:Pic_Kn_profin_prod} and \Cref{prop:Pic_Kn_colim}, we prove that the $K(n)$-local Picard group functor $\Pic_{K(n)}\colon \CAlg\left(\Sp_{K(n)}\right) \to \Pro(\Ab)$ commutes with profinite products (see \Cref{defn:profin_prod_Kn}) and filtered colimits. 
	
	\subsection{Generalized Moore spectra of type $n$}
	A generalized Moore spectrum $M$ of type $n$ is a finite CW spectrum constructed inductively as follows:
	\begin{enumerate}
		\item a type $0$  generalized Moore spectrum is simply the sphere spectrum $S^0$;
		\item a type $1$ generalized Moore spectrum is the cofiber of   the multiplication-by-$p^k$ self-map on the type $0$ spectrum $S^0$;
		\item  a type $n$ generalized Moore spectrum is the cofiber of a $v_{n-1}$-self map on a generalized Moore spectrum of type $n-1$. The existence of such self-maps was proved by Hopkins--Smith in \cite[Theorem 9]{DHS_nilp2}. 
	\end{enumerate}
	See the precise definition in  \cite[Definition 4.12]{Hovey-Strickland_1999}. Generalized Moore spectra are the building blocks of the $K(n)$-local category. This is illustrated by the following results of Hovey--Strickland.
	\begin{thm}[Hovey--Strickland, {\cite[Propositions 4.22 and 7.10.(e)]{Hovey-Strickland_1999}}]\label{thm:genMoore}
		For each height $n$, there is a tower of generalized Moore spectra
		\begin{equation*}
			\cdots \xrightarrow{g_3} M_2\xrightarrow{g_2} M_1\xrightarrow{g_1}M_0,
		\end{equation*}
		such that:
		\begin{enumerate}
			\item All the $M_j$'s are finite complexes of type $n$.
			\item $(E_n)_*(M_j)\cong (E_n)_*/J_j$ for some open invariant ideal $J_j\trianglelefteq (E_n)_*$.
			\item $\bigcap_j J_j=\{0\}$.
			\item $M_j$'s are $\mu$-spectra, i.e. spectra with a left unital multiplication (see \cite[Definition 4.8]{Hovey-Strickland_1999}). Their unit maps $\eta_j\colon S^0\to M_j$ satisfy $\eta_{j+1}=\eta_j\circ g_j$ for all $j$.
		\end{enumerate}
		Denote the Bousfield localization at the height $n$ Morava $E$-theory $E_n$ by $L_n$. For any  spectrum $X$, we have a natural equivalence:
		\begin{equation}\label{eqn:Moore_limit}
			L_{K(n)}X\simeq \holim_{j}(L_nX\Smash M_j).
		\end{equation}
		In particular, if $X$ is $E_n$-local, then $L_{K(n)}X\simeq \holim_{j}(X\Smash M_j)$. When $X\in \Sp_{K(n)}\subseteq \Sp_{E_n}$, we further have $X\simeq \holim_{j}(X\Smash M_j)$.
	\end{thm}
	This shows any $K(n)$-local spectrum is a \emph{pro-spectrum} indexed by a tower of generalized Moore spectra. In the last decade, significant progress  has been made in our understanding of multiplicative structures on generalized Moore spectra, for example \cite{BK_2022, Davis-Lawson_2014,Devinatz2017}. Recently, Burklund proved: 
	\begin{thm}[Burklund,  {\cite{Burkland_mult_Moore}}]\label{thm:genMoore_mult} Let $R$ be an $\Ebb_{m+1}$-ring spectrum  where $m\ge 2$. Suppose $v\in \pi_{\text{even}}(R)$ is an element such that the cofiber $R/v$ is a $\mu$-spectrum. Then the tower of $v^i$-cofibers of $R$
		\begin{equation*}
			\cdots \longrightarrow R/v^{i+1}\longrightarrow R/v^i\longrightarrow \cdots \longrightarrow R/v
		\end{equation*}
		satisfies that $R/v^q$ is an $\Ebb_j$-$R/v^{q+1}$-algebra when $j\le m$ and $q>j$. Moreover, we have: 
		\begin{enumerate}
			\item The Moore spectrum $S^0/2^q$ is an $\Ebb_j$-$S^0/2^{q+1}$-algebra when $q\ge \frac{3}{2}(j+1)$.
			\item For an odd prime $p$, the Moore spectrum $S^0/p^q$ is an $\Ebb_j$-$S^0/p^{q+1}$-algebra when $q\ge j+1$. 
			\item For any prime $p$, height $n$, and natural number $j$, there is a $p$-local generalized Moore spectrum $M$ of type $n$ that admits an $\Ebb_j$-ring spectrum structure. 
		\end{enumerate} 
	\end{thm}
	Burklund's result allows us to lift the equivalence in \eqref{eqn:Moore_limit} to the level of $K(n)$-local ring spectra in \Cref{thm:Kn_ring_limit}. We begin by rigidifying the tower of generalized Moore spectra in \Cref{thm:genMoore} as follows.
	\begin{prop}\label{prop:genMoore_mult}
		The generalized Moore spectrum $M_j$ in \Cref{thm:genMoore} can be chosen so that it is an $\mathbb{E}_{j}$-algebra over $M_k$ for any $k\ge j\ge 1$. 
	\end{prop}
	\begin{proof}
		The strategy is to incorporate the criterion in Burklund's  \Cref{thm:genMoore_mult} into Hovey--Strickland's proof of the first half of \Cref{thm:genMoore} in \cite[Proposition 4.22]{Hovey-Strickland_1999}. We will construct the tower $\{\cdots \to M_{j+1}\to M_j\to \cdots \to M_1\}$ by induction on the height $n$.
		
		For $n=1$, this is already stated in \Cref{thm:genMoore_mult}. Suppose that a tower of type-($n-1$) generalized Moore spectra $\{\cdots\to W_3\to W_2\}$ has been constructed such that $W_j$ is an $\Ebb_{j}$-algebra over $W_k$ for any $k\ge j\ge 2$.  We will construct a type-$n$ tower $\{M_j\}$ as cofibers of $v_n$-self maps of $\{W_{j+1}\}$. It is necessary to start with an $\Ebb_{j+1}$-algebra $W_{j+1}$ since we need $\Ebb_{j+1}$-algebras to produce $\Ebb_j$-quotients by \Cref{thm:genMoore_mult}. 
		
		Set $M_2=W_3/v^2$ where $v$ is some type-$n$ self map of $W_3$ such that $W_3/v$ is a $\mu$-spectrum. The existence of such a map $v$ is guaranteed by \cite[Proposition 4.11]{Hovey-Strickland_1999}. Then \Cref{thm:genMoore_mult} implies that $M_2$ is an $\Ebb_2$-$W_3$-algebra.   Suppose we have constructed $M_{j-1}\to\cdots\to M_2$ with desired properties and $M_{j-1}$ is an $\Ebb_{j-1}$-$W_{j}$-algebra. The induction step in the proof of \cite[Proposition 4.22]{Hovey-Strickland_1999} produces a map between cofiber sequences.
		\begin{equation*}
			\begin{tikzcd}
				\Sigma^{|w_{j+1}|}W_{j+1}\rar["w_{j+1}"]\ar[dd,"w_{j+1}^{N}\circ g_{j+1}"'] & W_{j+1} \ar[rr,"q_{j+1}"]\ar[dd,"g_{j+1}"]\ar[dr,phantom,very near end,"\ulcorner"] &[15pt]&[-15pt] M_{j}\ar[dd,"f_j"]\ar[dl,"\bar{g}_{j+1}"']\\
				[10pt]&&P\ar[dr,dashed,"f'_j"]&\\
				[-10pt]\Sigma^{|w_{j}|} W_{j}\rar["w_{j}"']&  W_{j}\ar[rr,"q_{j}"']\ar[ur,"\bar{q}_{j+1}"]&& M_{j-1} 
			\end{tikzcd}
		\end{equation*}
		A priori, the right square above consists of $\mu$-spectrum maps. By our inductive hypothesis, $g_{j+1}$ is a map of $\Ebb_{j}$-rings and $q_{j}$ is a map of $\Ebb_{j-1}$ rings. We need to find a self map $w_{j+1}$ of $W_{j+1}$ so that the map $q_{j+1}$ is an $\Ebb_{j}$-ring map and $f_{j}$ is an $\Ebb_{j-1}$-ring map. \Cref{thm:genMoore_mult} implies that $q_{j+1}$ can be made $\Ebb_{j}$ when we replace $w_{j+1}$ by its $j$-th power. Replacing $N$ with a larger power accordingly, we can make the left square commute as in the proof of \cite[Proposition 4.22]{Hovey-Strickland_1999}. Then we obtain a new $f_j$ as a map between cofibers. 
		
		To prove $f_j$ is $\Ebb_{j-1}$, consider the pushout $P$ of the $\Ebb_{j}$-ring maps $q_{j+1}$ and $g_{j+1}$ is the diagram above. Then $\bar{g}_{j+1}$ and $\bar{q}_{j+1}$ are both $\Ebb_{j}$-ring maps. So it suffices to show $f'_j\colon P=M_j\Smash_{W_{j+1}}W_{j}\to M_{j-1}$ is $\Ebb_{j-1}$. One can check that $f'_j$ sits in a cofiber sequence:
		\begin{equation*}
			M_j\Smash_{W_{j+1}}\Sigma^{N|w_{j}|}W_{j}  \xrightarrow{1\Smash w_{j}^N}  M_j\Smash_{W_{j+1}}W_{j}\xrightarrow{f_j'} M_{j-1}. 
		\end{equation*} 
		As the quotient-by-$w_{j}$ map $q_{j}$ is already $\Ebb_{j-1}$, the quotient-by-$w_{j}^N$ map $f_j'$ must also be $\Ebb_{j-1}$ by \Cref{thm:genMoore_mult}. Consequently, we have constructed a complex $M_j$ such that $q_{j+1}\colon W_{j+1}\to M_j$ is an $\Ebb_{j}$-ring map and $f_j\colon M_{j}\to M_{j-1}$ is an $\Ebb_{j-1}$-ring map. This finishes the inductive step. 
	\end{proof}
	\begin{thm}\label{thm:Kn_ring_limit}
		There is an equivalence of $K(n)$-local $\einf$-rings $S_{K(n)}^0\to \lim_j L_n M_j$. 
	\end{thm}	
	\begin{proof}
		The equivalence on the level of $K(n)$-local spectra follows from Hovey--Strickland's \eqref{eqn:Moore_limit}. By \Cref{prop:alg_O_lim}, it remains to show that this is a map of $\einf$-rings. \Cref{prop:genMoore_mult} implies that the map $S^0_{K(n)}\to L_nM_j$ is a map of $\Ebb_{j}$-rings for any $j\ge 1$.  Then it is a map of $\einf$-rings by  \Cref{prop:einf}. 
	\end{proof}
	\begin{rem}
		Davis--Lawson proved in \cite{Davis-Lawson_2014} that any tower of generalized Moore spectra $\{M_j\}$ of type $n$ in \Cref{thm:genMoore} is an $\einf$-algebra in the category of \emph{pro}-spectra. In view of \Cref{prop:genMoore_mult}, Burklund's \Cref{thm:genMoore_mult} in \cite{Burkland_mult_Moore} is a refinement of Davis--Lawson's result. 
	\end{rem}
	\subsection{Limits of module categories}
	In the remaining of the paper, fix a tower of generalized Moore spectra $\{M_j\}$ satisfying both \Cref{thm:genMoore} and \Cref{prop:genMoore_mult}. For $R\in \CAlg\left(\Sp_{K(n)}\right)$, we will lift the equivalence $R\simto \lim_j R\Smash M_j$ in \eqref{eqn:Moore_limit} to the level of module categories, first as $\infty$-categories in \Cref{thm:R-mod_moore}, then as symmetric monoidal $\infty$-categories in \Cref{prop:symm_equiv}. This lifting is not a formal result, as explained in \Cref{rem:lim_mod}. \Cref{thm:Kn_Pic_limit} is then obtained by applying the Picard space functor to the equivalence in \Cref{thm:R-mod_moore}.  
	\begin{thm}\label{thm:R-mod_moore}
		Let $R$ be a $K(n)$-local $\einf$-ring spectrum. Denote by $R_j$ the $\Ebb_{j}$-ring $R\Smash M_j$. Then the limit of base change functors is an equivalence of $\infty$-categories:
		\begin{equation*}
			F\colon \Mod_{K(n)}R\simto \lim_j\Mod_{K(n)} {R_j}\overset{\cong}{\longleftarrow} \lim_{j} \Mod(R_j).
		\end{equation*}
	\end{thm}
	Before the proof, we make some observations. 
	\begin{rem}\label{rem:lim_mod}
		Limits of (discrete) commutative rings do not necessarily commute with taking module (derived) categories. One  counterexample is  $\Dcal\Zp\not\simeq \lim\limits_{k} \Dcal\Z/p^k$, where $\Qp$ is a $\Zp$-module but not a limit of $\Z/p^k$-modules/complexes. This discrepancy is fixed if we consider $p$-complete (in the derived sense) $\Zp$-chain complexes instead. Indeed, there is an equivalence of categories:
		\begin{equation*}
			\left(\Dcal \Zp\right)^\wedge_p\simto \lim_k \Dcal \Z/p^k. 
		\end{equation*}
		This is essentially the height $n=1$ case of \Cref{thm:R-mod_moore}, since the Moore spectrum $S^0/p^k$ is a finite complex of type $1$.
		
		Alternatively, we can restrict to \emph{coherent} modules over $\Zp$. (Note $\Qp$ is not a coherent $\Zp$-module.) The equivalence then follows from the Grothendieck existence theorem below.
	\end{rem}
	\begin{rem}
		\label{rem:Grothendieck_existence} 	\Cref{thm:R-mod_moore} can be viewed as a Grothendieck existence theorem for $K(n)$-local commutative ring spectra.  In algebraic geometry, the Grothendieck existence theorem states:
		\begin{enumerate}
			\item \textup{(Coherent sheaf version, \cite[Theorem 5.1.4]{EGA_III1})}\quad Let $R$ be a Noetherian ring complete with respect to an ideal $I$, and $X$ be a proper $R$-scheme. Write $X_j:=X\times_R \spec R/I^j$ and $\mathfrak{X}=\colim_j X_j$ for the completion of $X$ at $I$. Then the restriction functor induces an equivalence of categories of coherent sheaves:
			\begin{equation*}
				\Coh(\mathfrak{X})\simto \lim_j\Coh(X_j). 
			\end{equation*}
			In \cite[Theorem 1.3]{Alper-Hall-Rydh}, Alper--Hall--Rydh proved this statement for certain quotient stacks. 
			\item \textup{(Functor points of view version, \cite[Theorems 2.20 and 2.22]{Alper2015artin})} \quad Let $\mathfrak{X}$ be an algebraic stack locally of finite type over a field $k$. Then for every complete local Noetherian $k$-algebra $(R,\mfrak)$, the restriction functor induces an equivalence of categories (groupoids):
			\begin{equation*}
				\mathfrak{X}(\spec R)\simto \lim_j \mathfrak{X}(\spec R/\mfrak^j).
			\end{equation*}
		\end{enumerate}
	\end{rem}
	\begin{lem}\label{lem:Rk-mod}
		Let $R$ be a $K(n)$-local ring spectrum. Then any $R_k=R\Smash M_k$-module spectrum $A$ in $\Sp$ is $K(n)$-local. Consequently, $K(n)$-localization induces an equivalence of $\Ebb_{k-1}$-monoidal $\infty$-categories:
		\begin{equation*}
			L_{K(n)}\colon \Mod(R_k)\simto \Mod_{K(n)}(R_k). 
		\end{equation*}
	\end{lem}
	\begin{proof}
		Note that $R_k$ is $K(n)$-local because $R$ is $K(n)$-local and $M_k$ is a finite complex. As an $M_k$-module spectrum, $R_k$ is a retract of $M_k\Smash R_k$, since the composition $R_k\xrightarrow{\eta_k\Smash 1} M_k\Smash R_k\to R_k$ is the identity map.  \cite[Theorem 13.1]{Hovey-Strickland_1999} then implies that $R_k\Smash A\in \Sp_{K(n)}$. From this, we conclude $A$ is $K(n)$-local as a retract of $R_k\Smash A$. 
		
		Now consider the monoidal functors $\Sp\xrightarrow{L_{K(n)}}\Sp_{K(n)}\xrightarrow{i} \Sp$. By \cite[Remark 1.1.11]{Ergus_thesis}, the two functors induce $\Ebb_{k-1}$-monoidal functors on the (left) module categories:
		\begin{equation*}
			\Mod(R_k)\longrightarrow \Mod_{K(n)}(R_k)\longrightarrow \Mod(R_k). 
		\end{equation*}
		As any  $R_k$-module is automatically $K(n)$-local, the two functors on module categories are inverse to each other. Hence, $L_{K(n)}\colon \Mod(R_k)\simto \Mod_{K(n)}(R_k)$ is an equivalence of $\Ebb_{k-1}$-monoidal $\infty$-categories.
	\end{proof}
	\begin{proof}[Proof of \Cref{thm:R-mod_moore}] 
		The isomorphism $\lim_j\Mod_{K(n)} {(R_j)}\cong \lim_j\Mod (R_j)$ follows from \Cref{lem:Rk-mod}. To establish the first equivalence $\Mod_{K(n)}R\simto \lim_j\Mod_{K(n)} {R_j}$, we prove the equivalence by defining two functors $F$ and $G$ inverse to each other
		\begin{equation*}
			\begin{tikzcd}
				F \colon \Mod_{K(n)}R\rar[shift left=1 ex,""{name=F}]& \lim\limits_{j} \Mod_{K(n)}R_j\lar[shift left=1 ex,""{name=G}] \ar[from=F,to=G,phantom,"{\perp}"] \rcolon G.
			\end{tikzcd}
		\end{equation*}
		By \cite[Remark 4.5.3.2 and  Proposition 4.6.2.17]{Lurie_HA}, we have a base change functor $\Mod_{K(n)}R\to \Mod_{K(n)}R_j$ which is left adjoint to a forgetful functor. This is a map in $\mathsf{Pr}^L$. The functor $F$ is then obtained using the universal property of limits in $\mathsf{Pr}^L$, which can be computed in $\Cat_\infty$ by \cite[Proposition 5.5.3.13]{Lurie_HA}. As the limit of left adjoints is still a left adjoint, we obtain the functor $G$ as a right adjoint of $F$. 
		
		Next, we give a more concrete description of the functors $F$ and $G$ on objects. Note that objects in the limit category are towers of $R_j$-module spectra $\{A_j\}$ with equivalences of $R_{j-1}$-modules $\alpha_j\colon R_{j-1} \Smash_{R_j}A_j  \simto  A_{j-1}$. The functor $F$ sends $A$ to $\{A_j = R_j\Smash_{R}A\simeq M_j\Smash A\}$. The functor $G$ sends a tower of $R_j$-modules $\{A_j\}$ to $L_{K(n)}\lim_{j} A_j$ in \emph{$K(n)$-local} $R$-modules, where structure maps are given by:
		\begin{equation*}
			\begin{tikzcd}
				A_j\simeq R_j\Smash_{R_j} A_j\rar& R_{j-1}\Smash_{R_j} A_j\rar["\sim","\alpha_j"']& A_{j-1}. 
			\end{tikzcd}
		\end{equation*}
		Note that $\lim_{j} A_j$ is already $K(n)$-local, since $\Sp_{K(n)}$ is closed under limits. 
		The $R$-module structure on $G(\{A_j\})=L_{K(n)}\lim_{j} A_j\simeq \lim_{j} A_j$ is given by the limit of the maps 
		\begin{equation*}
			R\Smash \left( \lim_{j} A_j\right)\longrightarrow R_k\Smash A_k\longrightarrow A_k. 
		\end{equation*}
		As $\lim_{j} A_j\in \Sp_{K(n)}$, the structure map of $\lim_{j} A_j$ as an $R$-module has a canonical factorization:
		\begin{equation*}
			\begin{tikzcd}[row sep=0]
				R\Smash \left( \lim\limits_{j} A_j\right)\ar[rr] \ar[dr,end anchor=west]&&\lim\limits_{j} A_j\\
				&R\hsmash \left( \lim\limits_{j} A_j\right)\ar[ur,dashed,start anchor=east]&
			\end{tikzcd}
		\end{equation*}
		This shows $\lim_{j} A_j\in \Mod_{K(n)}R$ and the functor $G$ is well-defined.  Next we prove that $F$ is an equivalence of $\infty$-categories with inverse $G$. Note that any $K(n)$-local spectrum $A$ is $E_n$-local, \eqref{eqn:Moore_limit} implies that the composition $G\circ F$ is naturally homotopic to $\id$. The harder direction is to show the composition $F\circ G$ is also homotopic to $\id$.  For any $ \{A_j\} \in \lim_{j} \Mod_{K(n)}R_j$, we have 
		\begin{equation*}
			F\circ G (\{A_j\}) = \left\{R_k\Smash_R\left(\lim_{j}A_j\right)\right\}\simeq \left\{\lim_{j}\left(M_k\Smash A_j\right)\right\}.
		\end{equation*}
		We now construct a map $\{f_j\}$ between the two towers $F\circ G (\{A_j\})$ and $\{A_j\}$ such that each $f_j$ is an equivalence. By \cite[Proposition 4.16]{Hovey-Strickland_1999}, there are splittings when $j> k$. 
		\begin{equation}\label{eqn:Mk_Aj_splitting}
			M_k\Smash A_j\simeq M_k\Smash M_j\Smash_{M_j} A_j\simeq   (M_k\Smash M_j)\Smash_{M_j} A_j\simeq \left(M_k\vee \bigvee_{i} \Sigma^{d_{j,i}}M_k\right)\Smash_{M_j} A_j\simeq A_k\vee \bigvee_{i} \Sigma^{d_{j,i}}A_k
		\end{equation}
		where $d_{j,i}>0$  and the last equivalence is a wedge of suspensions of the following equivalences:
		\begin{equation*}
			M_k\Smash_{M_j} A_j\simeq R_k\Smash_{R_j} A_j\simeq R_k\Smash_{R_{j-1}}R_{j-1}\Smash_{R_j} A_j\xrightarrow[1\Smash \alpha_j]{\sim} R_k\Smash_{R_{j-1}}A_{j-1}\simto\cdots \simto R_k\Smash_{R_{k+1}}A_{k+1}\xrightarrow[\alpha_{k+1}]{\sim}A_k. 
		\end{equation*}
		Taking the limit of $M_k\Smash A_j$ as $j\to \infty$, the structure map on the $A_k$-summand is the identity, and are multiplications by nilpotent elements in $R_k$ on all other summands. This yields an equivalence $f_k\colon A_k\simto \lim_j M_k\Smash A_j$, which is induced by the inclusion of the first summand in \eqref{eqn:Mk_Aj_splitting}.  The compatibility of $f_k$ then follows from the commutative diagram when $j\ge k$:
		\begin{equation*}
			\begin{tikzcd}
				A_k\rar\dar& M_k\Smash A_j\dar\\ 
				A_{k-1}\rar & M_{k-1}\Smash A_j.
			\end{tikzcd}
		\end{equation*}
		The maps between towers $\{f_j\colon A_j\to R_j\Smash_{R} \left(\lim_{k}A_k\right)\}$ then assemble into an equivalence in $\lim\limits_{j} \Mod_{K(n)}R_j$. Consequently, we have proved the composition $F\circ G$ is homotopic to $\id$. 
	\end{proof}
	Next we will upgrade \Cref{thm:R-mod_moore} to an equivalence of symmetric monoidal $\infty$-categories. By \Cref{prop:monoidal_mod_cat}, each of the categories $\Mod({R_j})$ in the limit system is $\Ebb_{j-1}$-monoidal. 
	\begin{prop}\label{prop:symm_equiv}
		The equivalence in \Cref{thm:R-mod_moore} is symmetric monoidal. 
	\end{prop}
	\begin{proof}
		Recall from \Cref{prop:genMoore_mult} that the generalized Moore spectrum $M_j$ is an $\Ebb_{j}$-ring over $M_{j+1}$. By \cite[Proposition 7.1.2.6]{Lurie_HA} and \Cref{lem:Rk-mod}, the base change and $K(n)$-localization functors are both $\Ebb_{j-1}$-monoidal:
		\begin{align*}
			\Mod_{K(n)}R&\xrightarrow{R_j\Smash_R-}  \Mod_{K(n)}R_j\xleftarrow[\simeq]{L_{K(n)}} \Mod({R_j}),\\
			\Mod_{K(n)} {R_{j+1}}&\xrightarrow{R_{j}\Smash_{R_{j+1}}-} \Mod_{K(n)}R_j\xleftarrow[\simeq]{L_{K(n)}} \Mod({R_j}).
		\end{align*}
		As equivalences in \Cref{thm:R-mod_moore} are  limits of $\Ebb_{j}$-monoidal maps, we conclude from \Cref{cor:lim_Ek_monoidal} that it is also a map of $\einf$-monoidal categories. 
	\end{proof}
	\begin{lem}
		Let $\{\cdots \to X_n\to X_{n-1}\to\cdots\to X_1\}$ be an inverse system of group-like spaces such that $X_k$ is a group-like $\Ebb_{k}$-space and $X_{k+1}\to X_k$ is a map of group-like $\Ebb_{k}$-spaces. Then $X=\holim X_k$ is a group-like $\einf$-space and the map $X\to X_k$ is a morphism of group-like $\Ebb_k$-spaces. 
	\end{lem}
	\begin{proof}
		Recall that a group-like $\Ebb_{k}$-space is a group-like $\Ebb_{k}$-algebra in $\Top$. The claim then follows from \Cref{prop:einf}. Alternatively, this statement can be proved using the Recognition Principle that $X_k$ is a group-like $\Ebb_{k}$-space iff it is equivalent to $\Omega^kY_k$ for some $Y_k$. 
	\end{proof}
	\begin{prop}[Mathew--Stojanoska, {\cite[Proposition 2.2.3]{MS_Picard}}]\label{prop:pic_lim_colim}
		The Picard space functor $\Cat^{\Ebb_k}_\infty\to \Alg^\mathrm{gp}_{\Ebb_k}(\Top)$ commutes with limits and filtered colimits. \footnote{While \cite[Proposition 2.2.3]{MS_Picard} is a statement about the Picard spectrum and space functors of  $\einf$-ring spectra, its proof works for $\Ebb_k$-ring spectra as well.}
	\end{prop}
	Applying the Picard space functor to the monoidal equivalence in \Cref{thm:R-mod_moore}, we finally conclude:
	\begin{mainthm}\label{thm:Kn_Pic_limit}
		Let $R$ be a $K(n)$-local $\einf$-ring spectrum.  The limit of base change maps of Picard spaces is an equivalence of group-like $\einf$-spaces:
		\begin{equation*}
			\pic_{K(n)}(R)\simto \lim_{j}\pic_{K(n)}(R_j)\overset{\cong}{\longleftarrow} \lim_{j}\pic(R_j).
		\end{equation*} 
		This induces an isomorphism of Picard groups:
		\begin{equation*}
			f\colon \Pic_{K(n)}(R)\simto\lim_{j}\Pic_{K(n)}(R_j)\cong\lim_{j}\Pic(R_j).
		\end{equation*}
	\end{mainthm}
	\begin{proof}
		We have a sequence of equivalences of group-like $\einf$-spaces: 
		\begin{equation*}
			\pic_{K(n)}(R):=\pic\left(\Mod_{K(n)}(R)\right)\xrightarrow[{\ref{thm:R-mod_moore},\ref{prop:symm_equiv}}]{\sim}  \pic\left(\lim_{j} \Mod_{K(n)}(R_j)\right)\xrightarrow[\ref{prop:pic_lim_colim}]{\sim} \lim_{j}\pic_{K(n)}(R_j)\xleftarrow[\ref{lem:Rk-mod}]{\cong} \lim_{j}\pic(R_j).
		\end{equation*}
		It remains to compute the homotopy groups of  $\lim_{j}\pic(R_j)$. The Milnor sequence of the homotopy limit of Picard spaces 
		\begin{equation*}
			\begin{tikzcd}
				0\rar & \lim\limits_j\!^1 \pi_{1}\pic(R_j) \rar & \Pic_{K(n)}(R)\rar["f"] &  \lim\limits_j \Pic(R_j)\rar &0,
			\end{tikzcd}
		\end{equation*}
		implies that $f$ is surjective. To show it is injective, let $X$ be an element in $\ker f$. Then for any $j$, we have $X\Smash M_j\simeq X\Smash_{R} R_j
		\simeq R_j$. It follows from Hovey--Strickland's \eqref{eqn:Moore_limit} that $X\simeq\lim_j X\Smash M_j\simeq \lim_j R_j\simeq R$. This proves the injectivity of $f$. 
	\end{proof}
	\begin{rem}
		We note that statements in this subsection work for a $K(n)$-local $\Ebb_k$-ring $R$ for $k\ge 2$ as well. Recall from \Cref{prop:monoidal_mod_cat}, the module category $\Mod_{K(n)}(R)$ is $\Ebb_{k-1}$-monoidal. The limit of base change maps induces an equivalence of $\Ebb_{k-1}$-monoidal $\infty$-categories:
		\begin{equation*}
			\Mod_{K(n)}(R)\simto \lim_{j}\Mod(R_j).
		\end{equation*}
		On the level of Picard spaces, this gives rise to an equivalence of group-like $\Ebb_{k-1}$-spaces $\pic_{K(n)}(R)\simeq\lim_{j}\pic(R_j)$, which induces an isomorphism of groups $\Pic_{K(n)}(R)\cong\lim_{j}\Pic(R_j)$.
	\end{rem}
	\subsection{Properties of $K(n)$-local Picard groups}
	In \Cref{thm:Kn_Pic_limit}, we have lifted the Picard group functor to the pro-abelian groups indexed by a tower of generalized Moore spectra. 
	\begin{prop}\label{prop:Pic_top}
		Let $R$ be a $K(n)$-local $\einf$-ring spectrum. 
		\begin{enumerate}
			\item \textup{(\cite[Theorem 14.3.(d)]{Hovey-Strickland_1999})} The Picard group $\Pic_{K(n)}(R)$ is a topological abelian group, where a basis for closed neighborhoods of $R$ as the unit in the Picard group is given by
			\begin{equation*}
				V_j=\{X\in \Pic_{K(n)}(R)\mid X\Smash M_j\simeq R\Smash M_j\}.
			\end{equation*}
			\item Let $f\colon R_1\to R_2$ be a morphism in $\CAlg\left(\Sp_{K(n)}\right)$. Then its induced map on $K(n)$-local Picard groups is continuous with respect to the inverse limit  topology described above.
		\end{enumerate}
	\end{prop}	
	\begin{rem}
		When $R=S^0_{K(n)}$, this inverse limit  topology on  $\Pic_{K(n)}:=\Pic_{K(n)}\left(S^0_{K(n)}\right)$ has been described in \cite{Hovey-Strickland_1999} (see also \cite[\S 1.2]{Goerss-Hopkins}). 
	\end{rem}
	\begin{prop}\label{prop:choice_Mk}
		Let $R$ be a $K(n)$-local $\einf$-ring spectrum. The inverse limit  topology on $\Pic_{K(n)}(R)$ described in \Cref{prop:Pic_top} does not depend on the choice of the tower of generalized Moore spectra $\{M_j\}$ in \Cref{thm:genMoore}.
	\end{prop}
	\begin{proof}
		This follows from the fact that any two towers of generalized Moore spectra of type $n$ are equivalent in the homotopy category of $\Pro(\Sp)$ by \cite[Proposition 4.22]{Hovey-Strickland_1999} (see restatement in  \cite[Theorem 6.1]{Davis-Lawson_2014}).  
	\end{proof}
	One important property of Picard spaces and groups in $\CAlg$ is that they commute with finite products and filtered colimits of rings.  We next show that the $K(n)$-local Picard functors $\pic_{K(n)}\colon \CAlg\left(\Sp_{K(n)}\right)\to \Pro(\Top_*)$ and $\Pic_{K(n)}\colon \CAlg\left(\Sp_{K(n)}\right)\to \Pro(\Ab)$ preserve \emph{profinite} products (see definition below) and filtered colimits. This is a key step in our study of profinite descent spectral sequences for $K(n)$-local Picard groups in \Cref{sec:profin_desc}. 
	\begin{defn}\label{defn:profin_prod_Kn}
		For any cofiltered limit of finite sets $X=\lim_{\alpha} X_\alpha$ and $R\in \CAlg\left(\Sp_{K(n)}\right)$, define the $K(n)$-local \emph{profinite} product of $R$ indexed by $X$ as:
		\begin{equation*}
			\map_c(X,R):=L_{K(n)}\colim_\alpha \map(X_{\alpha},R).
		\end{equation*}
	\end{defn}
	By construction, $\map_c(X,R)$ is a $K(n)$-local $\einf$-ring spectrum. We also have equivalences:
	\begin{align}
		\map_c(X,R)&\simeq \lim_j L_n\colim_\alpha\map(X_\alpha,R)\Smash M_j&&\text{\eqref{eqn:Moore_limit}}\nonumber\\
		&\simeq \lim_j \colim_\alpha\map(X_\alpha,R)\Smash M_j&&L_n\text{ is smashing}\nonumber\\
		&\simeq \lim_j \colim_\alpha\map(X_\alpha,R\Smash M_j).&&\label{eqn:profin_prod}
	\end{align}
	The last equivalence justifies the notation $\map_c(X,R)$. 
	\begin{prop}\label{prop:Pic_Kn_profin_prod}
		There are natural equivalences of group-like $\einf$-spaces and isomorphisms of pro-abelian groups:
		\begin{align*}
			\pic_{K(n)}\map_c(X,R)&\simeq \map_c\left(X,\pic_{K(n)}(R)\right):=\lim_j\colim_\alpha\left(X_\alpha, \pic(R\Smash M_j)\right),\\
			\Pic_{K(n)}\map_c(X,R)&\cong \map_c\left(X,\Pic_{K(n)}(R)\right):=\lim_j\colim_\alpha\left(X_\alpha, \Pic(R\Smash M_j)\right).
		\end{align*}
	\end{prop}
	\begin{proof}
		The statement for $K(n)$-local Picard groups follows from \cite[Proposition 2.4.1]{MS_Picard} and \Cref{thm:Kn_Pic_limit}. 
		\begin{align*}
			\Pic_{K(n)}\map_c(X,R)&\simeq \Pic_{K(n)}\left[\lim_j\colim_\alpha \map(X_{\alpha},R\Smash M_j)\right] && \textup{\eqref{eqn:profin_prod}}\\
			&\simeq \lim_j\Pic\left[\colim_\alpha \map(X_{\alpha},R\Smash M_j)\right]&&\textup{\Cref{thm:Kn_Pic_limit}}\\
			&\simeq \lim_j\colim_\alpha \Pic\left[\map(X_{\alpha},R\Smash M_j)\right]&&\textup{\cite[Proposition 2.4.1]{MS_Picard}}\\
			& \simeq \lim_j\colim_\alpha \map(X_{\alpha},\Pic\left(R\Smash M_j\right))& &\Pic \textup{ preserves finite products}\\
			& =\map_c\left(\lim_\alpha X_\alpha, \lim_j \Pic\left(R\Smash M_j\right)\right)&&\\
			& \simeq \map_c\left(X, \Pic_{K(n)}(R)\right). && \textup{\Cref{thm:Kn_Pic_limit}}
		\end{align*}
		The proof for $\pic_{K(n)}$ is entirely parallel. 
	\end{proof}
	By \cite[Proposition 2.4.1]{MS_Picard}, the Picard group functor $\Pic\colon \CAlg(\Sp)\to \Ab $ preserves filtered colimits. The $K(n)$-local Picard groups satisfy a similar property, but only as \emph{pro}-abelian groups.  
	\begin{prop}\label{prop:Pic_Kn_colim}
		The $K(n)$-local Picard group functor $\Pic_{K(n)}\colon \CAlg\left(\Sp_{K(n)}\right)\to \Pro(\Ab)$ preserves filtered colimits. 
	\end{prop}
	\begin{proof}
		Let $L_{K(n)}\colim_k R_k$ be a filtered colimit of $K(n)$-local $\einf$-ring spectra $R_k$.  Similar to the proof of \Cref{prop:Pic_Kn_profin_prod} above, we have isomorphisms:
		\begin{align*}
			\Pic_{K(n)}\left(L_{K(n)}\colim_k R_k\right)\cong \Pic_{K(n)}\left(\lim_j\colim_k R_k\Smash M_j\right)\cong \lim_j\colim_{k} \Pic(R_k\Smash M_j).
		\end{align*}
		From the inverse limit  topology on $\Pic_{K(n)}(R_k)$ in \Cref{thm:Kn_Pic_limit},  we can see the right hand side of the isomorphism is  the colimit  of $\Pic_{K(n)}(R_k)$ in $\Pro(\Ab)$. 
	\end{proof}
	\begin{rem}\label{rem:Pic_Kn_colim}
		The \emph{discrete} $K(n)$-local Picard group functor $\Pic_{K(n)}\colon \CAlg\left(\Sp_{K(n)}\right)\to \Pro(\Ab)\to \Ab$ does not preserve filtered colimits, since the limit functor $\Pro(\Ab) \to \Ab$ does not. We will also see explicit counterexamples in our computations of $K(1)$-local Picard groups of Galois extensions of the $K(1)$-local spheres in \Cref{rem:Pic_colim_p_odd} and \Cref{rem:Pic_colim_p2}. 
	\end{rem}
	\section{Profinite descent for $K(n)$-local Picard groups}\label{sec:profin_desc}
	In this section, we study $K(n)$-local Picard groups in the context of profinite descent in $\Sp_{K(n)}$.  We begin by reviewing the theory of $K(n)$-local (pro)finite Galois extensions in \cite{galext}. Our main examples are maps between homotopy fixed points $E_n^{hG}$ of $E_n$ constructed by Devinatz--Hopkins in \cite{fixedpt}.  In  \Cref{subsec:descent}, we connect profinite Galois extensions with descent theory in \cite{Mathew_Galois,Mathew_desc_nilp}. In particular, we show in \Cref{prop:BhH_descent} that if $A\to B\to C$ is a composition of two profinite Galois extensions such that $A\to C$ is descendable, then so are $A\to B$ and $B\to C$. In particular, this implies that $E_n^{hG}\to E_n$ admits descent for any closed subgroup $G\le \Gbb_n$. 
	
	Building on the preparations above, we set up the profinite descent spectral sequence for $K(n)$-local Picard groups and identify its entire $E_1$ and $E_2$-pages in \Cref{thm:pic_ss}. A key step is to  commute the $K(n)$-local Picard groups with profinite products in $\CAlg\left(\Sp_{K(n)}\right)$ (\Cref{prop:Pic_Kn_profin_prod}),  which allows us to compute the $s=0$ line on the $E_1$-page of the spectral sequence.  Specializing to our main examples of descendable $K(n)$-local profinite Galois extensions, we obtain profinite descent spectral sequences between Picard groups of the homotopy fixed points $E_n^{hG}$ in \Cref{cor:desc_EnhG1G2}. This will be our main computational tool in \Cref{sec:Pic_Mack}. By analyzing this spectral sequence, we obtain some algebraicity result for $\Pic_{K(n)}\left(E_n^{hG}\right)$ in \Cref{thm:descent_Pic_EnhG}. 
	\subsection{Review on $K(n)$-local profinite Galois extensions}\label{subsec:profin_gal}
	First, we recall the descent spectral sequences for \emph{finite} $G$-Galois extensions of $\einf$-ring spectra. 
	\begin{defn}[Rognes, {\cite[Definitions 4.1.3]{galext}}]\label{defn:fin_gal_extn}
		Let $G$ be a finite group. A map of ring spectra $A\to B$ is called $G$-Galois if there is an $A$-linear $G$-action on $B$ such that:
		\begin{equation*}
			A\simeq B^{hG},\qquad B\Smash_A B\simeq \map(G,B).
		\end{equation*}
		The extension is called faithful if the base change functor $B\Smash_A -\colon \Mod(A)\to \Mod(B)$ is conservative, i.e. $N\in \Mod(A)$ is zero iff $B\Smash_A N\simeq *$. 
	\end{defn}
	When $A\to B$ is a finite Galois extension, there is a \textbf{homotopy fixed point spectral sequence} (HFPSS) to compute $\pi_*(A)$ from $\pi_*(B)$ together with the $G$-actions on it. From the  definition of homotopy fixed points, we have
	\begin{equation}\label{eqn:HFP_tot}
		A\simeq B^{hG}:=\map(EG, B)^G\simeq \map(|EG_\bullet|, B)^G\simeq \Tot\left[\map(G^{\times \bullet+1}, B)^G\right]\simeq \Tot\left[\map(G^{\times \bullet}, B)\right].
	\end{equation}
	Associated to this totalization is a \textbf{Bousfield--Kan spectral sequence} (BKSS) whose $E_1$-page is
	\begin{equation}\label{eqn:HFPSS_fin_E1}
		\!^{\HFP} E_1^{s,t}=\pi_t\map(G^{\times s}, B)\cong \map(G^{\times \bullet}, \pi_t(B)) \Longrightarrow \pi_{t-s}\left(B^{hG}\right)\cong \pi_{t-s}(A).
	\end{equation}
	One can further check the $d_1$-differentials in this spectral sequence are the same as the cobar differentials to compute the group cohomology of $G$. This identifies the $E_2$-page of the spectral sequence as:
	\begin{equation}\label{eqn:HFPSS_fin_E2}
		\!^{\HFP}E_2^{s,t}=H^s(G;\pi_t(B))\Longrightarrow \pi_{t-s}(A). 
	\end{equation}
	In \cite{MS_Picard}, Mathew--Stojanoska studied Picard groups of finite Galois extensions of ring spectra. 
	\begin{thm}[{\cite[3153]{MS_Picard}}]
		Let $G$ be a finite group.  If $A\to B$ is a faithful $G$-Galois extension of $\einf$-ring spectra, then there is a natural equivalence of symmetric monoidal $\infty$-categories:
		\begin{equation*}
			\Mod(A)\simeq \left(\Mod(B)\right)^{hG}:= \Tot\left[\map(G^{\times \bullet},\Mod(B))\right]. 
		\end{equation*}
		This induces an equivalence of Picard spaces:
		\begin{equation*}
			\pic(A)\simeq \tau_{\ge 0} \Tot\left[\map(G^{\times \bullet},\pic(B))\right]. 
		\end{equation*}
		Similar to \eqref{eqn:HFPSS_fin_E1} and \eqref{eqn:HFPSS_fin_E2}, we obtain a homotopy fixed point spectral sequence:
		\begin{equation}\label{eqn:Picard_desc_fin}
			\!^{\pic}E_1^{s,t}=\map(G^{\times s},\pi_t\pic(B)),\qquad \!^{\pic}E_2^{s,t}=H^s(G;\pi_t\pic(B))\Longrightarrow \pi_{t-s}\pic(A). 
		\end{equation}
	\end{thm}
	By \Cref{cor:pic_E2_ring}, the two spectral sequences \eqref{eqn:HFPSS_fin_E1} and \eqref{eqn:Picard_desc_fin} have very similar $E_1$ and $E_2$-pages. It is a natural question to compare their differentials. Mathew--Stojanoska showed:
	\begin{thm}[{\cite[Comparison Tool 5.2.4]{MS_Picard}}]\label{thm:compare_diff}
		Denote by $\!^{\HFP}d_r^{s,t}$ and $\!^{\pic}d_r$ the differentials in the two spectral sequences \eqref{eqn:HFPSS_fin_E1} and \eqref{eqn:Picard_desc_fin}, respectively. When $t-s>0$ and $s\ge 0$, or $2\le r\le t-1$, we have $^{\pic}d_r^{s,t-1}=\!^{\HFP}d_r^{s,t}$. 
	\end{thm} 
	We make some observations on the cosimplicial spectra \eqref{eqn:HFP_tot}. The second assumption $B\Smash_A B\simeq \map(G,B)$ in \Cref{defn:fin_gal_extn} implies that for any $s\ge 0$, we have equivalences of ring spectra:
	\begin{equation*}
		\bigsmash\!_A^{s+1} B\simeq \bigsmash\!_B^s(B\Smash_A B)\simeq \bigsmash\!_B^s\map(G,B) \simeq \map(G^{\times s},B).
	\end{equation*}
	As a result, we have a level-wise equivalence of cosimplicial ring spectra $\left[\map(G^{\times \bullet}, B)\right]\simeq \left[\bigsmash\!_A^{\bullet+1}B\right]$. The right hand side of this equivalence is the cobar complex for the Galois extension $A\to B$. The first assumption $A\simeq B^{hG}$ in \Cref{defn:fin_gal_extn} then translates to an equivalence:
	\begin{equation*}
		A\simeq \Tot\left[\bigsmash\!_A^{\bullet+1}B\right]. 
	\end{equation*}
	
	We are now ready to study \emph{profinite} Galois extensions of $K(n)$-local $\einf$-ring spectra. 
	\begin{defn}[Rognes, {\cite[Definitions 8.1.1]{galext}}]\label{defn:profin_gal_ext}
		Let $G=\lim_{k}G/U_k$ be a cofiltered limit of finite groups. Consider a directed system of maps of $K(n)$-local $\einf$-ring spectra $A\to B_k$. Suppose each $A\to B_k$ is a faithful $G/U_k$-Galois extension in the sense of \Cref{defn:fin_gal_extn} such that there are natural equivalences $B_k\simeq B_j^{h(U_k/U_j)}$ when $U_j$ is an open normal subgroup of $U_k$. Then $A\to B:=L_{K(n)}\colim_k B_k$ is a $K(n)$-local \emph{pro}-$G$-Galois extension of $A$. 
	\end{defn}
	\begin{rem}
		A special case of profinite Galois extensions is when $G$ is \emph{countably based} (second countable). By \cite[Corollary 1.1.13]{Ribes-Zalesskii}, $G$ is countably based iff there is a sequence of open normal subgroups $U_1\supseteq U_2\supseteq \cdots$ of $G$ such that $G\cong \lim_k G/U_k$. For such a profinite group $G$, a profinite $G$ Galois extension is a sequence of extensions $A\to B_1\to B_2\to \cdots \to B$ such that $A\to B_k$ is a $G/U_k$-extension and $B\cong \colim_k B_k$. By \cite[Example 1.10]{davis2023homotopy}, any compact analytic $p$-adic group is countably based. This includes the Morava stabilizer group $\Gbb_n$ and its closed subgroups.
	\end{rem}
	Recall $\hsmash$ is the monoidal product in $\Sp_{K(n)}$.  In \cite{fixedpt}, Devinatz--Hopkins constructed "continuous homotopy fixed points" $E_n^{hG}$ for all closed subgroups $G\le \Gbb_n$.  This is our main example of $K(n)$-local profinite $G$-Galois extensions. 
	\begin{const}[Devinatz--Hopkins]\label{const:EnhG}
		For an \emph{open} subgroup $G\le \Gbb_n$, $E_n^{hG}$ is defined as 
		\begin{equation*}
			E_n^{hG}:= \Tot\left[ \map\left(\Gbb_n/G, E_n^{\hsmash\bullet+1}\right)\right].
		\end{equation*}
		In particular, $E_n^{hG}$ satisfies 
		\begin{equation*}
			E_n\hsmash E_n^{hG}\simeq \map(\Gbb_n/G,E_n). 
		\end{equation*}
		For a \emph{closed} subgroup $G$ of $\Gbb_n$, the continuous homotopy fixed points is defined as \cite[Definition 1.5]{fixedpt}:
		\begin{equation*}
			E_n^{hG}:=L_{K(n)}\colim_{k} E_{n}^{h(U_kG)},
		\end{equation*}
		where $\Gbb_n=U_0\supsetneq U_1 \supsetneq  U_2  \supsetneq \cdots$ is a decreasing sequence of open normal subgroups of $\Gbb_n$ such that $\bigcap_k U_k=\{e\}$.
	\end{const}
	
	\begin{thm}[Devinatz--Hopkins, {\cite{fixedpt}}]\label{thm:DH_fixedpt}
		The spectra $E_n^{hG}$ constructed above satisfy:
		\begin{enumerate}
			\item $E_n^{hG}$ is a $K(n)$-local $\einf$-ring spectrum. In particular, there is an equivalence $E_n^{h\Gbb_n}\simeq S^0_{K(n)}$.
			\item When $G\le \Gbb_n$ is a finite subgroup, there is an equivalence:
			\begin{equation*}
				E_n^{hG}\simto E_n^{h'G}:=\map\left(EG_+, E_n\right)^G,
			\end{equation*}
			where the right hand side is the $G$-homotopy fixed point spectrum for finite group actions. 
			\item Suppose $G\le \Gbb_n$ is a closed subgroup, $U\trianglelefteq G$ is a closed normal subgroup such that $G/U$ is finite. Then $E_n^{hG}\simeq \left(E_n^{hU}\right)^{h(G/U)}$, where the outer homotopy fixed point is the categorical homotopy fixed point for finite group actions. 
			\item There is a \emph{profinite} homotopy fixed point spectral sequence (HFPSS) with $E_2$-page:
			\begin{equation*}
				^\HFP\! E^{s,t}_2=H_c^s(G;\pi_t(E_n))\Longrightarrow \pi_{t-s}(E_n^{hG}). 
			\end{equation*} 
		\end{enumerate} 
	\end{thm}
	\begin{thm}[Rognes, {\cite[Theorem 5.4.4]{galext}}] \label{thm:Rognes_Gal}For inclusions $U\trianglelefteq G\le \Gbb_n$ of closed groups such that $U\trianglelefteq G$ is open (finite index) and normal, the extension $E_n^{hG}\simeq \left(E_n^{hU}\right)^{h(G/U)}\to E_n^{hU}$ is a faithful $K(n)$-local $G/U$-Galois extension. In particular:
		\begin{itemize}
			\item  For each finite subgroup $F$ of $\Gbb_n$, the map  $E_n^{hF}\to E_n$ is a faithful $F$-Galois extension.
			\item For each open normal subgroup $U$ of $\Gbb_n$, the map $S^0_{K(n)}\simeq E_n^{h\Gbb_n}\to E_n^{hU}$ is a faithful $(\Gbb_n/U)$-Galois extension.
		\end{itemize}
	\end{thm}
	\begin{cor}\label{cor:E_nhG_Gal}
		For any \emph{closed} subgroup $G\le \Gbb_n$, the extension $E_n^{hG}\to E_n$ is a $K(n)$-local profinite $G$-Galois extension. When $G\trianglelefteq \Gbb_n$ is a \emph{closed} normal subgroup, the extension $S_{K(n)}^0\simeq E_n^{h\Gbb_n}\to E_n^{hG}$ is a $K(n)$-local profinite $\Gbb_n/G$-Galois extension. 
	\end{cor}

	\begin{exmp}
		Another example of profinite Galois extensions of ring spectra is $K(1)$-local algebraic $K$-theory of profinite Galois extensions of discrete rings. In \cite{Thomason_alg_k_et}, Thomason showed $K(1)$-local algebraic $K$-theory satisfies \'etale descent. This means if $R_1\to R_2$ is a $G$-Galois extension of discrete rings for a finite group $G$, then $L_{K(1)}(R_1)\to L_{K(1)}K(R_2)$ is a $G$-Galois extension of $K(1)$-local $\einf$-ring spectra. 
		
		Suppose if $R_1\to R_2$ is a $G$-Galois extension for a profinite group $G\cong \lim_k G/U_k$ such that $R_2\cong \colim_k R_1^{U_k}$.  By \cite[Theorem 1.5]{Clausen-Mathew_hyperdescent}, \'etale $K$-theory commutes with filtered colimits of rings. It follows that $L_{K(1)}(R_1)\to L_{K(1)}(R_2)$ is a $K(1)$-local  profinite $G$-Galois extension. 
		
		When a prime $p$ is invertible in a ring $R$, the extension $R\to R[\zeta_{p^{\infty}}]:=R\otimes \Z[\zeta_{p^{\infty}}]$ is a $\Zpx$-Galois extension of rings. By \cite[Theorem 1.4]{BCM_rmk_k1_alg_k}, there is a $\Zpx$-equivariant equivalence of $K(1)$-local $\einf$-ring spectra $L_{K(1)}K(R[\zeta_{p^{\infty}}])\simeq KU^\wedge_p\hsmash L_{K(1)}K(R)$ over $L_{K(1)} K(R)$. As a result, the profinite Galois extension $L_{K(1)}K(R)\to L_{K(1)}K(R[\zeta_{p^{\infty}}])$ is equivalent to the base change of the $K(1)$-local $\Zpx$-Galois extension $S_{K(1)}^0\to E_1\simeq KU^\wedge_p$ in \Cref{cor:E_nhG_Gal} by $L_{K(1)}K(R)$. The assumption $p\in R^\times$ can be dropped since $L_{K(1)}K(R)\simeq L_{K(1)}K(R[1/p])$ by \cite[Theorem 1.1]{BCM_rmk_k1_alg_k}.
	\end{exmp}
	\begin{exmp}
		Let $T(n)=v_n^{-1}F_n$ be the $v_n$-mapping telescope of a finite complex $F_n$ of type $n$. In \cite[Theorem A]{carmeli2021chromatic}, Carmeli--Schlank--Yanovski showed that every finite abelian Galois extension of $S^0_{K(n)}$ in $\Sp_{K(n)}$ lifts to the $T(n)$-local category $\Sp_{T(n)}$. This gives a family of examples of $T(n)$-local profinite Galois extensions. 
	\end{exmp}
	\subsection{Descent and Picard groups}\label{subsec:descent}
	The foundation of our construction of profinite descent spectral sequences is the descent theory for ring spectra in \cite{Mathew_Galois}. 
	\begin{defn}[{\cite[Definitions 3.17,3.18]{Mathew_Galois}}]\label{defn:descent}
		Let $\Ccal$ be a stable symmetric monoidal $\infty$-category. A full subcategory $\Dcal\subseteq \Ccal$ is thick if it is closed under finite limits and colimits, and under retracts. It is called a $\otimes$-ideal  if for any $d\in \Dcal$ and $c\in \Ccal$, we have $d\otimes c\in \Dcal$. 
		
		A commutative algebra $A\in \CAlg(\Ccal)$ is said to be \textbf{descendable} (or admits descent) if the thick $\otimes$-ideal generated by $A$ is all of $\Ccal$.  We say a ring map $A\to B$ is descendable if $B$ is descendable as an object in $\CAlg_A(\Ccal)$. 
	\end{defn}
	\begin{prop}[{\cite[Propositions 3.20, 3.22]{Mathew_Galois}}]\label{prop:descent_criterion}
		Let $A\in \CAlg(\Ccal)$. If $A$ admits descent, then there is an equivalence:
		\begin{equation*}
			1_\Ccal\simto \Tot\left[A^{\otimes \bullet+1}\right]. 
		\end{equation*}
		Moreover, the equivalence above lifts to module categories:
		\begin{equation*}
			\Ccal=\Mod(1_\Ccal)\simto \Tot\left[\Mod\left({A^{\otimes \bullet+1}}\right)\right].
		\end{equation*}
	\end{prop}
	\begin{rem}
		The equivalence $1_\Ccal\simto \Tot\left[A^{\otimes \bullet+1}\right]$ does \emph{not} imply that $A$ admits descent. Instead, $A$ is descendable if the totalization on the right hand side defines a constant object in $\Pro(\Ccal)$. Equivalently, this happens iff $1_\Ccal$ is equivalent to a retract of a partial totalization $\Tot_m\left[A^{\otimes \bullet+1}\right]$ for some natural number $m$. 
	\end{rem}	
	Galois extensions are related to descent theory by:
	\begin{prop}[{\cite[Proposition 3.21]{Mathew_desc_nilp}}]
		Let $G$ be a finite group. A faithful $G$-Galois extension of ring spectra $A\to B$ admits descent. 
	\end{prop}
	Following the Devinatz--Hopkins' \Cref{const:EnhG} of $E_n^{hG}$ in \cite{fixedpt}, we define:
	\begin{const}\label{const:B_hH}
		Let $A\to B$ be a descendable $K(n)$-local profinite $G$-Galois extension, where $G=\lim_k G/U_k$ is a cofiltered limit of its quotients by a directed system of open normal subgroups $\{U_k\}$ of $G$ such that $\bigcap U_k=\{e\}$. 
		\begin{itemize}
			\item For an open subgroup $U\le G$, set \begin{equation*}
				B^{hU}:=\Tot\left[\map\left(G/U, \bighsmash_A^{\bullet+1}B\right)\right].
			\end{equation*}
			\item For an closed subgroup $H\le G$, set \begin{equation*}
				B^{hH}:=L_{K(n)}\colim_k B^{h(HU_k)}. 
			\end{equation*}
		\end{itemize} 
	\end{const}
	\begin{prop}\label{prop:BhH_descent}
		Let $A\to B$ be a descendable $K(n)$-local profinite $G$-Galois extension. Then for any closed subgroup $H\le G$, the map $B^{hH}\to B$ is a descendable $K(n)$-local profinite $H$-Galois extension. If $H$ is a closed normal subgroup, then $A\to B^{hH}$ is a descendable $K(n)$-local profinite $G/H$-Galois extension.
	\end{prop}
	\begin{rem}
		If $A\to B\to C$ admits descent, it is not always true that $B\to C$ admits descent. One counterexample is $R\hookrightarrow R[x]\twoheadrightarrow R[x]/(x)\cong R$ for any ring $R$.
	\end{rem}
	\begin{rem}
		In \cite{galext}, Rognes also defined profinite Galois extensions in $\Sp_L$, the Bousfield localization of $\Sp$ at a spectrum $L$. We note that all the definitions and constructions above go through $\Sp_L$. In particular, \Cref{prop:BhH_descent} holds for profinite Galois extensions in $\Sp_L$. 
	\end{rem}
	\begin{proof}
		Similar to \Cref{cor:E_nhG_Gal}, we can show that:
		\begin{itemize}
			\item  When $H\le G$ is closed, $B^{hH}\to B$ is a $K(n)$-local profinite $H$-Galois extension.
			\item When $H\trianglelefteq G$ is a closed normal subgroup, $A\to B^{hH} $ is a $K(n)$-local profinite $G/H$-Galois extension.
		\end{itemize}
		It remains to prove both maps are descendable. As the composition $A\to B^{hH}\to B$ is descendable by assumption, the first map $A\to B^{hH}$ is descendable by \cite[Proposition 3.24]{Mathew_Galois}. To show $F\colon B^{hH}\to B$ is also descendable, consider the pushout diagram of ring spectra:
		\begin{equation*}
			\begin{tikzcd}
				B^{hH}\rar["f"]\dar & B\dar \\
				B^{hH}\hsmash_A B\rar["f\hsmash 1"'] & \ar[from=ul, phantom,"\ulcorner" very near end] B\hsmash_AB.
			\end{tikzcd}
		\end{equation*}
		The left vertical map $B^{hH}\to B^{hH}\hsmash_A B$ is a base change of a descendable map $A\to B$ and hence admits descent by \cite[Corollary 3.21]{Mathew_Galois}. We now show that $B^{hH}\hsmash_A B$ is a retract of $B\hsmash_A B$, which implies that $f\hsmash 1\colon B^{hH}\hsmash_A B\to B\hsmash_A B$ admits descent by \Cref{defn:descent}. Then $f$ itself admits descent again by  \cite[Proposition 3.24]{Mathew_Galois}. 
		
		Similar to the computation in \eqref{eqn:B_otimes_s+1}, we can identify $f\hsmash 1$ with the natural  map $\map_c(G/H,B)\to \map_c(G,B)$ induced by the quotient $q\colon G\to G/H$.  By  \cite[Proposition 2.2.2, Exercise 2.2.3]{Ribes-Zalesskii}, there is a continuous section $s\colon G/H\to G$ of $q$ such that $q\circ s=\id$. This implies that the composition:
		\begin{equation*}
			\map_c(G/H,B)\xrightarrow{q^*} \map_c(G,B)\xrightarrow{s^*} \map_c(G/H,B)
		\end{equation*}
		is the identity. Hence $B^{hH}\hsmash_A B\simeq \map_c(G/H,B)$ is a retract of $B\hsmash_A B\simeq \map_c(G,B)$.
	\end{proof}
	Having set up the general theory, we apply it to study descent between continuous homotopy fixed points of $E_n$. 
	\begin{thm}[{\cite[Proposition 10.10]{Mathew_Galois}}]
		The $K(n)$-local unit map $L_{K(n)}S^0\to E_n$ admits descent. 
	\end{thm}
	\Cref{prop:BhH_descent} then implies: 
	\begin{cor}\label{cor:En_hG_descendable}
		Let $G_1\le G_2\le \Gbb_n$ be two closed subgroups such that $G_1$ is normal. Then $E_n^{hG_2}\to E_n^{hG_1}$ is a descendable $K(n)$-local profinite $G_2/G_1$-Galois extension. In particular, 
		\begin{enumerate}
			\item For any closed subgroup $G\le \Gbb_n$, $E_n^{hG}\to E_n$ is a descendable $K(n)$-local profinite $G$-Galois extension. 
			\item For any closed normal subgroup $G\trianglelefteq \Gbb_n$, $S^0_{K(n)}\to E_n^{hG}$ is a descendable $K(n)$-local profinite $\Gbb_n/G$-Galois extension. 
		\end{enumerate}
	\end{cor}
	\subsection{Profinite descent spectral sequences}\label{subsec:desc_ss}
	Our goal in this subsection is to extend the two spectral sequences \eqref{eqn:HFPSS_fin_E1} and \eqref{eqn:Picard_desc_fin} to $K(n)$-local profinite Galois extensions. Before stating the spectral sequences, we need one final technical condition: 
	\begin{cond}\label{cond:ML}
		Fix a tower of generalized Moore spectra $\{M_j\}$ of type $n$ as in \Cref{thm:genMoore} and \Cref{prop:genMoore_mult}. We say a $K(n)$-local spectrum $X$ satisfies ML if the inverse system $\{\pi_t(X\Smash M_j)\}_{j\ge1}$ satisfies the Mittag-Leffler condition for any $t\in \Z$.
	\end{cond}
	\begin{rem}
		For a $K(n)$-local spectrum $X$, we have $X\simeq \lim_j X\Smash M_j$ by  \eqref{eqn:Moore_limit}. Then the Milnor sequence yields short exact sequences of homotopy groups:
		\begin{equation*}		
			0\to \lim_j\!^1 \pi_{t+1}(X\Smash M_j)\longrightarrow \pi_t(X)\longrightarrow \lim_j\pi_{t}(X\Smash M_j)\to0.
		\end{equation*}
		This means the homotopy groups of $X$ are \emph{$L$-complete} (see \cite[Definition A.5]{Hovey-Strickland_1999}), but not necessarily \emph{complete}. The ML \Cref{cond:ML} implies the vanishing of the $\lim^1$-terms in the Milnor sequence. As a result, if $X\in \Sp_{K(n)}$ satisfies the ML condition, then $\pi_t(X)\cong \lim_j \pi_t(X\Smash M_j)$ has a natural  inverse limit  topology as a pro-abelian group.
	\end{rem}
	\begin{mainthm}\label{thm:pic_ss}
		Let $A\to B$ be a descendable $K(n)$-local profinite $G$-Galois extension of $\einf$-ring spectra and $B$ satisfies the ML \Cref{cond:ML}. Then there are profinite descent spectral sequences:
		\begin{align*}
			\!^{\HFP} E_1^{s,t}&=\map_c(G^{\times s}, \pi_t(B)),& \!^{\HFP} E_2^{s,t}&=H_c^s(G; \pi_t(B)) &&\Longrightarrow \pi_{t-s}(A);\\
			\!^{\pic} E_1^{s,t}&=\map_c(G^{\times s}, \pi_t\pic_{K(n)}(B)),&\!^{\pic} E_2^{s,t}&=H_c^s(G; \pi_t\pic_{K(n)}(B))&& \Longrightarrow \pi_{t-s}\pic_{K(n)}(A), \qquad t-s\ge0. 
		\end{align*}
		In particular, the second spectral sequence abuts to $\Pic_{K(n)}\left(A\right)$ when $t = s$. 
		The differentials in both spectral sequences are the form $d_r^{s,t}\colon E_r^{s,t} \to  E_r^{s+r,t+r-1}$. When either $t-s>0$ and $s>0$, or $2\le r\le t-1$, we have $^{\pic}d_r^{s,t}=\!^{\HFP}d_r^{s,t-1}$.
	\end{mainthm}
	\begin{proof}
		By \Cref{prop:descent_criterion}, the descendability assumption on the Galois extension $A\to B$ implies that there are equivalences:
		\begin{align}
			\label{eqn:profin_tot}A&\simto \Tot\left[\bighsmash_A^{\bullet+1}B\right],\\
			\label{eqn:profin_mod_tot}\Mod_{K(n)}A&\simto \Tot\left[\Mod_{K(n)}\left(\bighsmash_A^{\bullet+1}B\right)\right].
		\end{align}
		We first need to identify the ring spectra $B^{\hsmash_A s +1}$ for $s\ge 0$. Recall that $B\simeq L_{K(n)}\colim_kB_k$, where each $B_k$ is a finite $G/U_k$-Galois extension of $A$.  Similar to \cite[Equation 8.1.2]{galext}, we have equivalences:
		\begin{align}
			\bighsmash_A^{s+1} B&\simeq L_{K(n)}\colim_k B\hsmash_A \left(\bighsmash_A^sB_k\right)&&\nonumber\\
			& \simeq L_{K(n)}\colim_k \bighsmash_B^s \left(B\hsmash_A B_k\right)&&\nonumber\\
			& \simeq L_{K(n)}\colim_k \bighsmash_B^s\left(B\hsmash_{B_k} B_k\hsmash_A B_k\right)&&\nonumber\\
			&\simeq  L_{K(n)}\colim_k\bighsmash_B^s\left[B \hsmash_{B_k}\map(G/U_k, B_k) \right]&&\textup{\Cref{defn:fin_gal_extn}}\nonumber\\
			&\simeq  L_{K(n)}\colim_k\bighsmash_B^s\left[\map(G/U_k, B)\right]&&\nonumber\\
			&\simeq L_{K(n)}\colim_k\map((G/U_k)^{\times s}, B)&&\nonumber\\
			\label{eqn:B_otimes_s+1}&= \map_c(G^{\times s},B)&&\textup{\Cref{defn:profin_prod_Kn}}\\
			&\simeq \lim_j\colim_k\map((G/U_k)^{\times s}, B\Smash M_j).&&\nonumber
		\end{align}
		It follows that $E_1$-page of the BKSS associated to the cosimplicial spectrum \eqref{eqn:profin_tot} has the form:
		\begin{equation*}
			^\HFP E_1^{s,t}=\pi_t\left(\bighsmash_A^{s+1} B\right)\cong \pi_t\left(\lim_j\colim_k\map((G/U_k)^{\times s}, B\Smash M_j)\right).
		\end{equation*}
		The Mittag-Leffler assumption  on $\{\pi_t(B\Smash M_j)\}$ implies that the inverse system
		\begin{align*}
			\left\{\pi_t\left[\colim_k\map((G/U_k)^{\times s}, B\Smash M_j)\right]\right\}&\cong \left\{\colim_k\map\left((G/U_k)^{\times s}, \pi_t(B\Smash M_j)\right)\right\}\\
			& \cong \left\{\left[\colim_k\map\left((G/U_k)^{\times s}, \Z\right)\right]\otimes_\Z \pi_t(B\Smash M_j)\right\}
		\end{align*}
		also satisfies Mittag-Leffler since the condition is preserved under base change by \cite{Emmanouil_ML_lim1}. It follows that
		\begin{align}
			\nonumber ^\HFP E_1^{s,t}&\cong \pi_t\left[\lim_j\colim_k\map\left((G/U_k)^{\times s}, B\Smash M_j\right)\right]\\
			\nonumber&\cong \lim_j\pi_t\left[\colim_k\map\left((G/U_k)^{\times s}, B\Smash M_j\right)\right]\\
			\nonumber&\cong \lim_j\colim_k\map\left((G/U_k)^{\times s}, \pi_t(B\Smash M_j)\right)\\
			& =\map_c\left(G^{\times s},\pi_t(B)\right). \label{eqn:HFPSS_profin_E1}
		\end{align}
		This identifies the $E_1$-page of the first spectral sequence. For $K(n)$-local Picard groups, applying the functor $\pic$ to the equivalence \eqref{eqn:profin_mod_tot}, we have:
		\begin{align*}
			\pic_{K(n)}(A)&\simeq \tau_{\ge0}\Tot\left[\pic_{K(n)}\left(\bighsmash_A^{\bullet+1}B\right)\right]&& \textup{\Cref{prop:pic_lim_colim} (\cite[Proposition 2.2.3]{MS_Picard}) }\\
			& \simeq \tau_{\ge0}\Tot\left[\pic_{K(n)}\map_c\left(G^{\times \bullet},B\right)\right]&& \textup{\eqref{eqn:B_otimes_s+1}}\\
			&\simeq \tau_{\ge0}\Tot\left[\map_c\left(G^{\times \bullet},\pic_{K(n)}B\right)\right]. && \textup{\Cref{prop:Pic_Kn_profin_prod}}
		\end{align*}
		The $E_1$-page of this BKSS is:
		\begin{equation*}
			^\pic E_1^{s,t}=\pi_t\pic_{K(n)}\map_c\left(G^{\times s},B\right)=\begin{cases}
				\Pic_{K(n)}\map_c\left(G^{\times s},B\right),& t=0;\\
				\left[\pi_0\map_c\left(G^{\times s},B\right)\right]^\times, &t=1;\\
				\pi_{t-1}\map_c\left(G^{\times s},B\right), &t\ge 2.
			\end{cases}
		\end{equation*}
		When $t\ge 2$, this is the same as in \eqref{eqn:HFPSS_profin_E1}:
		\begin{equation*}
			^\pic E_1^{s,t}=\pi_{t-1}\map_c\left(G^{\times s},B\right)\cong \map_c\left(G^{\times s},\pi_{t-1}(B)\right).
		\end{equation*}
		When $t=1$, again by  \eqref{eqn:HFPSS_profin_E1}, we have an isomorphism of topological rings: 
		\begin{equation*}
			\pi_{0}\map_c\left(G^{\times s},B\right)\cong \map_c\left(G^{\times s},\pi_{0}(B)\right).
		\end{equation*}
		The unit of the continuous mapping ring  sits in a pullback diagram of topological spaces:
		\begin{equation*}
			\begin{tikzcd}
				\map_c\left(G^{\times s}, \pi_0(B)\right)^\times\rar \dar\ar[dr,phantom,very near start,"\lrcorner"]&\dar \map\left(G^{\times s}, \pi_0(B)\right)^\times\dar\rar["\cong"]& \map\left(G^{\times s}, \pi_0(B)^\times\right)\\
				\map_c\left(G^{\times s}, \pi_0(B)\right)\rar & \map\left(G^{\times s}, \pi_0(B)\right)&
			\end{tikzcd}
		\end{equation*}
		Notice that $(-)^\times \cong \hom_{\CAlg_\Z}(\Z[x^{\pm 1}],-)$ is a right adjoint. This implies that taking units commutes with arbitrary products of commutative rings.  
		Furthermore, the set of units in a topological commutative ring is a topological abelian group with subspace topology. This implies that a unit in $\map_c\left(G^{\times s},\pi_{0}(B)\right)$ is a continuous map $G^{\times s}\to \pi_0(B)$ whose image is in $\pi_0(B)^\times$.  As a result, the $t=1$ line on the $E_1$-page of the BKSS is:
		\begin{equation*}
			^\pic E_1^{s,1}=\map_c\left(G^{\times s}, \pi_0(B)\right)^\times\cong  \map_c\left(G^{\times s}, \pi_0(B)^\times\right).
		\end{equation*}
		When $t=0$, the computation follows from \Cref{prop:Pic_Kn_profin_prod} and \eqref{eqn:B_otimes_s+1}:
		\begin{equation*}
			^\pic E_1^{s,0}=\Pic_{K(n)}\map_c\left(G^{\times s},B\right)\cong \map_c\left(G^{\times s},\Pic_{K(n)}(B)\right).
		\end{equation*}
		So far, we have computed the $E_1$-pages of the BKSS's. They are isomorphic to the cobar resolutions for $\pi_t(B)$ and $\pi_t\left(\pic_{K(n)}(B)\right)$ as continuous $G$-modules, respectively. We can further prove that the differentials also agree with the cobar differentials to compute continuous group cohomology. This implies that the $E_2$-pages of the two spectral sequences are indeed as claimed:
		\begin{align*}
			^{\HFP} E_2^{s,t}=H_c^s(G; \pi_t(B))& \Longrightarrow \pi_{t-s}(A);\\
			^{\pic} E_2^{s,t}=H_c^s\left(G; \pi_t\pic_{K(n)}(B)\right) &\Longrightarrow \pi_{t-s}\pic_{K(n)}(A),\qquad t-s\ge 0. 
		\end{align*}
		At last, we need to compare differentials in the two BKSS's.  By \Cref{prop:pic_kn_truncation}, we have equivalences
		\begin{equation*}
			\tau_{\ge1}\pic_{K(n)}\left(\bighsmash_A^{s+1} B\right)\simeq \tau_{\ge1}\pic\left(\bighsmash_A^{s+1} B\right)\simeq \Sigma \gl_1\left(\bighsmash_A^{s+1} B\right).
		\end{equation*}
		The comparison of differentials then follows exactly the same way as in \Cref{thm:compare_diff} (\cite[Comparison Tool 5.2.4]{MS_Picard}).
	\end{proof}
	\begin{rem}
		The purpose of the ML \Cref{cond:ML} is  to identify the $E_1$ and $E_2$-pages of the spectral sequences. In the proof above, it is \emph{not} enough to just assume $\pi_t(B)\cong \lim_j\pi_t(B\Smash M_j)$, or equivalently $\lim_j^1\pi_{t+1}(B\Smash M_j)=0$. This is because vanishing of $\lim^1$ is \emph{not} preserved under base change. As is explained in \cite{Emmanouil_ML_lim1}, "the Mittag-Leffler condition is equivalent to the universal vanishing of $\varprojlim^1$" under base change.
	\end{rem}		
	Following \Cref{cor:En_hG_descendable},  we now apply \Cref{thm:pic_ss} to study descent spectral sequences between continuous homotopy fixed points of $E_n$ and their Picard groups. From the construction in \Cref{prop:genMoore_mult}, we have $\pi_*(E_n\Smash M_j)\cong \pi_*(E_n)/J_j$ for some decreasing sequence of  invariant open ideals $J_1\supsetneq J_2\supsetneq J_3\supsetneq\cdots$ of $\pi_0(E_n)$.  Hence $\pi_t(E_n)/J_j$ are all finite groups and $E_n$ satisfies the ML \Cref{cond:ML}. At this point, we have verified that the extension $E_n^{hG}\to E_n$ satisfies the two assumptions in \Cref{thm:pic_ss}.   From the associated cobar complex:
	\begin{equation}\label{eqn:Tot_EnhG}
		\Tot\left[ \bighsmash_{E_n^{hG}}^{\bullet+1}E_n\right]\simeq \Tot\left[ \map_c(G^{\times \bullet}, E_n)\right],
	\end{equation}
	we recover the continuous homotopy fixed point spectral sequence in \cite{fixedpt}:
	\begin{equation}\label{eqn:HFPSS_EnhG}
		^{\HFP}E_2^{s,t}\left(E_n^{hG}\right)= H_c^s(G;\pi_t(E_n))\Longrightarrow\pi_{t-s}\left(E_n^{hG}\right). 
	\end{equation}
	We note that the HFPSS in \cite{fixedpt} is obtained from the cobar complex of a descendable extension $E_n^{hG}\to E_n^{hG}\hsmash E_n$:		       \begin{equation}\label{eqn:Tot_EnhG_DH}
		\Tot\left[\bighsmash_{E_n^{hG}}^{\bullet+1}\left(E_n^{hG}\hsmash E_n\right)\right]\simeq \Tot\left[E_n^{hG}\hsmash E_n^{\hsmash \bullet+1}\right]\simeq \Tot\left[ \map_c\left((\Gbb_n/G)\times \Gbb_n^{\times \bullet}, E_n\right)\right].
	\end{equation}
	
	\begin{prop}
		The two BKSS's associated to the cosimplicial spectra in \eqref{eqn:Tot_EnhG} and \eqref{eqn:Tot_EnhG_DH} are isomorphic starting from their $E_2$-pages. 
	\end{prop}
	\begin{proof}
		Denote the entries in the BKSS for \eqref{eqn:Tot_EnhG_DH} by $^\mathrm{DH}E_r^{s,t}\left(E_n^{hG}\right)$. The $E_n^{hG}$-algebra map $E_n^{hG}\hsmash E_n\to E_n\hsmash E_n\to E_n$ induces a map between the cobar complexes
		\begin{equation*}
			f\colon    \Tot\left[\bighsmash_{E_n^{hG}}^{\bullet+1}\left(E_n^{hG}\hsmash E_n\right)\right]\longrightarrow \Tot\left[ \bighsmash_{E_n^{hG}}^{\bullet+1}E_n\right].
		\end{equation*}
		It follows that $f$ induces a map between the two spectral sequences $f_*\colon ~^\mathrm{DH}E_*^{*,*}\left(E_n^{hG}\right)\to\!^\HFP E_*^{*,*}\left(E_n^{hG}\right)$. On $E_1$-pages, the map \begin{equation*}
			f_*\colon~^\mathrm{DH}E_1^{s,t}\left(E_n^{hG}\right)=\map_c\left((\Gbb_n/G)\times \Gbb_n^{\times s}, \pi_t(E_n)\right)\to\!^\HFP E_1^{s,t}\left(E_n^{hG}\right)=\map_c\left(G^{\times s}, \pi_t(E_n)\right)
		\end{equation*} 
		is induced by the inclusion $G^{\times s}\subseteq (\Gbb_n/G)\times \Gbb_n^{\times s}$ of subspaces. By Shapiro's Lemma, $f_*$ is an isomorphism on $E_2$-pages:
		\begin{equation*}
			f_*\colon ~^\mathrm{DH}E_2^{s,t}\left(E_n^{hG}\right)=	H^s_c\left(\Gbb_n;\map_c\left(\Gbb_n/G,\pi_t(E_n)\right)\right)\simto ~^\HFP E_2^{s,t}\left(E_n^{hG}\right)=H^s_c\left(G;\pi_t(E_n)\right).
		\end{equation*}
		Then $f_*$ is an isomorphism on $E_r$-pages for all $r\ge 2$ by \cite[Theorem 5.3]{Boardman_ccss}. 
	\end{proof}
	On $K(n)$-local Picard groups,  we obtain a descent spectral  sequence:
	\begin{equation}\label{eqn:PicSS_EnhG}
		^{\pic}E_2^{s,t}\left(E_n^{hG}\right)= H_c^s(G;\pi_t\pic_{K(n)}(E_n))\Longrightarrow\pi_{t-s}\pic_{K(n)}\left(E_n^{hG}\right), \qquad t-s\ge 0. 
	\end{equation}
	\begin{rem}
		When $G=\Gbb_n$, $s=t=0$ or $t\ge 1$, Heard has identified the $E_2$-page of \eqref{eqn:PicSS_EnhG} in \cite[Example 6.18]{Heard_2021Sp_kn-local}. More recently, Mor uses the pro\'etale site to construct a model for the continuous action of $\Gbb_n$ on Morava $E$-theory, and identifies the entire $E_2$-page of the descent spectral sequence to compute Picard groups for $S^0_{K(n)}\simeq E_n^{h\Gbb_n}$ in \cite{mor2023picard}. 
	\end{rem}
	For computational purposes, we might need to compute $\Pic_{K(n)}\left(E_n^{hG}\right)$ by descending from $E_n^{hH}$ for some closed normal subgroup $H\trianglelefteq \Gbb_n$ contained in $G$.
	\begin{cor}\label{cor:desc_EnhG1G2}
		Let $G_1\le G_2\le \Gbb_n$ be two closed subgroups such that $G_1$ is normal in $G_2$. There are descent spectral sequences:
		\begin{align*}
			^\HFP E_2^{s,t}=H_c^s\left(G_2/G_1;\pi_t\left(E_n^{hG_1}\right)\right)
			&\Longrightarrow\pi_{t-s}\left(E_n^{hG_2}\right),&&\textup{\cite{Devinatz_LHS}}\\
			^\pic E_2^{s,t}=H_c^s\left(G_2/G_1;\pi_t\pic_{K(n)}\left(E_n^{hG_1}\right)\right)
			&\Longrightarrow\pi_{t-s}\pic_{K(n)}\left(E_n^{hG_2}\right), ~t-s\ge 0.&&
		\end{align*}
		Furthermore, when $t-s>0$ and $s\ge 0$, or $2\le r\le t-1$,  the two spectral sequences above are related by $^{\pic}d_r^{s,t-1}=\!^{\HFP}d_r^{s,t}$.  
	\end{cor}
	\begin{proof}
		By \cite[Lemma 3.5]{Devinatz_LHS}, $\pi_t\left(E_n^{hG}\Smash M\right)$ is a finite abelian group for any closed subgroup $G$ of $\Gbb_n$ and generalized Moore spectrum $M$ of type $n$. This means the system $\{\pi_t\left(E_n^{hG}\wedge M_j\right)\}$ satisfies the ML \Cref{cond:ML}. By \Cref{cor:En_hG_descendable}, the extension $E_n^{hG_2}\to E_n^{hG_1}$ is descendable. The claims then follow from \Cref{thm:pic_ss}. 
	\end{proof}	
	\subsection{The descent filtration on $K(n)$-local Picard groups}\label{subsec:desc_fil}
	The inverse limit  topology on $K(n)$-local Picard groups in \Cref{thm:Kn_Pic_limit} is a filtration on $\Pic_{K(n)}\left(E_n^{hG}\right)$. Another important filtration on $K(n)$-local Picard groups is the descent filtration \cite[\S3.3]{BBGHPS_exotic_h2_p2}\cite[\S1.3]{CZ_exotic_Picard}, defined through the Devinatz--Hopkins homotopy fixed spectral sequence \eqref{eqn:HFPSS_EnhG} to compute $\pi_*\left(S^0_{K(n)}\right)$. We now use the descent spectral sequence for $K(n)$-local Picard groups in \Cref{thm:pic_ss} to study this descent filtration on $\Pic_{K(n)}\left(E_n^{hG}\right)$. Similar to the HFPSS in \Cref{thm:pic_ss}, we have more generally
	\begin{prop}\label{prop:DSS_mod}
		Let $A\to B$ be a descendable $K(n)$-local profinite $G$-Galois extension. For a $K(n)$-local $A$-module $X$, if the inverse system $B\hsmash_A X$ satisfies the ML \Cref{cond:ML}, then there is a homotopy fixed point spectral sequence: 
		\begin{equation*}
			^{\HFP}E_1^{s,t}(X)=\map_c\left(G^{\times s}, \pi_t\left(B\hsmash_A X\right)\right),\qquad ^{\HFP}E_2^{s,t}(X)=H_c^s\left(G; \pi_t\left(B\hsmash_A X\right)\right) \Longrightarrow \pi_{t-s}(X).
		\end{equation*}
	\end{prop}
	\begin{proof}
		By \eqref{eqn:profin_mod_tot}, there is an equivalence of $A$-modules:
		\begin{equation*}
			X\simto \Tot\left[\left(\bighsmash_A^{\bullet+1}B\right)\hsmash_AX\right]\simeq \Tot \left[\map_c\left(G^{\times \bullet},B\Smash_A X\right)\right].
		\end{equation*}
		The Mittag-Leffler condition on $\{\pi_t(B\Smash_A X\Smash M_j)\}$ implies that the $E_1$-page of the BKSS associated to this cosimplicial spectra is:
		\begin{equation*}
			^{\HFP}E_1^{s,t}(X)=\pi_t\left[\map_c\left(G^{\times \bullet},B\Smash_A X\right)\right]\cong \map_c\left(G^{\times \bullet},\pi_t(B\Smash_A X)\right).
		\end{equation*}
		This identifies the $E_1$-page of the spectral sequence. Similar to \Cref{thm:pic_ss}, the $d_1$-differentials are the same as those in the cobar resolution of $\pi_t(B\Smash_A X)$ as a continuous $G$-module. It follows that the terms on the $E_2$-page are continuous group cohomology of $G$. 
	\end{proof}
	\begin{rem}
		Let $\mfrak=(p,u_1,\cdots,u_{n-1})$ be the maximal ideal of $\pi_0(E_n)$. In \cite[Theorem 4.3]{BH_E2term}, Barthel--Heard showed that if $\pi_t(E_n\hsmash X)$ is either pro-free and finitely generated as an $\pi_0(E_n)$-module or has bounded $\mfrak$-torsion, then $E_2$-terms of the HFPSS for $X$ are also continuous group cohomology of $G$. Note that both assumptions imply that $E_n\hsmash X$ satisfies the ML \Cref{cond:ML}. 
	\end{rem}
	\begin{const}\label{const:desc_fil}
		A descent filtration on $\Pic_{K(n)}\left(E_n^{hG}\right)$ can be constructed following \cite[Definition 3.27]{BBGHPS_exotic_h2_p2} and \cite[Equation 1.22]{CZ_exotic_Picard} as below:
		\begin{equation}\label{eqn:Pic_desc_filtration}
			\Pic_{K(n)}\left(E_n^{hG}\right)\supseteq \Pic^0_{K(n)}\left(E_n^{hG}\right)\supseteq\kappa\left(E_n^{hG}\right)\supseteq \kappa^{(2)}\left(E_n^{hG}\right)\supseteq \kappa^{(3)}\left(E_n^{hG}\right)\supseteq\cdots.
		\end{equation}
		\begin{itemize}
			\item By \Cref{prop:Pic_top}, the map of $K(n)$-local $\einf$-ring spectra $E_n^{hG}\to E_n$ induces a continuous group homomorphism $\phi_0\colon \Pic_{K(n)}\left(E_n^{hG}\right)\to \Pic_{K(n)}\left(E_n\right)$.  By \eqref{eqn:Pic_En}, we have $\Pic_{K(n)}\left(E_n\right)\cong \Pic\left(E_n\right)\cong \Z/2$, generated by $\Sigma E_n$. As a result, $\phi_0$ is surjective. Set
			\begin{equation*}
				\Pic^0_{K(n)}\left(E_n^{hG}\right):=\ker\phi_0=\left\{X\in \Pic_{K(n)}\left(E_n^{hG}\right)\mid E_n\hsmash_{E_n^{hG}}X\simeq E_n \right\}.
			\end{equation*}
			\item For any $X\in  \Pic^0_{K(n)}\left(E_n^{hG}\right)$, we note $E_n\hsmash_{E_n^{hG}} X\simeq E_n$ satisfies the ML \Cref{cond:ML}.\footnote{This is also true for any $X\in  \Pic_{K(n)}\left(E_n^{hG}\right)$, since $ E_n\hsmash_{E_n^{hG}}X\simeq \Sigma E_n$ if $X\notin \Pic^0_{K(n)}\left(E_n^{hG}\right)$. } Then \Cref{prop:DSS_mod} yields a  homotopy fixed spectral sequence with $E_2$-page
			\begin{equation}\label{eqn:HFPSS_Kn_local}
				E_2^{s,t}(X)=H^s_c\left(G;\pi_t\left(E_n\hsmash_{E_n^{hG}}X\right)\right)\Longrightarrow \pi_{t-s}(X). 
			\end{equation}			
			The homotopy groups  $\pi_t\left(E_n\hsmash_{E_n^{hG}}X\right)$ have a natural continuous $G$-action, which is adjoint to 
			\begin{equation*}
				\pi_t\left(E_n\hsmash_{E_n^{hG}}X\right)\to \pi_t\left(E_n\hsmash_{E_n^{hG}}E_n\hsmash_{E_n^{hG}}X\right)\cong \map_c\left(G, \pi_t\left(E_n\hsmash_{E_n^{hG}}X\right)\right). 
			\end{equation*}
			Notice we have a $G$-equivariant isomorphism $\pi_t\left(E_n\hsmash_{E_n^{hG}}X\right)\cong \pi_0\left(E_n\hsmash_{E_n^{hG}}X\right)\hotimes_{\pi_0(E_n)} \pi_t(E_n)$, where $G$ acts on $\pi_t(E_n)$ as a subgroup of $\Gbb_n$.  This implies that the $G$-action on $\pi_t\left(E_n\hsmash_{E_n^{hG}}X\right)$ is determined by that on $\pi_0\left(E_n\hsmash_{E_n^{hG}}X\right)$, which is non-equivariantly isomorphic to $\pi_0(E_n)$.  All such $G$-actions on $\pi_0(E_n)$ is then classified by  the (even) \textbf{algebraic $K(n)$-local Picard group} of $E_n^{hG}$:
			\begin{equation*}
				\Pic^{0,alg}_{K(n)}\left(E_n^{hG}\right):= H^1_c\left(G;\pi_0(E_n)^\times \right).
			\end{equation*}
			The assignment $X\mapsto\pi_0\left(E_n\hsmash_{E_n^{hG}}X\right) $ as a continuous $G$-module then defines a map $\phi_1\colon \Pic^0_{K(n)}\left(E_n^{hG}\right)\to \Pic^{0,alg}_{K(n)}\left(E_n^{hG}\right)$.  Set the \textbf{exotic $K(n)$-local Picard group} of $E_n^{hG}$ to be
			\begin{equation*}
				\kappa\left(E_n^{hG}\right):= \phi_1^{-1}\left(E_n^{hG}\right)=\left\{X\in \Pic^0_{K(n)}\left| ~\pi_0\left(E_n\hsmash_{E_n^{hG}} X\right)\cong \pi_0(E_n) \text{ as $G$-modules}\right.\right\} .
			\end{equation*}
			It is conventional to write $\kappa_n$ for $\kappa\left(E_n^{h\Gbb_n}\right)$. 
			\item When $X\in \kappa\left(E_n^{hG}\right)$, we have isomorphisms $E_2^{s,t}\left(E_n^{hG}\right)\cong E_2^{s,t}(X)$ for all $s,t$. Let $\iota_X\in E_2^{0,0}(X)$ be the image of the unit $1\in  H_c^0(G; \pi_0(E_n))=E_2^{0,0}\left(E_n^{hG}\right)$ under this isomorphism. Define
			\begin{equation*}
				\kappa^{(r)}\left(E_n^{hG}\right):= \left\{ X\in \kappa\left(E_n^{hG}\right)\mid \iota_X \text{ is an $r$-cycle in \eqref{eqn:HFPSS_Kn_local}}\right\}.
			\end{equation*}
			If $X\in \kappa^{(r)}\left(E_n^{hG}\right)$, then $\iota_X$ defines an isomorphism $E_m^{s,t}\left(E_n^{hG}\right)\simto E_m^{s,t}(X)$ for  $2\le m\le r+1$ and all $s,t$. For each $r$, the assignment $X\mapsto d_{r+1}(\iota_X)$ defines a map $\phi_{r+1}\colon \kappa^{(r)}\left(E_n^{hG}\right)\to E^{r+1,r}_{r+1}(X)\cong E^{r+1,r}_{r+1}\left(E_n^{hG}\right)$. 
		\end{itemize}
	\end{const}
	\begin{rem}
		For any inclusion of closed subgroups $G_1\le G_2$ of $\Gbb_n$, the filtration above is preserved under the base change maps $\Pic_{K(n)}\left(E_n^{hG_2}\right)\to \Pic_{K(n)}\left(E_n^{hG_2}\right)$.
	\end{rem}
	When $G=\Gbb_n$, we have $\Pic_{K(n)}\left(E_n^{h\Gbb_n}\right)=\Pic\left(\Sp_{K(n)}\right)$. This Picard group has been studied extensively using the filtration above.  At height $1$ and all primes, $\Pic_{K(1)}$ were computed in \cite{HMS_picard, Strickland_interpolation}:
	\begin{equation*}
		\Pic_{K(1)}=\begin{cases}
			\Zp\oplus \Z/2(p-1), & p>2;\\
			\Z_2\oplus\Z/4\oplus\Z/2, &p=2.
		\end{cases}
	\end{equation*}
	At height $2$, the algebraic $K(2)$-local Picard groups are computed by Hopkins when $p\ge 5$ (see \cite{S_E2, Hopkins_AWS2019,Lader_thesis}) and by Karamanov \cite{Karamanov_Picard} when $p=3$. In both cases, we have
	\begin{equation*}
		\Pic^{alg,0}_{K(2)}=H^1_c\left(\Gbb_2;\pi_0(E_2)^\times\right)\cong H^1_c\left(\Gbb_2;\W\Fbb_{p^2}^\times\right)\cong\Zp\oplus\Zp\oplus \Z/(p^2-1),\qquad  p\ge 3. 
	\end{equation*}
	The $K(2)$-local exotic Picard group $\kappa_2$ vanishes when $p\ge 5$ \cite{HMS_picard,Strickland_interpolation}. It was computed at $p=3$ by Goerss--Henn--Mahowald--Rezk  in \cite{GHMR_Picard}, and more recently at $p=2$ by Beaudry--Bobkova--Goerss--Henn--Pham--Stojanoska in \cite{BBGHPS_exotic_h2_p2}. 
	\begin{equation*}
		\kappa_2=\begin{cases}
			(\Z/8)^2\oplus(\Z/2)^3, & p=2;\\
			\Z/3\oplus\Z/3, &p=3;\\
			0, & p\ge 5.
		\end{cases}
	\end{equation*}
	\begin{rem}\label{rem:kappa_comparison}
		In \cite[Definition 3.3]{BBGHPS_exotic_h2_p2}, the authors used $\kappa(G)$ to denote
		\begin{equation*}
			\kappa(G)=\left\{\begin{array}{c|>{\raggedright\arraybackslash}m{0.475\textwidth}}
				X\in\kappa_n& 	There exists a $z\in \pi_0\left(E_n^{hG}\hsmash X\right)$ such that its image in $\pi_0\left(E_n\hsmash X\right)$ is a $\Gbb_n$-equivariant generator
			\end{array}\right\}.
		\end{equation*}
		By \cite[Proposition 3.7]{BBGHPS_exotic_h2_p2}, we have $\kappa(G)$ is contained in $\ker\left(\kappa_n\to \kappa\left(E_n^{hG}\right)\right)$. \cite[Lemma 3.9]{BBGHPS_exotic_h2_p2} says the  inclusion in the other direction is also true if the edge homomorphism $\pi_0\left(E_n^{hG}\right)\to H^0_c(G;\pi_0(E_n))$ in the HFPSS \eqref{eqn:HFPSS_EnhG} is surjective. This happens in all cases with complete computations (\cite[Example 3.10]{BBGHPS_exotic_h2_p2}) . 
	\end{rem}
	\begin{lem}\label{lem:desc_fil_fin}
		\begin{enumerate}
			\item The descending filtration \eqref{eqn:Pic_desc_filtration} is Hausdorff in the sense that $\bigcap_{r} \kappa^{(r)}\left(E_n^{hG}\right)=\left\{E_n^{hG}\right\}$.
			\item The maps $\phi_0,\phi_1,\phi_2,\cdots$ in \Cref{const:desc_fil} are all group homomorphisms such that $\kappa^{(r)}\left(E_n^{hG}\right)=\ker \phi_r$ for all $r\ge 2$. 
			\item For each fixed prime $p$, height $n$, and closed subgroup $G\le \Gbb_n$, the filtration \eqref{eqn:Pic_desc_filtration} is finite in the sense $\kappa^{(r)}\left(E_n^{hG}\right)=0$ when $r\gg 0$. 
		\end{enumerate}
	\end{lem}
	\begin{proof}
		\begin{enumerate}
			\item As a invertible module over itself, $E_n^{hG}\in \kappa\left(E_n^{hG}\right)$ by construction. The unit $1\in E_2^{0,0}\left(E_n^{hG}\right)$ is a permanent cycle, since it converges to the unit map $S^0\to E_n^{hG}$ of $E_n^{hG}$ as a ring spectra. This implies that $E_n^{hG}\in \bigcap_{r} \kappa^{(r)}\left(E_n^{hG}\right)$. 
			
			Conversely, let $X\in \bigcap_{r} \kappa^{(r)}\left(E_n^{hG}\right)$. Then $\iota_X$ is a permanent cycle, representing a map of $K(n)$-local spectra $f\colon S_{K(n)}^0\to X$. As $X$ is an $E_n^{hG}$-module, the map $f$ is adjoint to an $E_n^{hG}$-module map $\wt{f}\colon E_n^{hG}\to X$. One can then check $\wt{f}$ induces on isomorphism on the $E_2$-pages of the HFPSS's $E_2^{*,*}\left( E_n^{hG}\right)\to E_2^{*,*}\left(X\right)$ sending $1$ to $\iota_X$. By \cite[Theorem 5.3]{Boardman_ccss}, $\wt{f}$ is a weak equivalence. 
			\item  We need to show all the $\phi_r$'s are additive. The map $\phi_0\colon \Pic_{K(n)}\left(E_n^{hG}\right)\to \Pic_{K(n)}(E_n)$ is a group homomorphism by \Cref{prop:Pic_top}. 
			For $\phi_1$, notice that any $X, Y\in \Pic_{K(n)}^0\left(E_n^{hG}\right)$, we have a  K\"unneth isomorphism  
			\begin{equation*}
				\pi_0\left(E_n\hsmash_{E_n^{hG}} X\hsmash_{E_n^{hG}} Y\right)\cong \pi_0\left(E_n\hsmash_{E_n^{hG}} X\right)\hotimes_{\pi_0(E_n)}\pi_0\left(E_n\hsmash_{E_n^{hG}} Y\right).
			\end{equation*}			
			This implies that $\phi_1\colon  \Pic_{K(n)}^0\left(E_n^{hG}\right)\to  \Pic_{K(n)}^{alg,0}\left(E_n^{hG}\right)$ is a group homomorphisms. When $r\ge 2$, the Leibniz rule on differentials implies that $\phi_r$ is a group homomorphism. 
			\item 	If $X\in \kappa^{(r)}$, then $E_{r+1}^{s,t}(X)\cong E_{r+1}^{s,t}\left(E_n^{hG}\right)$ for all $s,t$. By \cite[Section 2.3]{BGH22} \cite[Theorem 5.3]{fixedpt} \cite[Corollary 15]{Str00}, there is a horizontal vanishing line at $s=N$ on $E_r^{*,*}\left(E_n^{hG}\right)$ when $r\gg 0$. This means if $X\in \kappa^{(N)}\left(E_n^{hG}\right)$, then $\iota_X$ is automatically an $m$-cycle for all $m\ge N$.  Similar to the second half of the proof of part (1), it follows that $X\simeq E_n^{hG}$. Hence $\kappa^{(m)}\left(E_n^{hG}\right)=\{E_n^{hG}\}$ when $m\ge N$. \qedhere
		\end{enumerate}
	\end{proof}
	\begin{prop}\label{prop:desc_fil_comparison}
		The descent filtration \eqref{eqn:Pic_desc_filtration} on $\Pic_{K(n)}\left(E_n^{hG}\right)$ is  same as the filtration associated to the $t-s=0$ stem of the descent spectral sequence \eqref{eqn:PicSS_EnhG} for $\Pic_{K(n)}\left(E_n^{hG}\right)$. 
	\end{prop}
	\begin{proof}
		Recall that the descent spectral sequence \eqref{eqn:PicSS_EnhG} is the Bousfield--Kan spectral sequence of the cosimplicial diagram:
		\begin{equation*}
			\pic_{K(n)}\left({E^{hG}_{n}}\right)\simeq \Tot\left[\pic_{K(n)}\left(\bighsmash_{E_n^{hG}}^{\bullet+1}E_n\right) \right].
		\end{equation*}
		The filtration on $\Pic_{K(n)}\left({E^{hG}_{n}}\right)$ from this spectral sequence is defined to be:
		\begin{align*}
			\Fil^{\ge m} \pi_0\pic_{K(n)}\left({E^{hG}_{n}}\right)&:=\ker\left(\pi_0\pic_{K(n)}\left({E^{hG}_{n}}\right)\longrightarrow \pi_0\Tot_{ m-1}\left[\pic_{K(n)}\left(\bighsmash_{E_n^{hG}}^{\bullet+1}E_n\right) \right]\right)\\
			&~\cong \ker\left(\Pic\left(\Mod_{K(n)} {E^{hG}_{n}}\right)\longrightarrow \Pic\left(\Tot_{ m-1}\left[\Mod_{K(n)}\left({\bighsmash_{E_n^{hG}}^{\bullet+1}E_n}\right) \right]\right)\right),
		\end{align*}
		where 
		\begin{itemize}
			\item $\Tot_{ m-1}$ is the partial totalization of the cosimplicial diagram for $0\le \bullet\le m-1$; 
			\item the functor $\Mod_{K(n)} {E^{hG}_{n}}\to \Tot_{ m-1}\left[\Mod_{K(n)}\left({\bighsmash_{E_n^{hG}}^{\bullet+1}E_n}\right)\right]$ sends a $K(n)$-local $E_n^{hG}$-module spectrum $X$ to its $(m-1)$-truncated $K(n)$-local $E_n$-Adams tower $\Tot_{ m-1}\left[\left(\bighsmash_{E_n^{hG}}^{\bullet +1}E_n\right)\hsmash_{E_n^{hG}} X\right]$.
		\end{itemize}	
		As a result, $X\in \Fil^{\ge m} \Pic_{K(n)}\left({E^{hG}_{n}}\right)$ iff its image in $\Tot_{ m-1}\left[\Mod_{K(n)}\left({\bighsmash_{E_n^{hG}}^{\bullet+1}E_n}\right)\right]$ is equivalent to that of $E_n^{hG}$, i.e. there is an equivalence of $(m-1)$-partial totalizations:
		\begin{equation}\label{eqn:Tot_HFPSS_equiv}
			\Tot_{ m-1}\left[\bighsmash_{E_n^{hG}}^{\bullet+1}E_n \right]\simeq \Tot_{ m-1}\left[\left(\bighsmash_{E_n^{hG}}^{\bullet+1}E_n\right) \hsmash_{E_n^{hG}} X\right].
		\end{equation}
		By \cite[Proposition 2.25]{Mathew_desc_nilp}, this partial totalization is equivalent to the $m-1$-truncated $K(n)$-local $E_n$-Adams tower of $E_n^{hG}$-modules.  
		
		When $m=1$, \eqref{eqn:Tot_HFPSS_equiv} is an equivalence of spectra $E_n\simeq E_n\hsmash_{E_{hG}} X$. It follows that $\Fil^{\ge 1}\Pic_{K(n)}\left({E^{hG}_{n}}\right)=\Pic^0_{K(n)}\left({E^{hG}_{n}}\right)$.
		
		When $m=2$, the equivalence of truncated Adams towers becomes:
		\begin{equation*}
			\begin{tikzcd}
				E_n\rar["\simeq "] \dar[shift left=1.5 ex] \dar[shift right=1.5ex]&E_n\hsmash _{E_n^{hG}}  X \dar[shift left=1.5ex] \dar[shift right=1.5ex]\\
				E_n\hsmash _{E_n^{hG}}E_n\rar["\simeq "] \uar& E_n\hsmash _{E_n^{hG}}E_n\hsmash _{E_n^{hG}} X\uar.
			\end{tikzcd}
		\end{equation*}
		Taking $\pi_0$, the diagram translates to a $G$-equivariant isomorphism $\iota_X\colon \pi_0(E_n)\simto \pi_0\left(E_n\hsmash_{E_{hG}} X\right)$. From this, we obtain $\Fil^{\ge 2}\Pic_{K(n)}\left({E^{hG}_{n}}\right)=\kappa\left({E^{hG}_{n}}\right)$.
		
		When $m\ge 3$, set $[\iota_X]:=\iota_X(1)\in E_2^{0,0}\left(E_n^{hG}\right)$. Notice that $d_r([\iota_X])$ is defined only using the first $r$-terms in the $K(n)$-local $E_n$-Adams tower of $X$ as an $E_n^{hG}$-module.   Then the equivalence \eqref{eqn:Tot_HFPSS_equiv} implies that the $d_r$-differentials supported the $s=0$  lines of  the HFPSS's for $X$ and $E_n^{hG}$ are the same for $r\le m-1$. In particular, we have $d_2([\iota_X])=\cdots =d_{m-1}([\iota_X])=0$. This yields the inclusion $\Fil^{\ge m}\Pic_{K(n)}\left({E^{hG}_{n}}\right)\subseteq \kappa^{(m-1)}\left({E^{hG}_{n}}\right)$. 
		
		Conversely, if $[\iota_X]$ is an $(m-1)$-cycle in the BKSS for $\Tot\left[\bighsmash_{E_n^{hG}}^{\bullet+1}E_n \hsmash_{E_n^{hG}} X\right]$, then it is also an $(m-1)$-cycle for the BKSS associated to the partial totalization $\Tot_{ m-1}\left[\bighsmash_{E_n^{hG}}^{\bullet+1}E_n \hsmash_{E_n^{hG}} X\right]$. Notice the latter BKSS collapses on the $E_m$-page by construction, $[\iota_X]$ is a permanent cycle there and converges to an element $\bar{\iota}_X\in \pi_0\Tot_{ m-1}\left[\bighsmash_{E_n^{hG}}^{\bullet+1}E_n \hsmash_{E_n^{hG}} X\right]$. Similar to the proof of part (1) of  \Cref{lem:desc_fil_fin}, we can show that  $\bar{\iota}_X$ is adjoint to a weak equivalence \begin{equation*}
			f_X\colon \Tot_{ m-1}\left[\bighsmash_{E_n^{hG}}^{\bullet+1}E_n \right]\longrightarrow\Tot_{ m-1}\left[\left(\bighsmash_{E_n^{hG}}^{\bullet+1}E_n \right)\hsmash_{E_n^{hG}} X\right].
		\end{equation*} This proves the inclusion in the other direction $\Fil^{\ge m}\Pic_{K(n)}\left({E^{hG}_{n}}\right)\supseteq \kappa^{(m-1)}\left({E^{hG}_{n}}\right)$ when $m\ge 3$. 
	\end{proof}
	\begin{quest}
		\Cref{prop:desc_fil_comparison} only identifies the descent filtration in \Cref{const:desc_fil} with the one arising from the descent spectral sequence for Picard groups \eqref{eqn:PicSS_EnhG}. When $r\ge 2$, one can further wonder if the detection maps $\phi_r$ in \Cref{const:desc_fil} are identified with the maps $\Fil^{\ge r}\Pic_{K(n)}\left(E_n^{hG}\right)\to ~^\pic E_r^{r,r}\left(E_n^{hG}\right)$ in the descent spectral sequence \eqref{eqn:PicSS_EnhG} for $\pic_{K(n)}\left(E_n^{hG}\right)$. (\Cref{prop:desc_fil_comparison} only implies that their kernels are the same.)
		
		More precisely, for any $X\in \Fil^{\ge r}\Pic_{K(n)}\left(E_n^{hG}\right), r\ge 1$, pick an equivalence $f_X\colon E_n\simto E_n\hsmash_{E_{hG}} X$ and let $\iota_X\in \pi_0\left(E_n\hsmash_{E_{hG}} X\right)$ be the image of $1\in \pi_0(E_n)$ under $(f_X)_*$. Does associated graded map of the filtration \begin{equation*}
			\begin{tikzcd}
				\Fil^{\ge r}\Pic_{K(n)}\left(E_n^{hG}\right)\rar[->>]&^\pic E_\infty^{r,r}\left(E_n^{hG}\right)\subseteq~^\pic E_r^{r,r}\left(E_n^{hG}\right)\cong\!^\HFP E_r^{r,r-1}\left(E_n^{hG}\right) 
			\end{tikzcd}
		\end{equation*}
		send $X$ to $\left(f_X^{-1}\right)_*(d_r(\iota_X))$?
	\end{quest}
	By analyzing the HFPSS for the $K(n)$-local sphere, Hopkins-Mahowald-Sadofsky proved in \cite[Proposition 7.5]{HMS_picard} that the exotic $K(n)$-local Picard group $\kappa_n=\kappa\left(S^0_{K(n)}\right)$ vanishes when $(p-1)\nmid n$ and $2p-1>n^2$. According to \Cref{prop:desc_fil_comparison}, we can study the descent filtration on $\Pic_{K(n)}\left(E_n^{hG}\right)$ from the descent spectral sequence \eqref{eqn:PicSS_EnhG} for Picard groups.  This leads to the following algebraicity result:
	\begin{mainthm}\label{thm:descent_Pic_EnhG}
		Fix a prime $p>2$.  Let $G\le \Gbb_n$ be a closed subgroup such that  $G\cap \zpx$ is cyclic of order $m$, where $\zpx\le \Zpx=Z(\Gbb_n)$ is the torsion subgroup of the center $\Zpx$ of $\Gbb_n$. Denote the $p$-adic cohomological dimension of $G$ by $\cd_p G$. 
		\begin{enumerate}
			\item When $2m+1>\cd_p G$, the exotic Picard group $\kappa\left(E_n^{hG}\right)$ vanishes and the descent filtration on $K(n)$-local Picard group $\Pic_{K(n)}\left(E_n^{hG}\right)$ is:
			\begin{equation*}
				\begin{tikzcd}
					0\rar& H_c^{1}(G;\pi_{0}(E_n)^\times)\rar & \Pic_{K(n)}\left(E_n^{hG}\right)\rar["\phi_0"] &\Z/2\rar & 0.
				\end{tikzcd}	 	
			\end{equation*}
			\item When $2m+1=\cd_p G$, the map $\phi_1\colon\Pic_{K(n)}^0\left(E_n^{hG}\right)\to \Pic_{K(n)}^{alg,0}\left(E_n^{hG}\right)=H_c^{1}(G;\pi_{0}(E_n)^\times)$ is surjective. 
		\end{enumerate}
	\end{mainthm}
	\begin{rem}
		There is no guarantee that the extension above splits or not. We will see examples of both scenarios in \Cref{thm:Mackey_Pic_k1_p=odd} and \Cref{thm:Mackey_Pic_k1_p=2}. 
	\end{rem}
	\begin{rem}
		When $G=\Gbb_n$ and $(p-1)\nmid n$, the bound in (1) becomes $2p-1>\cd_p \Gbb_n=n^2$. Pstr\k{a}gowski proved a slightly weaker bound $2p-2>n^2+n$ in \cite[Theorem 1.1]{Piotr_2018} using the algebraicity of the category of $E_n$-local spectra $\Sp_{E_n}$ in \cite{Pstragowski_chromatic_algebraic}. 
		
		This case has also been discussed in \cite[Remark 2.6]{Piotr_2018} and \cite[Proposition 1.25]{CZ_exotic_Picard} as a consequence of Heard's partial identification of the $E_2$-page of the descent spectral sequence \eqref{eqn:PicSS_EnhG} in \cite[Example 6.18]{Heard_2021Sp_kn-local}.   
	\end{rem}
	\begin{proof}
		Consider the Hochschild-Lyndon-Serre spectral sequence to compute the $E_2$-page of the descent spectral sequence \eqref{eqn:PicSS_EnhG} of $\pic_{K(n)}\left(E_n^{hG}\right)$. 
		\begin{equation*}
			E_2^{r,s}=H^r_c\left(G/(G\cap \zpx) ;H^s\left(G\cap \zpx;\pi_{2k}(E_n)\right)\right)\Longrightarrow H_c^{r+s}\left(G;\pi_{2k}(E_n)\right).
		\end{equation*}
		The finite group $G\cap \zpx$ has an order $m$ coprime to $p$ since $p>2$. This means $E_2^{r,s}=0$ unless $s=0$. The center $Z(\Gbb_n)=\Zpx$ of the Morava stabilizer group $\Gbb_n$ acts on $\pi_{2k}(E_n)$ by multiplication by the $k$-th power. Since the $G\cap \zpx$ is cyclic of order $m$, its action on $\pi_{2k}(E_n)$ has no fixed points unless $m$ divides $k$ (it is trivial then).  As a result, the HLSSS collapses and $H_c^s\left(G;\pi_{2k}(E_n)\right)=0$ unless $m\mid k$.   The $p$-adic cohomological dimension $G$ then gives a horizontal vanishing line on the $E_2$-page of the descent spectral sequence \eqref{eqn:PicSS_EnhG} for $\pic_{K(n)}\left(E_n^{hG}\right)$. The claim follows by analyzing the $t-s=-1,0,1$ stems on the $E_2$-page of the descent spectral sequence in \Cref{fig:DSS_Pic_EnhG}.
		\begin{figure}[ht]
			\begin{tikzpicture}[xscale=4]
				\node at (-1,0) {$0$};
				\node at (0,0) {$\Z/2$};
				\node at (1,0) (G) {$H^0_c(G;\pi_0(E_n)^\times)$};
				\node at (-1,1) {$H^1_c(G;\Z/2)$};
				\node at (0,1) (A) {$H^1_c(G;\pi_0(E_n)^\times)$};
				\node at (1,1) {$0$};
				\node at (-1,2) {$H^2_c(G;\pi_0(E_n)^\times)$};
				\node at (0,2) {$0$};
				\node at (1,2) {$\cdots$};
				\node at (-1,3) {$0$}; 
				\node at (0,3) {$\cdots$};
				\node at (1,3) {$0$};
				\node at (-1,4) {$\cdots$};
				\node at (0,4) {$0$};
				\node at (-1,5) {$0$};
				\node at (1,4) (E) {$H^{2m}_c(G;\pi_{2m}(E_n))$};
				\node at (0,5) (C) {$H^{2m+1}_c(G;\pi_{2m}(E_n))$};
				\node at (1,5) {$0$};
				\node at (-1,6) (B) {$H^{2m+2}_c(G;\pi_{2m}(E_n))$};
				\node at (0,6) {$0$};
				\draw[color=red,thick,dashed, ->] (A) -- node[below left,color=red] {$d_{2m+1}$?} (B);
				\draw[color=red,thick,dashed, ->] (G) -- node[below left,color=red] {$d_{2m+1}$?} (C);
				\node at (1,6) {$\cdots$};
				\node at (-1,7) {$0$}; 
				\node at (0,7) {$\cdots$};
				\node at (1,7) {$\cdots$};
				\draw[thick, ->] (-1.5,-0.5) -- (1.5,-0.5);
				\draw[thick, ->] (-1.5,-0.5) -- (-1.5,7.5);
				\node[below] at (-1.125,-0.5) {$t-s=-1$};
				\node[below] at (0,-0.5) {$0$};
				\node[below] at (1,-0.5) {$1$};
				\node[left] at (-1.5,0) {$s=0$};
				\node[left] at (-1.5,1) {$1$};
				\node[left] at (-1.5,2) {$2$};
				\node[left] at (-1.5,3) {$3$};
				\node[left] at (-1.5,4) {$\cdots$};
				\node[left] at (-1.5,5) {$2m+1$};
				\node[left] at (-1.5,6) {$2m+2$};
				\node[left] at (-1.5,7) {$2m+3$};
			\end{tikzpicture}
			\caption{The $-1,0,1$ stems of the DSS for $\pic_{K(n)}\left(E_n^{hG}\right)$\\ {\small (adapted from \cite[\S1.3]{CZ_exotic_Picard})}}\label{fig:DSS_Pic_EnhG}
		\end{figure}
		\begin{itemize}
			\item When $2m+1\ge \cd_p(G)$, the first two nonzero terms on the $0$-stem are  $E_2^{0,0}=\Z/2$ and $E_2^{1,1}=\Pic_{K(n)}^{alg,0}=H^1_c(G;\pi_0(E_n)^\times)$. They are both permanent cycles, since the targets of all potential non-zero differentials supported by them are above the horizontal vanishing line at $s=\cd_p(G)$. This implies that $\phi_1\colon\Pic_{K(n)}^0\left(E_n^{hG}\right)=\Fil^{\ge 1}\Pic_{K(n)}\left(E_n^{hG}\right)\twoheadrightarrow \!^\pic E_\infty^{1,1}=E_2^{1,1}= \Pic_{K(n)}^{alg,0}\left(E_n^{hG}\right)=H_c^{1}(G;\pi_{0}(E_n)^\times)$ is a surjection.
			\item When $2m+1>\cd_p(G)$, the two terms $E_2^{0,0}$ and $E_2^{1,1}$ are the \emph{only} non-zero permanent cycles in the $0$-stem, since the next term $E_2^{2m+1,2m+1}$ is above the horizontal vanishing line at $s=\cd_p(G)$. As a result, $\phi_1$ is also injective in this case and we obtain the descent filtration. \qedhere		
		\end{itemize}
	\end{proof}
	\begin{rem}
		Similar to  \Cref{thm:descent_Pic_EnhG}, Culver and the second author  proved in \cite{CZ_exotic_Picard} that when $(p-1)\nmid n$, the map \begin{equation*}
			\phi_{2p-1}\colon \kappa_n=\kappa^{(2p-2)}_n=\Fil^{\ge 2p-1}\Pic_{K(n)} \to  \!^\pic E_{2p-1}^{2p-1,2p-1}\left(E_n^{h\Gbb_n}\right)= \!^\pic E_{2}^{2p-1,2p-1}\left(E_n^{h\Gbb_n}\right)=H_c^{2p-1}(G;\pi_{2p-2}(E_n))
		\end{equation*}
		is an isomorphism when $4p-3>n^2$, and is a surjection when $4p-3=n^2$. This is because the term $^\pic E_2^{0,1}=H^0_c(\Gbb_n;\pi_0(E_n)^\times)=\Zpx$ does not support any differential in the DSS \eqref{eqn:PicSS_EnhG}, as any $\alpha\in \Zp\cong H^0_c(\Gbb_n;\pi_0(E_n))$ is a permanent cycle in the HFPSS for the ring spectrum $E_n^{h\Gbb_n}\simeq S^0_{K(n)}$.  When $(p-1)\nmid n$, we also have $\cd_p(\Gbb_n)=n^2$.
		
		This claim could be generalized to the exotic Picard group $\kappa\left(E_n^{hG}\right)$ for closed subgroups $G$ of $\Gbb_n$ if the term $^\pic E_2^{0,1}=H^0_c(G;\pi_0(E_n)^\times)$, or equivalently $^\HFP E_2^{0,0}=H^0_c(G;\pi_0(E_n))$ does not support any differentials. As is mentioned in \Cref{rem:kappa_comparison} \cite[Example 3.10]{BBGHPS_exotic_h2_p2}, this happens in all cases with complete computations. 
	\end{rem} 
	\section{Computations of $K(1)$-local Picard Mackey functors}\label{sec:Pic_Mack}
	As an application to the profinite descent spectral sequences for Picard spaces in \Cref{cor:desc_EnhG1G2}, we compute $K(1)$-local Picard groups of $E_1^{hG}$ for all closed subgroups of $\Gbb_1=\Zpx$ at all primes in the remainder of the paper. When $G\le \Zpx$ is a pro-cyclic open subgroup, the homotopy fixed point $E_1^{hG}$ is equivalent to the $K(1)$-local algebraic $K$-theory spectrum of a finite field (see \Cref{prop:algK_finfld_p} and \Cref{rem:algK_finfld_2}).
	\subsection{$K(n)$-local Picard groups and Mackey functors} The Picard groups $\Pic_{K(n)}\left(E_n^{hG}\right)$ are not just a collection of (pro-)abelian groups. When we vary the closed subgroups $G$, those individual Picard groups form a \emph{Mackey functor}, i.e. there are restriction and transfer maps between them. 
	\begin{exmps}
		Let $G$ be a finite group and $A$ be an abelian group. 
		\begin{enumerate}
			\item The \textbf{constant $G$-Mackey functor} with value $A$ is defined by $\underline{A}(G/H):=A$ for all subgroups $H\le G$. The restriction maps are all identities and the transfer maps are determined by the double coset formula.
			\item The constant Mackey functor $\underline{A}$ has an \emph{opposite} $\underline{A^\op}$ whose underlying groups are still $A$, but the restriction and transfer maps are switched. 
			\item When $A$ has  a $G$-action, the \textbf{cohomological $G$-Mackey functor} of $A$ is defined by 
			\begin{equation*}
				\underline{H^s(-;A)}(G/H):= H^s(H;A).
			\end{equation*}
			where a subgroup $H\le G$ acts on $A$ by restricting the $G$-action. Given subgroups $H_1\le H_2$ of $G$, the restriction and transfer maps of the Mackey functor is given the corresponding map in group cohomology. See definitions in  \cite{Brown_cohomology}.
			\item When $G$ acts trivially on $A$, we have natural isomorphisms $H^1(G;A)\cong \hom(G,A)$.  This implies that $\underline{\hom(-;A)}$ is a cohomological $G$-Mackey functor. This is an example of \emph{globally-defined} Mackey functor in \cite[829]{Webb_guide_Mackey}.
			\item When a group $G$ acts trivially on $A=\Q/\Z$, we have $H^1(G;\Q/\Z)\cong \hom(G,\Q/\Z)$ is the Pontryagin dual of $G$, which is non-canonically isomorphic to $G$ if $G$ is finite abelian. We will denote the cohomological Mackey functor $\underline{\hom(-;\Q/\Z)}$ by $\underline{(-)^\vee}$. 
		\end{enumerate}
		The examples above (except for $\underline{A^\op}$) can be generalized to profinite groups if we only consider transfer maps between closed subgroups $H_1\le H_2$ such that $[H_2:H_1]<\infty$. 
	\end{exmps}
	\begin{prop}[{\cite[Proposition 3.1 and Corollary 3.12]{BBHS_PicE2hC4}}]
		Let $G$ be a finite group and $X$ be an $\einf$-ring spectrum with a $G$-action. Suppose $X^{hG}\to X$ is a faithful $G$-Galois extension, then the assignment $G/H\mapsto \Pic\left(X^{hH}\right)$ is a $G$-Mackey functor. More precisely, let $H_1\le H_2\le G$ be subgroups. 
		\begin{itemize}
			\item The restriction map $\res\colon \Pic\left(X^{hH_2}\right)\to \Pic\left(X^{hH_1}\right)$ is induced by the base change of invertible $X^{hH_2}$-modules $A\mapsto A\Smash_{X^{hH_2}}X^{hH_1}$.
			\item The transfer map $\tr\colon \Pic\left(X^{hH_1}\right)\to \Pic_{H_2}\left(X^{hH_1}\right)\cong \Pic\left(X^{hH_2}\right)$ is induced by the multiplicative norm $A\mapsto \mathrm{Norm}_{H_1}^{H_2} A$. 
		\end{itemize}
	\end{prop}
	We note that the restriction maps can be defined for any pairs of subgroups $H_1\le H_2$, whereas this definition of transfer maps only applies when $[H_2:H_1]$ is finite. 
	As $\Gbb_1=\Zpx$ is abelian, all of its (closed) subgroups are normal. This means for any $H_1\le H_2\le \Zpx$ with $[H_2:H_1]<\infty$, the map $E_1^{hH_2}\to E_1^{hH_1}$ is a faithful $K(1)$-local $(H_2/H_1)$-Galois extensions. Hence there are restriction and transfer maps between their Picard groups. 
	\subsection{Computations at height $1$ and odd primes}\label{subsec:mackey_odd}
	Let $p$ be an odd prime. Next we compute the $K(1)$-local Picard group at $p$ as a Mackey functor for the profinite group $\Zpx$. Let's first recall some basic facts about the structure of this group. The first step in the computation of the $\Zpx$-Picard Mackey functor is to determine all of its \emph{closed} subgroups. The group $\Zpx$ is pro-cyclic, since it is the limit of cyclic groups $\zx{p^v}\cong C_{(p-1)p^{v-1}}$.  For any element $\alpha\in \Zpx$, let $\langle\alpha\rangle$ be the closed subgroup of $\Zpx$ generated by $\alpha$. 
	\begin{lem}\label{lem:Zpx} When $p$ is an odd prime, the profinite group $\Zpx$ has the following properties:
		\begin{enumerate}
			\item The group $\Zpx$ is procyclic. 
			\item  For $1\le k\le \infty$, the subset $1+p^k\Zp=\{a\in \Zpx\mid a\equiv 1\mod p^k\}$ is a closed subgroup of $\Zpx$, where $1+p^\infty\Zp:=\{1\}$. Moreover, for any element $a\equiv 1\mod p$, the closed subgroup of $\Zpx$ generated by $a$ is $1+p^k\Zp$ where $k=v_p(a-1)$. 
			\item The maximal finite subgroup of $\Zpx$ is the group of $(p-1)$-st roots of unity in $\Zp$.
			\item The profinite group $\Zpx$ decomposes as a direct product of its closed subgroups: $\Zpx\cong \zpx\times(1+p\Zp)$.
			\item Closed subgroups $G$ of $\Zpx$ are of the form $G = G_\fin \times G_\pro$ where $G_\fin$ is a subgroup of $\zpx$ and $G_\pro \cong 1+p^k\Zp$ for $1\leqslant k \leqslant \infty$ is a pro-$p$-group.
		\end{enumerate}
	\end{lem}
	\begin{proof}
		\begin{enumerate}
			\item By definition, $\Zpx= \lim\limits_{k}\zx{p^k}$. Each of the finite groups  $\zx{p^k}$ in the inverse system is cyclic and we can pick a compatible system of generators $\alpha_k$ of $\zx{p^k}$ such that $\alpha_m\equiv \alpha_k\mod p^k$ for any $m\ge k$. The sequence $\{\alpha_k\}$ converges to an element $\alpha\in\Zpx$. It is a pro-generator, since the smallest closed group containing $\alpha$ is $\Zpx$ itself. 
			\item  When $k<\infty$, the subset $1+p^k\Zp$ is the kernel of the projection map $\Zpx\to \zx{p^k}$.  This implies that it is a closed (and also open) subgroup of  $\Zpx$.  When $v_p(a-1)=k>0$ and $v\ge k$, the residue class of $a$ in $\zx{p^v}$ generates the subgroup $\{b\in \zx{p^v}\mid b\equiv 1 \mod p^k\}$. It follows that the smallest closed subgroup containing $a$ is $1+p^k\Zp$. 
			\item This follows from Hensel's Lemma. 
			\item Consider the Teichm\"uller character $\omega\colon \Zpx\to \zpx \to \Zpx$, which sends $a\in \Zpx$ to the unique element $\omega(a)\in \Zpx$ such that $\omega(a)\equiv a\mod p$ and $ \omega(a)^{p-1}=1$. Explicitly, the character is given by the formula $\omega(a)=\lim_{n\to \infty} a^{p^n}$. The direct product decomposition $f\colon  \Zpx\to \zpx\times (1+p\Zp)$ is then given by $f(a)=(\omega(a), a\cdot \omega(a)^{-1})$. One can check this map is a continuous isomorphism of profinite groups.  
			\item A closed subgroup $G\le \Zpx$ is necessarily (pro-)cyclic since $\Zpx$ is. Pick a (pro-)generator $\alpha\in G$. The explicit formula above implies that  $\omega(\alpha)\in G$, which generates a finite subgroup $\langle \omega(\alpha)\rangle=G\cap \zpx$. The element $\alpha\cdot \omega(\alpha)^{-1}$ is then congruent to $1$ modulo $p$. Set $k=v_p(\alpha\cdot \omega(\alpha)^{-1}-1)$. We then have $G=\langle \omega(\alpha)\rangle \times (1+p^k\Zp) \le \Zpx$ as claimed. \qedhere
		\end{enumerate}
	\end{proof}
	Let $\Fq$ be a finite field with $q$ elements. Quillen's computation  \cite{Quillen_finfld} of algebraic $K$ theory of finite fields implies that when $p\nmid q$, we have an equivalence of $K(1)$-local $\einf$-ring spectra $L_{K(1)}K(\Fq)\simeq \left(KU^\wedge_p\right)^{h\langle q\rangle}$. The converse is also true. 
	\begin{prop}\label{prop:algK_finfld_p}
		Any finite Galois extension of $S^0_{K(1)}$ is equivalent to the $K(1)$-local algebraic $K$-theory spectrum for some finite field $\Fq$ when $p>2$. 
	\end{prop}
	\begin{proof}
		By \Cref{lem:Zpx}, any open subgroup $G$ of $\Zpx$ is pro-cyclic. It suffices to find a topological generator $q\in G$ that is (a power of) a prime number. Pick a topological generator $\alpha\in G$ and write $G= \langle\omega(\alpha)\rangle\times 1+p^k\Zp$ as in the proof of \Cref{lem:Zpx}. Then any integer $q$ satisfying $q\equiv \alpha\mod p^{k+1}$ is a topological generator of $G$. This $q$ can be chosen to be a prime number by Dirichlet's theorem on arithmetic progressions. 
	\end{proof}
	For a closed subgroup $G\le \Zpx$, we compute $\Pic\left(E_1^{hG}\right)$ in two steps. First, we compute $\Pic_{K(1)}\left(E_1^{hG_\fin}\right)$ by the descent spectral sequence \cite{MS_Picard}:
	\begin{equation}\label{eqn:fin_PicSS}
		E^{s,t}_2=H^s(G_\fin; \pi_t (\pic_{K(1)}(E_1))) \Longrightarrow \pi_{t-s} \pic_{K(1)}\left(E_1^{hG_\fin}\right).
	\end{equation}
	For our purpose, we only need to compute $\pi_0 \pic_{K(1)} \left(E_1^{hG_\fin}\right) = \Pic_{K(1)}\left(E_1^{hG_\fin}\right)$.	At height $1$, the Morava stabilizer group $\Zpx$ acts trivially on $\pi_0\left(\pic_{K(1)}(E_1)\right)\cong \Pic_{K(1)}(E_1)\cong \Z/2$ and $\pi_1 \left(\pic_{K(1)}(E_1)\right)\cong \pi_0(E_1)^\times \cong \Zpx$.  
	\begin{prop}\label{prop:pic_E1hG_odd}
		There is a non-split extension of $\zpx$-Mackey functors: 
		\begin{equation*}
			\begin{tikzcd}
				0\rar &\underline{(-)^\vee}\rar&\Pic_{K(1)}(E_1^{h-})\rar& \underline{\Z/2}\rar&0,
			\end{tikzcd}
		\end{equation*}
		In particular, the group $\Pic_{K(1)}\left(E_1^{hG}\right)$ is cyclic of order $2|G|$ for each subgroup $G \leq \zpx$.
	\end{prop}
	\begin{proof}
		Let $G\le \Zpx$ be a finite subgroup. It is contained in $\zpx$ by \Cref{lem:Zpx}. The $E_2$-page of the descent spectral sequence \eqref{eqn:fin_PicSS} is illustrated in \Cref{fig:DSS_pic_E1hG}.
		\begin{figure}[ht]
			\begin{tikzpicture}[xscale=2.5]		
				\fill[color=gray!30] (-2.5, 3.5) -- (-1.5, 3.5) -- (-1.5, 2.5) -- (-0.5,2.5) -- (-0.5,1.5) -- (0.5, 1.5) -- (0.5, 0.5) -- (2.5, 0.5) -- (2.5,4.5) -- (-2.5, 4.5) -- cycle;
				\node at (1,3) {Vanishing Regions: $\left\{\begin{array}{l}
						t\ge2\text{ and }s\ge 1;\\
						t <0.
					\end{array}\right.$};
				\fill[color=gray!30] (-0.5,-0.5) -- (-0.5,0.5) -- (-1.5,0.5) -- (-1.5, 1.5) -- (-2.5, 1.5) -- (-2.5, 2.5) -- (-2.5, -0.5) -- cycle;
				\node at (0,0) (A) {$\Z/2$};
				\node at (-1,1) (C1) {$\hom(G,\Z/2)$};
				\node at (-2,2) (D1){$H^2(G;\Z/2)$};	
				\node at (1,0) {$\Zpx$};
				\node at (0,1) (B) {$G^\vee$};
				\node at (-1,2) {$H^2(G;\Zpx)$};	
				\node at (-2,3) (C2) {$H^3(G;\Zpx)$};	
				\draw[thick] (A) -- (B);
				\draw[->,thick] (-2.5,-0.5) -- (2.5,-0.5);
				\draw[->,thick] (-0.5,-0.5) -- (-0.5,4.5);
				\node[left] at (-2.5,0) {$s=0$};
				\node[left] at (-2.5,1) {$1$};
				\node[left] at (-2.5,2) {$2$};
				\node[left] at (-2.5,3) (D2){$3$};
				\node[left] at (-2.5,4) {$4$};
				\node[below] at (-2.25,-0.5) {$t-s=-2$};  
				\node[below] at (-1,-0.5) {$-1$};
				\node[below] at (0,-0.5) {$0$};
				\node[below] at (1,-0.5) {$1$};
				\node[below] at (2,-0.5) {$2$};
			\end{tikzpicture}
			\caption{DSS for $\pic_{K(1)}\left(E_1^{hG}\right)$} \label{fig:DSS_pic_E1hG}
		\end{figure}
		
		On this page of the spectral sequence:
		\begin{itemize}
			\item All elements with bigrading $(s,t)$ satisfying $s\ge 1$ and $t\ge2$ are trivial. This is because $\pi_t (\pic_{K(1)}(E_1))\cong \pi_{t-1}(E_1)$ is $p$-complete when $t\ge 2$, and $G$ is a subgroup of $\zpx$ whose order $p-1$ is coprime to $p$. 
			\item $E_2^{0,0}=\Pic_{K(1)}((E_1))^G\cong \mathbb{Z}/2$. The generator of this group is a permanent cycle since it detects $\Sigma E^{hG}\in \Pic_{K(1)}\left(E^{hG}\right)$. Notice the $\zpx$-action on $\Pic_{K(1)}((E_1))=\Z/2$ is trivial, the restriction maps on this Mackey functor are identity and the transfer maps are multiplication by the subgroup index. As a result, $\underline{E_2^{0,0}}=\underline{\Z/2}$ is a constant $\zpx$-Mackey functor.  
			\item As the order of $G$ is coprime to $p$ and the group acts trivially on $\pi_1 \left(\pic_{K(1)}(E_1)\right)\cong \Zpx$, we have isomorphisms of $\zpx$-Mackey functors: \begin{equation*}
				\underline{E_2^{1,1}}=H^1(-;\Zpx)\cong \hom_c(-,\Zpx) \cong \hom(-,\zpx)\cong \underline{(-)^\vee}.
			\end{equation*} As $G\le \zpx$ is a finite group, the last term is non-canonically isomorphic to $G$ itself. We write $\underline{(-)^\vee}$ to stress the Mackey functor structure. Elements in $E_2^{1,1}$ are permanent cycles in the spectral sequence for degree reasons.
		\end{itemize}
		As a result, elements on the $0$-stem of the $E_2$-page are all permanent cycles. Passing to the $E_\infty$-page,  we need to solve an extension problem: 
		\[
		0\longrightarrow G^\vee\cong G \longrightarrow \Pic_{K(1)}\left(E_1^{hG}\right) \longrightarrow \Z/2 \longrightarrow 0.
		\]
		Consider the homotopy fixed point spectral sequence:
		\begin{equation*}
			E_2^{s,t}=H^s(G;\pi_t(E_1))\Longrightarrow \pi_{t-s}\left(E_1^{hG}\right).
		\end{equation*}
		The group $G\le \zpx$ acts on $\pi_{2k}(E_1)\cong \Zp$ by the character $G\hookrightarrow \zpx\xrightarrow{\omega}\Zpx\xrightarrow{(-)^k}\Zpx$, where $\omega$ is the Teichm\"uller character. Then we can compute
		\begin{equation*}
			H^s(G;\pi_t(E_1))=\left\{\begin{array}{ll}
				\pi_{t}(E_1), & s=0\text{ and }2|G|\text{ divides } t;\\
				0, & \text{else}.
			\end{array}\right.
		\end{equation*}
		As a result, the spectrum $E_1^{hG}$ has minimal periodicity $2|G|$, which implies that $\Pic_{K(1)}\left(E_1^{hG}\right)$ has a cyclic subgroup of order $2|G|$. Together with the extension above, we conclude that $\Pic_{K(1)}\left(E_1^{hG}\right)$ is indeed cyclic of order $2|G|$. 
	\end{proof}
	\begin{rem}
		This computation of $\Pic_{K(1)}\left(E_1^{hG}\right)$ also follows from \cite[Corollary 2.4.7]{MS_Picard}, since $\pi_*\left(E_1^{hG}\right)\cong \Zp[u^{\pm m}]$ with $|u^m|=-2m$ for a finite subgroup $G=\Z/m\le \Zpx$.  
	\end{rem}
	Write $G=G_\fin\times \left(1+p^k\Zp\right)$ as in \Cref{lem:Zpx} and assume $1\le k<\infty$.  By \Cref{cor:desc_EnhG1G2}, we compute $\Pic_{K(1)}\left(E_1^{hG}\right)$ via the profinite descent spectral sequence:
	\begin{equation*}
		E_2^{s,t}= H^s_c\left(1+p^k\Zp;\pi_t\pic_{K(1)} \left(E_1^{hG_\fin}\right)\right)\Longrightarrow \pi_{t-s}\left(\pic_{K(1)}\left(E_1^{hG}\right)\right).
	\end{equation*}
	Note that $1+ p^k \Zp$ is pro-$p$ and the order of $G_\fin$ is coprime to $p$. The homotopy fixed point spectral sequence vanishes from filtration $2$ and collapses on the $E_2$-page. At stem $0$, we have $\Pic\left(E_1^{G_\fin}\right)$ at filtration $0$ and $\hom_c (G_\pro, 1+p\Zp) \cong \hom_c(G, 1+p\Zp)$ at filtration $1$. 
	\begin{mainthm}
		\label{thm:Mackey_Pic_k1_p=odd}
		Let $p>2$ be an odd prime number,  $k\ge 1$ and $m\mid (p-1)$ be some positive integers. The Picard groups of $E_1^{hG}$ for all closed subgroups $G\le \Zpx$ are listed below:
		\begin{equation*}
			\Pic_{K(1)}\left(E_1^{hG}\right)=\left\{\begin{array}{lll}
				\Z/(2m)\oplus \Zp, & G=\Z/m\times (1+p^k\Zp);& \textup{\cite{HMS_picard,Strickland_interpolation}}\text{ for $G=\Zpx$}\\
				\Z/(2m), & G=\Z/m. & \textup{\cite{MS_Picard, Baker-Richter_invertible_modules}}
			\end{array}\right.
		\end{equation*}
		We have isomorphisms of $\Zpx$-Mackey functors:
		\[
		\Pic_{K(1)}\left(E_1^{h(-)}\right) \cong \hom_c(-,1+p\mathbb{Z}_p) \times \Pic_{K(1)}\left(E_1^{h(-)_\fin}\right).
		\]
		More precisely, let $m_1\mid m_2\mid (p-1)$ and $k_2\le k_1$ be some positive integers. The formulas for restriction and transfer maps between subgroups of finite indices are:
		\begin{equation*}
			\begin{tikzcd}[column sep=tiny]
				\Pic_{K(1)}\left(E_1^{h\left(\Z/m_2\times  (1+p^{k_2}\Zp)\right)}\right)\rar[symbol=\cong]&\Zp\oplus \Z/2m_2\vResAr[p^{k_1-k_2}\oplus 1]&~&~&\Pic_{K(1)}\left(E_1^{h\Z/m_2}\right)\rar[symbol=\cong]& \Z/2m_2\vResAr[1]\\
				\Pic_{K(1)}\left(E_1^{h\left(\Z/m_1\times (1+p^{k_1}\Zp)\right)}\right)\rar[symbol=\cong]& \Zp\oplus \Z/2m_1,\vTrAr[1\oplus (m_2/m_1)]&~&~&
				\Pic_{K(1)}\left(E_1^{h\Z/m_1}\right)\rar[symbol=\cong]&  \Z/2m_1\vTrAr[m_2/m_1],
			\end{tikzcd}				
		\end{equation*}
		where the numbers $d$ denote multiplication by $d$ on the corresponding summands. 
	\end{mainthm}
	\begin{proof}
		It remains to compute the explicit formulas of the restriction and transfer maps.  The finite subgroup cases were discussed in \Cref{prop:pic_E1hG_odd}. For profinite subgroups. For the pro-$p$ part, the induced restriction and transfers maps on the $\hom$-groups are
		\begin{equation*}
			\begin{tikzcd}[column sep=tiny]
				\hom_c(1+p^{k_2}\Zp,\Zpx)\rar[symbol=\cong]&\Zp\vResAr[p^{k_1-k_2}]\\
				\hom_c(1+p^{k_1}\Zp,\Zpx)\rar[symbol=\cong]& \Zp.\vTrAr[1]
			\end{tikzcd}				
		\end{equation*}
		Combining the restriction and transfer maps  from the finite subgroups and pro-$p$ subgroups of $\Zpx$, we obtain   the Mackey functor structure on $\Pic_{K(1)}\left(E_1^{h(-)}\right)$.
	\end{proof}
	\begin{rem}\label{rem:Pic_E1hG_topgen}
		When $G=\Zpx$,  the group $\Pic_{K(1)}=\Zp\oplus\Z/(2p-2)$ is topologically generated by $S^1_{K(1)}$, which corresponds to $(1,1)$ under the identification. The computation in \Cref{thm:Mackey_Pic_k1_p=odd}  implies that for $G=\Z/m\times (1+p^k\Zp)$ with $k<\infty$, the image of $S^1_{K(1)}$ in $\Pic_{K(1)}\left(E_1^{hG}\right)=\Zp\oplus \Z/(2m)$ is $(p^{k-1},1)$. As this element is uniquely divisible by $p^{k-1}$ in the Picard group, we obtain a unique $K(n)$-local spectrum $X=\Sigma^{1/p^{k-1}}E_1^{hG}$, whose $p^{k-1}$-st smash power over $E_1^{hG}$ is $\Sigma E_1^{hG}$. This element is a topological generator of $\Pic_{K(1)}\left(E_1^{hG}\right)$. 
	\end{rem}
	The computations in \Cref{thm:Mackey_Pic_k1_p=odd} allows us to verify  \Cref{prop:Pic_Kn_colim} in the example below: 
	\begin{cor}\label{cor:p_odd_colim_isom}
		Let $\Z/m\le \Zpx$ be a finite subgroup of order $m$. Then we have  colimits in the category of pro-abelian groups:
		\begin{equation*}
			\colim_{\res} \Pic_{K(1)}\left(E_1^{h\left(\Z/m\times (1+p^k\Zp)\right)}\right)\cong \Z/(2m)\cong \Pic_{K(1)}\left(E_1^{h\Z/m}\right).
		\end{equation*}
	\end{cor}
	\begin{rem}\label{rem:Pic_colim_p_odd}
		Note that the claim fails if we took the colimit in $\Ab$, as the colimit will have an extra $\Qp$ summand. This is an example of \Cref{rem:Pic_Kn_colim}  where the discrete Picard functor $\Pic_{K(1)}\colon \CAlg\left(\Sp_{K(1)}\right)\to \Ab$ does not preserve filtered colimits unless we lift it to $\Pro(\Ab)$ as in \Cref{thm:Kn_Pic_limit}.
	\end{rem}
	\subsection{Computations at height $1$ and prime $2$} 	\label{subsec:mackey_2}
	The $p=2$ case is more complicated than the odd prime cases, largely due to the extension problems. It is necessary to use the Mackey functor perspective (restriction maps) to resolve this issue. 
	\begin{lem}\label{lem:Z2x_lattice}
		Denote the closed subgroup generated by $\alpha\in \Z_2^\times$ by $\langle \alpha\rangle$. Then the closed subgroup lattice of $\Z_2^\times= \{\pm 1\} \times (1+4\Z_2)$ is:		
		\begin{equation*}
			\begin{tikzcd}[every arrow/.append style={dash,thick},column sep =small]
				\Z_2^\times \ar[dr]\rar \ar[ddr] &\{\pm 1\}\times \langle 9\rangle \rar\ar[dr]\ar[ddr] &  \{\pm 1\}\times \langle 17\rangle \rar\ar[dr]\ar[ddr] &\cdots \rar\ar[dr]\ar[ddr]& \{\pm 1\}\times \langle 2^{k}+1\rangle  \rar\ar[dr]\ar[ddr]& \{\pm 1\}\times \langle 2^{k+1}+1\rangle  \rar\ar[dr]\ar[ddr]& \cdots\rar &\{\pm 1\}\ar[ddr] & \\
				& \langle3\rangle \ar[dr] & \langle7\rangle \ar[dr]& \cdots\ar[dr]& \langle 2^{k-1}-1\rangle \ar[dr]& \langle 2^{k}-1\rangle \ar[dr]&\cdots\ar[drr] && \\
				&  \langle5\rangle \rar &  \langle9\rangle\rar & \cdots\rar &  \langle 2^{k-1}+1\rangle \rar & \langle 2^{k}+1\rangle \rar & \cdots\ar[rr] &&\{1\}.
			\end{tikzcd}
		\end{equation*}
		In the diagram above,
		\begin{itemize}
			\item When $\alpha\equiv 1\mod 4$, we have $\langle \alpha\rangle=1+2^{v_2(\alpha-1)}\Z_2$.
			\item $\{\pm 1\}\times \langle 2^k+1\rangle = \{\pm 1\}\times \langle 2^k-1\rangle$ as closed subgroups of $\Z_2^\times$.
			\item Each dash, except for the ones from $\{1\}$ and $\{\pm 1\}$ to profinite groups, indicates a subgroup of index $2$.  Subgroups in the $m$-th column have index $2^{m}$ in $\Z_2^\times$. 
		\end{itemize}
	\end{lem}
	\begin{proof}
		The direct product decomposition $\Z_2^\times= \{\pm1 \}\times (1+4\Z_2)$ preserves the profinite topology on both sides. As such, a subgroup $H\le \Z_2^\times$ is closed iff $H\cap (1+4\Z_2)$ is a closed subgroup of $1+4\Z_2$, which is pro-cyclic. This implies that $H\cap (1+4\Z_2)=1+2^k\Z_2$ for some $2\le k\le \infty$.
		
		When $k=2$, the closed subgroup $H$ contains $1+4\Z_2$. Such subgroups are in one-to-one correspondence to those of the quotient $\Z_2^\times /(1+4\Z_2)\cong C_2$. As a result, $H=\Z_2^\times$ or $1+4\Z_2$. 
		
		When $2<k< \infty$, we claim that $H$ is a subgroup of $\{\pm 1\}\times  (1+2^{k-1}\Z_2)$. Otherwise, $H$ would contain an element $\alpha$ such that $v_2(\pm\alpha-1)<k-1$, which implies that $v_2(\alpha^2-1)\le k-1$. The closed subgroup generated by $\alpha^2$ then contains $1+2^{k-1}\Z_2$, which contradicts the assumption that $H\cap( 1+4\Z_2) =1+2^k\Z_2$. It follows that $H$ corresponds to a subgroup of the subquotient $\left[\{\pm 1\}\times (1+2^{k-1}\Z_2)\right]/ (1+2^k\Z_2)$ of $\Z_2^\times$. This subquotient is isomorphic to the Klein group $C_2\times C_2$. From its subgroup lattice, we conclude $H$ is either $\langle2^k+1\rangle=1+2^k\Z_2$, $\langle 2^{k-1}-1\rangle$, or $\{\pm1 \}\times (1+2^k\Z_2)=\{\pm1 \}\times \langle2^k+1\rangle$. 
		
		When $k=\infty$, we have $1+2^\infty\Z_2=\{1\}$ and $H=\{1\}$ or $\{\pm1 \}$.  Otherwise, pick an element $\alpha\in H$ not equal to $\pm 1$. Let $m=v_2(\alpha^2-1)< \infty$. Then we have $H\supseteq \langle \alpha^2\rangle = 1+2^{m}\Z_2$, contradicting the assumption that $H\cap (1+4\Z_2)=\{1\}$. 
	\end{proof}
	\begin{rem}\label{rem:algK_finfld_2}
		Similar to \Cref{prop:algK_finfld_p},  if $G\le \Z_2^\times$ is pro-cyclic and open, then the homotopy fixed point $\left(KU^\wedge_2\right)^{hG}$ is equivalent to $L_{K(1)}K(\Fq)$ for some finite field $\Fq$. By \Cref{lem:Z2x_lattice}, $G=\langle\alpha\rangle$ where $\alpha=2^k\pm 1$. Any integer $q$ satisfying $q\equiv \alpha \mod 2^{k+1}$ is a topological generator of $G$. We can choose $q$ to be a prime number by Dirichlet's theorem on arithmetic progressions. 
	\end{rem}
	Before we start the discussion of $\Pic_{K(1)}(E_1^{h-})$ at the prime $2$ as $\Z_2^\times$-Mackey functor, let's recall the computation of $\pi_*(E_1^{h\Z_2^\times})$. One approach is to separate the finite group and pro $2$-group descent processes:  we first compute homotopy groups of $E_1^{hC_2}=KO^\wedge_2$ from $\pi_*(E_1)=\pi_*(KU^\wedge_2)$ using the Bott periodicity and then compute the HFPSS for $S_{K(1)}^0\simeq \left(KO^\wedge_2\right)^{h(1+4\Z_2)}$.  In particular, $E_1=KU_2^\wedge$ is equivalent to the $2$-completion of Aityah's Real $K$-theory $\KR$ as a $C_2$-equivariant commutative spectrum. 
	\begin{mainthm}
		\label{thm:Mackey_Pic_k1_p=2}
		Let $G\le \Z_2^\times$ be a closed subgroup. Then 
		\begin{equation*}
			\Pic_{K(1)}\left(E_1^{hG}\right)= \left\{\begin{array}{ll>{\raggedright\arraybackslash}p{0.28\textwidth}}
				\Z_2\oplus\Z/4\oplus \Z/2, &G=\Z_2^\times;& \textup{\cite[Theorem 3.3]{HMS_picard}\quad \cite[50]{Strickland_interpolation}}\\
				\Z_2\oplus\Z/2,& G=\langle 5\rangle \text{ or } \langle 3\rangle;&\\
				\Z_2\oplus \Z/8\oplus \Z/2, &G=\{\pm 1\}\times \langle 2^k+1\rangle, k\ge 3;&\\
				\Z_2\oplus \Z/2\oplus \Z/2, &G=\langle 2^k+1\rangle \text{ or }\langle2^k-1\rangle, k\ge 3;&\\
				\Z/8, & G=\{\pm 1\};&\textup{\cite[Theorem 7.1.2]{MS_Picard}}\quad \textup{\cite[Proposition 7.15]{Gepner-Lawson_Brauer}}\\
				\Z/2, & G=\{1\}.& \textup{\cite[Theorem 43 ff.]{Baker-Richter_invertible_modules}}
			\end{array}\right.
		\end{equation*}
		The restriction and transfer maps between any two closed subgroups $G_1\le G_2\le \Z_2^\x$ with $[G_2:G_1]=2$ are given by the following seven cases ($k\ge 3$ in Cases IV, V, VI):
		\renewcommand{\arraystretch}{3}
		\begin{equation*}
			\begin{tikzcd}
				\Pic_{K(1)}\left[E_1^{hG_2}\right]\vResAr[\res]\\ \Pic_{K(1)}\left[E_1^{hG_1}\right]\vTrAr[\tr]
			\end{tikzcd}=\begin{cases}
				\begin{tikzcd}
					\Z_2\langle1,\sigma\rangle/(4(\sigma-1))\vResAr[1,\sigma\mapsto 1]\\\Z_2\vTrAr[1\mapsto 1]
				\end{tikzcd}\bigoplus\underline{\Z/2},& \textup{Case I: } \langle 3\rangle\text{ or }\langle 5\rangle \le  \Z_2^\times;\\
				\begin{tikzcd}
					\Z_2\langle1,\sigma\rangle/(4(\sigma-1))\vResAr[\substack{1\mapsto (1,1)\\ \sigma\mapsto (1,-1)}]\\\Z_2\oplus \Z/8\vTrAr[\substack{(1,0)\mapsto 1+\sigma\\(0,1)\mapsto 1-\sigma }] 
				\end{tikzcd}\bigoplus\underline{\Z/2}^\op,&\textup{Case II: }  \{\pm 1\}\times\langle 9\rangle\le  \Z_2^\times;\\
				\begin{tikzcd}
					\Z_2\vResAr[1\mapsto (1,1)]\\\Z_2\oplus \Z/2 \vTrAr[\substack{(1,0)\mapsto 2\\ (0,1)\mapsto 0}]
				\end{tikzcd}\bigoplus\underline{\Z/2},& \textup{Case III: } \langle 9\rangle \le  \langle 3\rangle\text{ or }\langle 5\rangle;\\
				\underline{(\Z_2\bigoplus\Z/2)}\bigoplus \underline{\Pic}(\KR), &  \textup{Case IV: }  \langle 2^k-1\rangle\text{ or }\langle 2^k+1\rangle \le  \{\pm 1\}\times \langle 2^k+1\rangle;\\
				\underline{(\Z_2\bigoplus \Z/2)}^\op\bigoplus \underline{\Z/8}, &  \textup{Case V: } \{\pm 1\}\times \langle 2^{k+1}+1\rangle \le  \{\pm 1\}\times \langle 2^k+1\rangle;\\
				\underline{(\Z_2\bigoplus \Z/2)}^\op\bigoplus \underline{\Z/2}, &\textup{Case VI: }   \langle 1+2^{k+1}\rangle \le  \langle 2^k-1\rangle \text{ or } \langle 2^k+1\rangle;\\
				\underline{\Pic}(\KR)= \begin{tikzcd}
					\Z/8 \vResAr[1]\\ \Z/2\vTrAr[0]
				\end{tikzcd}, & \textup{Case VII: } \{1\}\le  \{\pm 1\}. 
			\end{cases}
		\end{equation*}
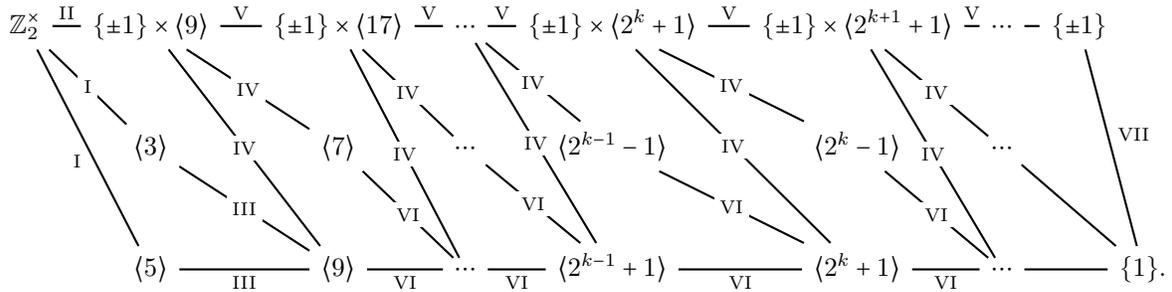
		\begin{equation*}
			\begin{tikzcd}[every arrow/.append style={dash,thick}
				,column sep =tiny, row sep =large
				]
				\Z_2^\times \ar[ddr,"\textup{I}"']\ar[r,"\textup{II}"] \ar[dr,"\textup{I}"description] &[5 pt]\{\pm 1\}\times \langle 9\rangle \ar[r,"\textup{V}"]\ar[dr,"\textup{IV}"description]\ar[ddr,"\textup{IV}"description] & [10 pt] \{\pm 1\}\times \langle 17\rangle \ar[r,"\textup{V}"]\ar[dr,"\textup{IV}"description]\ar[ddr,"\textup{IV}"description] &[5 pt]\cdots \ar[r,"\textup{V}"]\ar[dr,"\textup{IV}"description]\ar[ddr,"\textup{IV}"description]& [5 pt]\{\pm 1\}\times \langle 2^{k}+1\rangle  \ar[r,"\textup{V}"]\ar[dr,"\textup{IV}"description]\ar[ddr,"\textup{IV}"description]&[10 pt] \{\pm 1\}\times \langle 2^{k+1}+1\rangle  \ar[r,"\textup{V}"]\ar[dr,"\textup{IV}"description]\ar[ddr,"\textup{IV}"description]& \cdots\rar &\{\pm 1\}\ar[ddr,"\textup{VII}"] & [-10pt]\\ 
				& \langle3\rangle \ar[dr,"\textup{III}"description] & \langle7\rangle \ar[dr,"\textup{VI}"description]& \cdots\ar[dr,"\textup{VI}"description]& \langle 2^{k-1}-1\rangle \ar[dr,"\textup{VI}"description]& \langle 2^{k}-1\rangle \ar[dr,"\textup{VI}"description]&\cdots\ar[drr] && \\
				&  \langle5\rangle \ar[r,"\textup{III}"'] &  \langle9\rangle\ar[r,"\textup{VI}"'] & \cdots\ar[r,"\textup{VI}"'] &  \langle 2^{k-1}+1\rangle \ar[r,"\textup{VI}"'] & \langle 2^{k}+1\rangle \ar[r,"\textup{VI}"'] & \cdots\ar[rr] &&\{1\}.
			\end{tikzcd}
		\end{equation*}
	\end{mainthm} 
	\begin{proof}
		We start with the known Mackey functor $\underline{\Pic}(\KR)$ in Case VII and use the descent spectral sequence in \Cref{cor:desc_EnhG1G2}. Set $G=\langle 2^k+1\rangle$ or $\langle 2^k-1\rangle$ for $k\ge 2$. Recall that $\{\pm1\}\times \langle2^k-1\rangle= \{\pm1\}\times \langle2^k+1\rangle$ as subgroups of $\Z_2^\times$, we have 
		\begin{equation*}
			\left(KO^\wedge_2\right)^{h\langle2^k-1\rangle} \simeq \left(KU^\wedge_2\right)^{h(\{\pm1\}\times \langle2^k-1\rangle)}\simeq \left(KU^\wedge_2\right)^{h(\{\pm1\}\times \langle2^k+1\rangle)}\simeq \left(KO^\wedge_2\right)^{h\langle2^k+1\rangle}.
		\end{equation*} 
		Notice that $\cd_2(G)=\cd_2(\Z_2)=1$. This implies that the descent spectral sequences
		\begin{align*}
			H_c^s\left(G; \pi_t\pic_{K(1)}\left(KO^\wedge_2\right)\right)&\Longrightarrow \pi_{t-s}\pic_{K(1)}\left(\left(KO^\wedge_2\right)^{hG}\right)\cong\pi_{t-s}\pic_{K(1)}\left(E_1^{h(\{\pm 1\}\times G)}\right)\\
			H_c^s\left(G;\pi_t\pic_{K(1)}\left(KU^\wedge_2\right)\right)&\Longrightarrow \pi_{t-s}\pic_{K(1)}\left(\left(KU^\wedge_2\right)^{hG}\right)\cong\pi_{t-s}\pic_{K(1)}\left(E_1^{hG}\right)
		\end{align*} 
		collapse on the $E_2$-pages,  yielding extension problems:
		\begin{equation*}
			\begin{tikzcd}[row sep =0]
				0\rar &H^1_c\left(G;\pi_0\left(KO^\wedge_2\right)^\times\right)\rar& \Pic_{K(1)}\left[\left(KO^\wedge_2\right)^{hG}\right]\rar& \left[\Pic_{K(1)}\left(KO^\wedge_2\right)\right]^{G}\rar&0,\\
				0\rar &H^1_c\left(G;\pi_0\left(KU^\wedge_2\right)^\times\right)\rar& \Pic_{K(1)}\left[\left(KU^\wedge_2\right)^{hG}\right]\rar& \left[\Pic_{K(1)}\left(KU^\wedge_2\right)\right]^{G}\rar&0.
			\end{tikzcd}
		\end{equation*}
		The group $G$ acts trivially on $\pi_0$ and $\pi_1$ of the Picard spectra $\pic_{K(1)}\left(KO^\wedge_2\right)$ and $\pic_{K(1)}\left(KU^\wedge_2\right)$. It follows that the $H^1_c$ in the short exact sequences above is isomorphic to $\hom_c$, and every element in the Picard groups is fixed. As the group $G\cong \Z_2$ is a $2$-complete pro-cyclic group and $\pi_0\left(KO^\wedge_2\right)^\times\cong \pi_0\left(KU^\wedge_2\right)^\times\cong \Z_2^\times$ is also $2$-complete, we have isomorphisms:
		\begin{align*}
			H^1_c\left(G;\pi_0\left(KO^\wedge_2\right)^\times\right)&\cong \hom_c\left(G,\pi_0\left(KO^\wedge_2\right)^\times\right)\cong \Z_2^\times,\\ H^1_c\left(G;\pi_0\left(KU^\wedge_2\right)^\times\right)&\cong\hom_c\left(G,\pi_0\left(KU^\wedge_2\right)^\times\right)\cong \Z_2^\times .
		\end{align*}	
		In the end, the extension problems are:
		\begin{equation*}
			\begin{tikzcd}[row sep =0]
				0\rar &\Z_2^\times \rar& \Pic_{K(1)}\left[\left(KO^\wedge_2\right)^{hG}\right]\rar& \Z/8\rar&0,\\
				0\rar &\Z_2^\times\rar& \Pic_{K(1)}\left[\left(KU^\wedge_2\right)^{hG}\right]\rar& \Z/2\rar&0.
			\end{tikzcd}
		\end{equation*}
		To solve them, our starting point is the base case computed in \cite[Theorem 3.3]{HMS_picard}: 
		\begin{equation*}
			\Pic_{K(1)}\left(E_1^{h\Z_2^\times}\right)\cong \Z_2\langle1,\sigma\rangle/(4(\sigma-1))\oplus \Z/2\cong \Z_2\oplus\Z/4 \oplus\Z/2.
		\end{equation*}
		We will bootstrap from there using restriction maps in group cohomology. Fix $G=\langle3 \rangle$ or $\langle5\rangle$. To compute $\Pic_{K(1)}\left(E_1^{hG}\right)$, consider the following diagram between extensions:
		\begin{equation}\label{eqn:caseI_proof}
			\begin{tikzcd}[column sep =small]
				0\rar &\hom_c\left(G,\pi_0\left(KO^\wedge_2\right)^\x\right)\rar\dar[symbol=\cong]& \Pic_{K(1)}\left[\left(KO^\wedge_2\right)^{hG}\right]\rar\dar[symbol=\cong]& \left[\Pic_{K(1)}\left(KO^\wedge_2\right)\right]^{G}\rar\dar[symbol=\cong]&0\\ [-2.5 ex]
				0\rar &\Z_2^\times \rar["{g\mapsto (1+\sigma,0)}","{-1\mapsto (0,1)}"']\vResAr[1]& \Z_2\langle1,\sigma\rangle/(4(\sigma-1))\oplus \Z/2\rar\vResAr[\res]& \Z/8\rar\vResAr[1] &0\\ [2.5 ex]
				0\rar & \Z_2^\times \rar["i"]\dar[symbol=\cong] \vTrAr[(-)^2] & {\color{red}??} 
				\rar\dar[symbol=\cong] \vTrAr[\tr ]
				& \Z/2 \rar\dar[symbol=\cong] \vTrAr[0]&0\\ [-2.5 ex]
				0\rar &\hom_c\left(G,\pi_0\left(KU^\wedge_2\right)^\x\right)\rar& \Pic_{K(1)}\left[\left(KU^\wedge_2\right)^{hG}\right]\rar& \left[\Pic_{K(1)}\left(KU^\wedge_2\right)\right]^{G}\rar&0
			\end{tikzcd}
		\end{equation}
		The left column in this diagram is obtained by applying the functor $\hom_c(G,(-)^\times)$ to the Tambara functor:
		\begin{equation*}
			\begin{tikzcd}[column sep=tiny]
				\pi_0(KO^\wedge_2)\rar[symbol=\cong]&\Z_2\vResAr[1]\\
				\pi_0(KU^\wedge_2)\rar[symbol=\cong]& \Z_2.\vTrAr[2]
			\end{tikzcd}
		\end{equation*}
		The right column is obtained by applying the $G$-fixed point functor to the Mackey functor $\underline{\Pic}(\KR)$ in Case VII. An $\ext$-group computation shows $\Pic_{K(1)}\left(E_1^{hG}\right)$ is either $\Z_2\oplus \Z/2$ or $\Z_2\oplus \Z/2\oplus \Z/2$. Consider the image of the generator $g\in\Z_2^\times$ under the map $i\colon \Z_2^\times \to \Pic_{K(1)}\left[\left(KU^\wedge_2\right)^{hG}\right]$. For the left square in \eqref{eqn:caseI_proof} to commute under restriction maps, the projection $i(g)$ onto the free summand must be a multiple of $2$. This forces an isomorphism:
		\begin{equation*}
			\Pic_{K(1)}\left(E_1^{hG}\right)\cong \Z_2\oplus\Z/2.
		\end{equation*} 
		The restriction and transfer maps can then be computed via diagram chase. This finishes the computation of Case I.  All other cases can be proved in similar ways.  
		\begin{enumerate}[label=\textup{\Roman*.}]
			\item $G=\langle 3\rangle$ or $\langle 5\rangle\le \Z_2^\x$. 
			\begin{equation*}
				\begin{tikzcd}[column sep =small]
					0\rar &\hom_c\left(G,\pi_0\left(KO^\wedge_2\right)^\x\right)\rar\dar[symbol=\cong]& \Pic_{K(1)}\left[\left(KO^\wedge_2\right)^{hG}\right]\rar\dar[symbol=\cong]& \left[\Pic_{K(1)}\left(KO^\wedge_2\right)\right]^{G}\rar\dar[symbol=\cong]&0\\ [-2.5 ex]
					0\rar &\Z_2^\times \rar["{g\mapsto (1+\sigma,0)}","{-1\mapsto (0,1)}"']\vResAr[1]& \Z_2\langle1,\sigma\rangle/(4(\sigma-1))\oplus \Z/2\rar\vResAr[\substack{(1,0),(\sigma,0)\mapsto (1,0)\\(0,1)\mapsto (0,1)}]& \Z/8\rar\vResAr[1] &0\\ [2.5 ex]
					0\rar & \Z_2^\times \rar["{g\mapsto (2,0)}","{-1\mapsto (0,1)}"'] \dar[symbol=\cong] \vTrAr[(-)^2] & \Z_2\oplus \Z/2\rar\dar[symbol=\cong] \vTrAr[\substack{(1,0)\mapsto (1+\sigma,0) \\(0,1)\mapsto(0,0)} ]& \Z/2 \rar\dar[symbol=\cong] \vTrAr[0]&0\\ [-2.5 ex]
					0\rar &\hom_c\left(G,\pi_0\left(KU^\wedge_2\right)^\x\right)\rar& \Pic_{K(1)}\left[\left(KU^\wedge_2\right)^{hG}\right]\rar& \left[\Pic_{K(1)}\left(KU^\wedge_2\right)\right]^{G}\rar&0
				\end{tikzcd}
			\end{equation*}
			\item $\{\pm 1\}\times(1+8\Z_2)\le \Z_2^\x$.
			\begin{equation*}
				\begin{tikzcd}[column sep =small]
					0\rar &\hom_c\left(1+4\Z_2,\pi_0\left(KO^\wedge_2\right)^\x\right)\rar\dar[symbol=\cong]& \Pic_{K(1)}\left[\left(KO^\wedge_2\right)^{h(1+4\Z_2)}\right]\rar\dar[symbol=\cong]& \left[\Pic_{K(1)}\left(KO^\wedge_2\right)\right]^{1+4\Z_2}\rar\dar[symbol=\cong]&0\\ [-2.5 ex]
					0\rar &\Z_2^\times \rar["{g\mapsto (1+\sigma,0)}","{-1\mapsto (0,1)}"']\vResAr[(-)^2]& \Z_2\langle1,\sigma\rangle/(4(\sigma-1))\oplus \Z/2\rar\vResAr[\substack{(1,0)\mapsto (1,1,0)\\ (\sigma,0)\mapsto (1,-1,0)\\ (0,1)\mapsto (0,0,0)}]& \Z/8\rar\vResAr[1] &0\\ [2.5 ex]
					0\rar & \Z_2^\times \rar["{g\mapsto (1,0,0)}", "{-1\mapsto (0,0,1)}"'] \dar[symbol=\cong] \vTrAr[1] & \Z_2\oplus \Z/8\oplus \Z/2\rar\dar[symbol=\cong] \vTrAr[\substack{(1,0,0)\mapsto(1+\sigma,0)\\ (0,1,0)\mapsto (1-\sigma, 0)\\ (0,0,1)\mapsto (0,0,1)}]& \Z/8 \rar\dar[symbol=\cong] \vTrAr[2]&0\\ [-2.5 ex]
					0\rar &\hom_c\left(1+8\Z_2,\pi_0\left(KO^\wedge_2\right)^\x\right)\rar& \Pic_{K(1)}\left[\left(KO^\wedge_2\right)^{h(1+8\Z_2)}\right]\rar& \left[\Pic_{K(1)}\left(KO^\wedge_2\right)\right]^{1+8\Z_2}\rar&0
				\end{tikzcd}
			\end{equation*}
			\item $\langle 9\rangle\le G=\langle 3\rangle$ or $\langle 5\rangle$.
			\begin{equation*}
				\begin{tikzcd}[column sep =small]
					0\rar &\hom_c\left(G,\pi_0\left(KU^\wedge_2\right)^\x\right)\rar\dar[symbol=\cong]& \Pic_{K(1)}\left[\left(KU^\wedge_2\right)^{hG}\right]\rar\dar[symbol=\cong]& \left[\Pic_{K(1)}\left(KU^\wedge_2\right)\right]^{G}\rar\dar[symbol=\cong]&0\\ [-2.5 ex]
					0\rar &\Z_2^\times \rar["{g\mapsto (2,0)}","{-1\mapsto (0,1)}"']\vResAr[(-)^2]& \Z_2\oplus \Z/2\rar\vResAr[\substack{(1,0)\mapsto (1,0,1)\\ (0,1)\mapsto (0,0,0)}]& \Z/2\rar\vResAr[1] &0\\ [2.5 ex]
					0\rar & \Z_2^\times \rar["{g\mapsto (1,0,0)}","{-1\mapsto(0,1,0)}"'] \dar[symbol=\cong] \vTrAr[1] & \Z_2\oplus \Z/2\oplus\Z/2\rar\dar[symbol=\cong] \vTrAr[\substack{(1,0,0) \mapsto (2,0)\\ (0,1,0)\mapsto(0,1)\\ (0,0,1)\mapsto (0,0)}]& \Z/2 \rar\dar[symbol=\cong] \vTrAr[0]&0\\ [-2.5 ex]
					0\rar &\hom_c\left(1+8\Z_2,\pi_0\left(KU^\wedge_2\right)^\x\right)\rar& \Pic_{K(1)}\left[\left(KU^\wedge_2\right)^{h(1+8\Z_2)}\right]\rar& \left[\Pic_{K(1)}\left(KU^\wedge_2\right)\right]^{1+8\Z_2}\rar&0
				\end{tikzcd}
			\end{equation*}
			\item $G=\langle 2^k-1\rangle$ or $\langle 2^k+1\rangle\le \{\pm 1\}\times (1+2^k\Z_2)$, where $k\ge 3$.
			\begin{equation*}
				\begin{tikzcd}[column sep =small, ampersand replacement = \&]
					0\rar \&\hom_c\left(G,\pi_0\left(KO^\wedge_2\right)^\x\right)\rar\dar[symbol=\cong]\& \Pic_{K(1)}\left[\left(KO^\wedge_2\right)^{hG}\right]\rar\dar[symbol=\cong]\& \left[\Pic_{K(1)}\left(KO^\wedge_2\right)\right]^{G}\rar\dar[symbol=\cong]\&0\\ [-2.5 ex]
					0\rar \&\Z_2^\times \rar\vResAr[1]\& \Z_2\oplus \Z/2\oplus \Z/8\rar\vResAr[\tiny{\begin{pmatrix}
							1&&\\&1&\\&&1
					\end{pmatrix}}]\& \Z/8\rar\vResAr[1] \&0\\ [2.5 ex]
					0\rar \& \Z_2^\times \rar["{g\mapsto (1,0,0)}","{-1\mapsto(0,1,0)}"'] \dar[symbol=\cong] \vTrAr[(-)^2] \& \Z_2\oplus \Z/2\oplus \Z/2\rar\dar[symbol=\cong] \vTrAr[\tiny{\begin{pmatrix}
							2&&\\&0&\\&&0
					\end{pmatrix}}]\& \Z/2 \rar\dar[symbol=\cong] \vTrAr[0]\&0\\ [-2.5 ex]
					0\rar \&\hom_c\left(G,\pi_0\left(KU^\wedge_2\right)^\x\right)\rar\& \Pic_{K(1)}\left[\left(KU^\wedge_2\right)^{hG}\right]\rar\& \left[\Pic_{K(1)}\left(KU^\wedge_2\right)\right]^{G}\rar\&0
				\end{tikzcd}
			\end{equation*}
			\item $\{\pm 1\}\times(1+2^{k+1}\Z_2)\le \{\pm 1\}\times(1+2^k\Z_2)$, where $k\ge 3$.
			\begin{equation*}
				\begin{tikzcd}[column sep =small,ampersand replacement=\&]
					0\rar \&\hom_c\left(1+2^k\Z_2,\pi_0\left(KO^\wedge_2\right)^\x\right)\rar\dar[symbol=\cong]\& \Pic_{K(1)}\left[\left(KO^\wedge_2\right)^{h(1+2^k\Z_2)}\right]\rar\dar[symbol=\cong]\& \left[\Pic_{K(1)}\left(KO^\wedge_2\right)\right]^{1+2^k\Z_2}\rar\dar[symbol=\cong]\&0\\ [-2.5 ex]
					0\rar \&\Z_2^\times \rar\vResAr[(-)^2]\& \Z_2\oplus \Z/2\oplus\Z/8\rar\vResAr[\tiny{\begin{pmatrix}
							2&&\\&0&\\&&1
					\end{pmatrix}}]\& \Z/8\rar\vResAr[1] \&0\\ [2.5 ex]
					0\rar \& \Z_2^\times \rar \dar[symbol=\cong] \vTrAr[1] \& \Z_2\oplus \Z/2\oplus \Z/8\rar\dar[symbol=\cong] \vTrAr[\tiny{\begin{pmatrix}
							1&&\\&1&\\&&2
					\end{pmatrix}}]\& \Z/8 \rar\dar[symbol=\cong] \vTrAr[2]\&0\\ [-2.5 ex]
					0\rar \&\hom_c\left(1+2^{k+1}\Z_2,\pi_0\left(KO^\wedge_2\right)^\x\right)\rar\& \Pic_{K(1)}\left[\left(KO^\wedge_2\right)^{h(1+2^{k+1}\Z_2)}\right]\rar\& \left[\Pic_{K(1)}\left(KO^\wedge_2\right)\right]^{1+2^{k+1}\Z_2}\rar\&0
				\end{tikzcd}
			\end{equation*}
			\item $ \langle 2^{k+1}+1 \rangle \le G= \langle 2^k+1 \rangle$ or $ \langle 2^k-1 \rangle $, where $k\ge 3$.
			\begin{equation*}
				\begin{tikzcd}[column sep =small,ampersand replacement=\&]
					0\rar \&\hom_c\left(G,\pi_0\left(KU^\wedge_2\right)^\x\right)\rar\dar[symbol=\cong]\& \Pic_{K(1)}\left[\left(KU^\wedge_2\right)^{hG}\right]\rar\dar[symbol=\cong]\& \left[\Pic_{K(1)}\left(KU^\wedge_2\right)\right]^{G}\rar\dar[symbol=\cong]\&0\\ [-2.5 ex]
					0\rar \&\Z_2^\times \rar\vResAr[(-)^2]\& \Z_2\oplus \Z/2\oplus\Z/2\rar\vResAr[\tiny{\begin{pmatrix}
							2&&\\&0&\\&&1
					\end{pmatrix}}]\& \Z/2\rar\vResAr[1] \&0\\ [2.5 ex]
					0\rar \& \Z_2^\times \rar \dar[symbol=\cong] \vTrAr[1] \& \Z_2\oplus \Z/2\oplus \Z/2\rar\dar[symbol=\cong] \vTrAr[\tiny{\begin{pmatrix}
							1&&\\&1&\\&&0
					\end{pmatrix}}]\& \Z/2 \rar\dar[symbol=\cong] \vTrAr[0]\&0\\ [-2.5 ex]
					0\rar \&\hom_c\left(1+2^{k+1}\Z_2,\pi_0\left(KU^\wedge_2\right)^\x\right)\rar\& \Pic_{K(1)}\left[\left(KU^\wedge_2\right)^{h(1+2^{k+1}\Z_2)}\right]\rar\& \left[\Pic_{K(1)}\left(KU^\wedge_2\right)\right]^{1+2^{k+1}\Z_2}\rar\&0
				\end{tikzcd}
			\end{equation*}\qedhere
		\end{enumerate}
	\end{proof}
	\begin{rem}\label{rem:Pic_E1hG_topgen_p=2}
		When $p=2$, the group $\Pic_{K(1)}\cong \mathrm{RO}(C_2)^\wedge_2/(4(1-\sigma))\oplus\Z/2$ is not topologically cyclic. The element $S^1_{K(1)}$ corresponds to $(1,0)$ under the identification. From the formulas in \Cref{thm:Mackey_Pic_k1_p=2}, we observe that:
		\begin{equation*}
			\left(\Sigma E_1^{hG} \in \Pic_{K(1)}\left(E_1^{hG}\right)\right) =\begin{cases}
				(1,0)\in \mathrm{RO}(C_2)^\wedge_2/(4(1-\sigma))\oplus \Z/2, & G=\Z_2^\times;\\
				(1,0)\in \Z_2\oplus \Z/2, & G=\langle 3\rangle \text{ or }\langle 5\rangle;\\
				(2^{k-3},1,0)\in \Z_2\oplus\Z/8\oplus \Z/2, &  G=\{\pm 1\}\times \langle 2^k+1\rangle, k\ge 3;\\
				(2^{k-3},1,0)\in \Z_2\oplus\Z/2\oplus \Z/2, & G=\langle 2^k-1\rangle \text{ or }\langle 2^k+1\rangle, k\ge 3;\\
				1\in \Z/8, &G=\{\pm 1\};\\
				1\in \Z/2, & G=\{1\}. 
			\end{cases}
		\end{equation*}
		As a result, the element $\Sigma E_1^{hG}$ is \emph{not} divisible by $2$ (or its powers) in the Picard group $\Pic_{K(1)}\left(E_1^{hG}\right)$ at prime $2$ for any closed subgroup $G$.  This is very different from the odd prime case in  \Cref{rem:Pic_E1hG_topgen}. 
	\end{rem}
	From the computations in \Cref{thm:Mackey_Pic_k1_p=2}, we verify \Cref{prop:Pic_Kn_colim} in an example below: 
	\begin{cor}\label{cor:colim_res_p=2}
		The colimits of $\Pic_{K(1)}\left(E_1^{h-}\right)$ as pro-abelian groups under transfinite compositions of restriction maps are:
		\begin{align*}
			\colim_{\res} \Pic_{K(1)}\left[\left(KO^\wedge_2\right)^{h(1+2^k\Z_2)}\right]&\cong \Z/8\cong \Pic_{K(1)}\left(KO^\wedge_2\right),\\ 
			\colim_{\res} \Pic_{K(1)}\left[\left(KU^\wedge_2\right)^{h(1+2^k\Z_2)}\right]&\cong \Z/2\cong \Pic_{K(1)}\left(KU^\wedge_2\right).
		\end{align*}
	\end{cor}
	\begin{proof}
		By \Cref{thm:Mackey_Pic_k1_p=2},  the restriction maps 
		\begin{align*}
			\Pic_{K(1)}\left[\left(KO^\wedge_2\right)^{h(1+2^k\Z_2)}\right]\longrightarrow \Pic_{K(1)}\left(KO^\wedge_2\right)&&\Pic_{K(1)}\left[\left(KU^\wedge_2\right)^{h(1+2^k\Z_2)}\right]\longrightarrow \Pic_{K(1)}\left(KU^\wedge_2\right)
		\end{align*}
		are projections onto summands when $k\ge 3$. The claim now follows by reading off the restriction maps in Cases II, V and Cases I, III, respectively. 
	\end{proof}
	\begin{rem}\label{rem:Pic_colim_p2}
		Similar to \Cref{rem:Pic_colim_p_odd}, the claim above fails when we take colimits in $\Ab$. There will be an extra $\Q_2$ summand in the colimit then. This gives another example of \Cref{rem:Pic_Kn_colim}.
	\end{rem}
	One important class in $\Pic_{K(1)}\left(S^{0}_{K(1)}\right)$ at $p=2$ is the exotic element $\Ecal_{K(1)}$.  Under the isomorphism $\Pic_{K(1)}\left(S^{0}_{K(1)}\right)\cong \mathrm{RO}(C_2)^\wedge_2/(4(1-\sigma))\oplus \Z/2$, the exotic element corresponds to $2-2\sigma$. Our computation above implies:
	\begin{cor}\label{cor:exotic_res}
		The exotic element in $\Pic_{K(1)}\left(S^{0}_{K(1)}\right)$ is detected by a closed subgroup $G\le \Z_2^\times$ iff $\{\pm 1\}\le G$. Equivalently, $E_1^{hG}\hsmash \Ecal_{K(1)}\simeq E_1^{hG}$ iff $-1\notin G$.  
	\end{cor}
	\begin{proof}
		We will compute the images of $\Ecal_{K(1)}$ under the restriction maps in the $K(1)$-local Picard Mackey functor. By \Cref{lem:Z2x_lattice}, the index $2$ subgroups of $\Z_2^\times$ are $\{\pm 1\}\times 1+8\Z_2$, $\langle 3\rangle$, and $\langle5\rangle$. Reading off the restriction maps in \Cref{thm:Mackey_Pic_k1_p=2}, we can see:
		\begin{itemize}
			\item Case I  implies that  images of $\Ecal_{K(1)}$ in $\Pic_{K(1)}\left(E_1^{h\langle3\rangle}\right)\cong \Pic_{K(1)}\left(E_1^{h\langle5\rangle}\right)\cong \Z_2\oplus \Z/2$ are both zero. 
			\item  Case II implies that its image n $\Pic_{K(1)}\left(E_1^{h(\{\pm 1\}\times 1+8\Z_2)}\right)\cong \Z_2\oplus\Z/8\oplus \Z/2$ is $(0,4,0)$. 
			\item Subgroups of $\Z_2^\times$ not contained in $\langle 3\rangle$, and $\langle5\rangle$ are of the form $\{\pm 1\}\times 1+2^k\Z_2$. Case V states that the restriction map  $\Pic_{K(1)}\left(E_1^{h(\{\pm 1\}\times 1+2^k\Z_2)}\right)\to \Pic_{K(1)}\left(E_1^{h(\{\pm 1\}\times 1+2^{k+1}\Z_2)}\right)$ restricts to identity on the $\Z/8$-summands. Hence images of $\Ecal_{K(1)}$ in them are all no trivial.
			\item It is well-known that $KO^\wedge_2\hsmash \Ecal_{K(1)}\simeq \Sigma^4 KO^\wedge_2$ is nontrivial. This can also be obtained from the transfinite restriction computation in \Cref{cor:colim_res_p=2}. \qedhere
		\end{itemize}
	\end{proof}
	\emergencystretch=1em
	\printbibliography
\end{document}